\documentclass{amsart}

\usepackage{graphicx}
\usepackage{amsfonts}
\usepackage{amsmath}
\usepackage{amscd}
\usepackage{amssymb}
\usepackage{verbatim}
\usepackage{hyperref}

\newenvironment{proof*}{\noindent\emph{Proof}}{$\square$\smallskip}

\newtheorem{theorem}{Theorem}[section]

\newtheorem{Definition}[theorem]{Definition}
\newtheorem{lemma}[theorem]{Lemma}
\newtheorem{Example}[theorem]{Example}
\newtheorem{corollary}[theorem]{Corollary}
\newtheorem{Remark}[theorem]{Remark}
\newtheorem{proposition}[theorem]{Proposition}

\newtheorem{Exercise}[theorem]{Exercise}
\newtheorem{Exercises}[theorem]{Exercises}
\newtheorem{Notation}[theorem]{Notation}
\newtheorem{Convention}[theorem]{Convention}
\newtheorem{hypothesis}[theorem]{Hypothesis}

\newenvironment{definition}{\begin{Definition}\normalfont}{\end{Definition}}
\newenvironment{example}{\begin{Example}\normalfont}{\end{Example}}
\newenvironment{remark}{\begin{Remark}\normalfont}{\end{Remark}}

\newenvironment{notation}{\begin{Notation}\normalfont}{\end{Notation}}
\newenvironment{convention}{\begin{Convention}\normalfont}{\end{Convention}}

\newcommand{\bz}{\ensuremath{\mathbb{Z}}} 
\newcommand{\bn}{\ensuremath{\mathbb{N}}} 
\newcommand{\id}{\ensuremath{\mathrm{id}}} 
\newcommand{\incl}{\ensuremath{\mathrm{incl}}} 
\newcommand{\e}{\ensuremath{\mathrm{e}}} 
\newcommand{\image}{\ensuremath{\mathrm{im}}} 

\newcommand{\co}{\ensuremath{\colon}} 

\newcommand{\sm}{\ensuremath{{\rm Sim}}} 

\newcommand{\lk}{\ensuremath{{\rm lk}}} 
\newcommand{\sta}{\ensuremath{{\rm st}}} 
\newcommand{\dlk}{\ensuremath{{\lk_{\downarrow}}}} 
\newcommand{\expa}{\ensuremath{{\rm expansion}}} 
\newcommand{\depth}{\ensuremath{{\rm depth}}} 
                            
\title[Finiteness properties]{Finiteness properties of some groups of local similarities}  
\author[D.~S.~Farley]{Daniel S. Farley}
\address{Department of Mathematics and Statistics\\ Miami University\\ Oxford, OH 45056 U.S.A.}
\email{farleyds@muohio.edu}

\author[B.~Hughes]{Bruce Hughes}
\address{Department of Mathematics\\ Vanderbilt University\\ Nashville, TN 37240 U.S.A.}
\email{bruce.hughes@vanderbilt.edu}
\thanks{The second-named author was supported in part by NSF Grant DMS--0504176.}

\date{\today}

\begin{document}

\begin{abstract} Hughes has defined a class of groups, which we call FSS (finite similarity structure) groups.
Each FSS group acts on a compact ultrametric space by local similarities. The best-known example is Thompson's group $V$.

Guided by previous work on Thompson's group $V$, we establish a number of new results about
FSS groups.
Our main result is that a class of FSS groups are of type $F_\infty$. This generalizes work of Ken Brown from the 1980s. 
Next, we develop methods
for distinguishing between isomorphism types of some of the Nekrashevych-R\"{o}ver groups $V_{d}(H)$, where $H$ is a finite group, and show that all such groups
$V_{d}(H)$ have simple subgroups of finite index. 
Lastly, 
we show that FSS groups defined by small $\mathrm{Sim}$-structures are braided diagram groups
over tree-like semigroup presentations. This generalizes a result of Guba and Sapir, who first showed that 
Thompson's group $V$ is a braided diagram group.

\end{abstract}

\maketitle

\setcounter{tocdepth}{1} \tableofcontents

{\footnotesize {\hspace*{15pt}}2000 Math.\ Subject Class. Primary 20F65, 22D10, 54E45 {\hspace*{30pt}} }
\tableofcontents

\section{Introduction}

In \cite{HugLSHP}, Hughes defined a class of groups that act by homeomorphisms on compact ultrametric spaces.
Let $X$ be a compact ultrametric space. A finite similarity structure $\mathrm{Sim}_{X}$ on $X$ assigns to each 
pair of balls $B_{1}, B_{2} \subseteq X$ a finite set $\mathrm{Sim}_{X}(B_{1}, B_{2})$ of surjective similarities from
$B_{1}$ to $B_{2}$. The sets $\mathrm{Sim}_{X}(B_{1}, B_{2})$ are required to have certain additional properties,
such as closure under compositions and under restrictions to subballs. (A complete list of the required properties appears in 
Definition \ref{def:fin sim struct}.) Given a finite similarity structure, one defines an associated group
$\Gamma(\mathrm{Sim}_{X})$: it is the group of homeomorphisms of $X$ that locally resemble elements of $\mathrm{Sim}_{X}$. 
We will call the groups $\Gamma(\mathrm{Sim}_{X})$ finite similarity structure (FSS) groups. 
Perhaps the best-known example of an FSS group is Thompson's group
$V$. Section \ref{section:review} contains a review of FSS groups.

Hughes \cite{HugLSHP} proved that all FSS
groups have the Haagerup property. His argument even established the stronger conclusion that all FSS groups
act properly by isometries on CAT(0) cubical complexes. This greatly extended earlier 
results of Farley \cite{Far}, who showed that $V$ has the 
Haagerup property. 

The results of \cite{HugLSHP} left many open questions about the new class of FSS groups. In this paper, guided by previous work on
Thompson's group $V$ and related groups, we will establish several new properties of FSS groups. 
For instance, Brown \cite{Brown1987} proved
that Thompson's group $V$ has type $F_{\infty}$. 
It seems natural to expect
some more general class of FSS groups to have type $F_{\infty}$ as well.
Our main theorem states a fairly general sufficient condition for an FSS group to have type $F_{\infty}$.
Recall that a group $\Gamma$ has \emph{type $F_{\infty}$} means there exists a $K(\Gamma,1)$-complex each of whose skeleta is finite.  

\begin{theorem}[Main Theorem]
\label{thm: main}
Let $X$ be a compact ultrametric space together with a finite similarity structure $\sm_X$ that is rich in simple contractions
and has at  most finitely many $\sm_X$-equivalence classes of balls of $X$.
If $\Gamma$ is the FSS group associated to $\sm_X$,
then $\Gamma$ is of type $F_\infty$.
\end{theorem} 
This theorem is proved as Theorem~\ref{thm: main in body} below.
Thompson's group $V$ is covered by the  theorem above, and our method of proof can be considered a generalization
of Brown's original argument. The strategy can be briefly sketched as follows. We show 
that every FSS group $\Gamma$ acts on a certain simplicial complex $K$,
which we call its similarity complex. Under the hypothesis that there are finitely many $\mathrm{Sim}_X$-equivalence classes of balls 
(Definition \ref{def:simequivalenceclass}),
we show 
that the complex $K$ will be filtered by 
$\Gamma$-finite subcomplexes. If the finite similarity structure $\mathrm{Sim}_X$ is also rich in simple
contractions (Definition \ref{def: rich in contractions}), then one can argue that the connectivity of the $\Gamma$-finite subcomplexes tends to infinity.
The fact that $\Gamma$ has type $F_{\infty}$ then follows from well-established principles.
The proof of Theorem \ref{thm: main} occupies Sections \ref{sec:The similarity complex}-\ref{sec: zipper}. 

Section~\ref{sec: zipper} also contains a proof that for an arbitrary FSS group $\Gamma$, the similarity complex $K$ is a model for 
$E_{Fin}= \underline{E}\Gamma$,
the classifying space for proper $\Gamma$ actions. 

In Section \ref{sec:isomorphismtype}, we investigate the problem of determining when two Nekrashevych-R\"{o}ver groups $V_{d'}(H')$ and 
$V_{d}(H)$ are isomorphic. (The definition of these groups is recalled at the end of Section \ref{section:review}; in particular, each
Nekrashevych-R\"{o}ver group is indeed an FSS group, and has type $F_{\infty}$ by Theorem \ref{thm: main}. Note that, in this paper, the groups
$H$ in $V_{d}(H)$ are always finite, which is not necessarily the case in \cite{NekJOT}, for instance.)  
Our approach uses results of Rubin \cite{Rubin1996, Rubin1989}. The basic idea is to analyze the germs of the action
of the FSS group $\Gamma$ on the compact ultrametric space $X$. In the event that $V_{d'}(H')$ and $V_{d}(H)$ are isomorphic, Rubin's work implies that
there will be a 
homeomorphism $h: X \rightarrow Y$ between the associated compact ultrametric spaces, and this homeomorphism will induce an isomorphism between
the germ group at $x$ and the germ group at $h(x)$ for every $x\in X$. We can thus distinguish between $V_{d'}(H')$ and $V_{d}(H)$ by showing that they have
different germ groups. We show how to compute the germ group of any group $V_{d}(H)$ at any point $x$, and give a sample application
(Proposition \ref{prop: nonisomorphism}). Our results do not give complete information on the isomorphism types of the Nekrashevych-R\"{o}ver examples,
but should allow one to distinguish between two given groups in many cases.

In Section \ref{sec:simplicity}, we establish simplicity results. Each of the generalized Thompson groups $V_{d}$ is either simple or has
a simple subgroup of index two. We show more generally that every group $V_{d}(H)$ has a simple subgroup of finite index. Specifically, we define
a group $V'_{d}(H)$ that has index one or two in $V_{d}(H)$. We determine the abelianization of $V'_{d}(H)$, show that it is always finite, and prove
that the commutator subgroup $[V'_{d}(H), V'_{d}(H)]$ is simple. Our arguments draw on work of Nekrashevych \cite{NekJOT} and Brin \cite{BrinHDTG}.

In Section \ref{sec:diagramgroups}, we show that every braided diagram group over a tree-like semigroup presentation is an FSS group. Thompson's
group $V$ (and the more general class of generalized Thompson groups $V_{d}$) are all braided diagram groups of this type by 
\cite{Far3} and \cite{GubaSapir}.
It is an open question whether all FSS groups are braided diagram groups.

\section{Groups defined by finite similarity structures}
\label{section:review}

\subsection*{Review of finite similarity structures}

We begin with a review of finite similarity structures on  compact, ultrametric spaces, as defined in Hughes \cite{HugLSHP}.

\begin{definition} An {\it ultrametric space} is a metric space $(X,d)$ such that $d(x,y)\leq\max\{d(x,z), d(z,y)\}$ for all
$x,y,z\in X$.
\end{definition}

If $(X,d)$ is a metric space, $x\in X$, and $r>0$, then $B(x,r)=\{ y\in X ~|~ d(x,y) \leq r\}$ denotes the closed ball about $x$ of radius $r$.
In an ultrametric space, closed balls are open sets; in a compact ultrametric space, closed balls are also open balls (perhaps with a different radius).
Moreover, in an ultrametric space, if two balls intersect, then one must contain the other. 

Throughout this paper, a \emph{ball} in $X$ means a closed ball in $X$.

\begin{definition}
If $\lambda > 0$, then a map $g\co X\to Y$ between metric spaces
$(X,d_X)$ and $(Y,d_Y)$
is a {\it $\lambda$-similarity} provided $d_Y(gx, gy)=\lambda d_X(x,y)$
for all $x,y\in X$.
\end{definition}

\begin{definition} A homeomorphism $g\co X\to Y$ between metric spaces is
a {\it local similarity} if for every $x\in X$ there exists $r, \lambda >0$
such that $g$ restricts to a surjective $\lambda$-similarity 
$g|\co B(x,r)\to B(gx,\lambda r)$.
In this case,
$\lambda$ is the {\it similarity modulus of $g$ at $x$} and we write $sim(g,x)=\lambda$.
A \emph{local similarity embedding} is a local similarity onto its image.
\end{definition}

\begin{convention}
\label{convention:isolated}
For a local similarity $g$, the similarity modulus $sim(g,x)$ is uniquely determined by $g$ and $x$, except in the
case $x$ is an isolated point of $X$. In that case, we will always take $sim(g,x)=1$.
Likewise, if $g\co X\to Y$ is a map between metric spaces and $X=\{ x\}$ is a singleton, then $g$ will only be referred to 
as a $\lambda$-similarity for $\lambda=1$.
\end{convention}

The \emph{group of all local similarities} of a metric space $X$ onto $X$ is denoted $LS(X)$
and is a subgroup of the group of self-homeomorphisms on $X$.

Let $(X, d)$ be a compact ultrametric space.
The metric will usually not be explicitly mentioned.

\begin{definition}
\label{def:fin sim struct}
A {\it finite similarity structure for $X$} is a function
$\sm_X$ that assigns to
each ordered pair $B_1, B_2$ of balls in $X$
a (possibly empty)
set $\sm_X(B_1,B_2)$ of surjective similarities
$B_1\to B_2$ such that whenever 
$B_1, B_2, B_3$ are balls in $X$, the following properties
hold:
\begin{enumerate}
\item (Finiteness) $\sm_X(B_1,B_2)$ is a finite set.
\item (Identities) $\id_{B_1}\in\sm_X(B_1,B_1)$.
\item (Inverses) If $h\in\sm_X(B_1,B_2)$, then $h^{-1}\in\sm_X(B_2,B_1)$.
\item (Compositions) 
If $h_1\in\sm_X(B_1,B_2)$ and $h_2\in\sm_X(B_2,B_3)$, then
$h_2h_1\in\sm_X(B_1,B_3)$.
\item (Restrictions)
If $h\in\sm_X(B_1,B_2)$ and $B_3\subseteq B_1$, 
then $$h|B_3\in\sm_X(B_3,h(B_3)).$$
\end{enumerate}
\end{definition}

In other words, $\sm_X$ is a category whose objects are the balls of $X$ and whose morphisms are finite sets of surjective similarities together 
with a restriction operation.

\begin{definition}
If $B$ is a ball in $X$, then an embedding $h\co B\to X$ is  {\it locally determined by 
$\sm_X$} provided for every $x\in B$, there exists
a ball $B'$ in $X$ such that $x\in B'\subseteq B$, $h(B')$ is a ball in  $X$, and
$h|B'\in\sm_X(B', h(B'))$.
\end{definition}

\begin{definition}
The \emph{finite similarity structure} 
(\emph{FSS}) \emph{group
$\Gamma = \Gamma(\sm_X)$ associated to $\sm_X$} is the set of all homeomorphisms $h\co X\to X$ such that $h$ is locally determined by
$\sm_X$.
\end{definition}

Properties (2)--(5) of Definition~\ref{def:fin sim struct}
imply that $\Gamma(\sm_X)$ is indeed a group.
In fact, it is the maximal subgroup of the homeomorphism group of $X$
consisting of homeomorphisms locally determined by $\sm_X$.
Moreover, 
$\Gamma(\sm_X)$ is a subgroup of the group $LS(X)$.

\begin{definition}
\label{def:loc det}
A subgroup of $\Gamma(\sm_X)$ is said to be a {\it group locally determined by
$\sm_X$}. 
\end{definition}

\subsection*{Examples of FSS groups} 

We recall standard alphabet language and notation.
An {\it alphabet} is a non-empty finite set $A$. Finite (perhaps empty) $n$-tuples
of $A$ are {\it words}. We typically write a word as a string of letters from $A$. 
The set of all words is denoted $A^*$
and the set of {\it infinite words} is denoted
$A^\omega $; that is,
$$A^* = \coprod_{n=0}^\infty A^n ~~~\text{and}~~~ 
A^\omega = \prod_1^\infty A.$$
The set of non-empty words is denoted $A^+$; that is, $A^+ = \coprod_{n=1}^\infty A^n$.
If $u\in A^*$, then $|u|=n$ means $u\in A^n$.
If $u\in A^*$ with $u\not=\emptyset$ and $n$ is a non-negative integer, then $u^n:= uu\cdots u$ ($n$ times) $\in A^*$ 
and $\bar{u} := uuu\cdots\in A^\omega$.

Let $T_A$ be the tree associated to $A$. The vertex set of $T_A$ is $A^*$. Two words $v,w$ are connected by an edge
if and only if there exists $x\in A$ such that
$v=wx$ or $vx=w$. The root of $T_A$ is $\emptyset$.
Thus, $A^\omega = \mathrm{Ends}(T_A,\emptyset)$, the end space of the tree $T_A$ with root $\emptyset$, and so comes with a natural
ultrametric  $d$ making $A^\omega$  compact.
That is, if $x=x_1x_2x_3\dots$ and $y=y_1y_2y_3\dots$ are in $A^\omega$, then
$$d(x,y) = \begin{cases} 0 &\text{if $x=y$}\\
                        \e^{1-n} &\text{if $n=\min\{ k ~|~ x_k\not= y_k\}$}
\end{cases}.$$                        

\begin{remark}\label{rem: characterization of balls}
The metric balls in $A^\omega$ are of the form $wA^\omega$, where $w\in A^*$.
\end{remark}

We may assume that $A$ is totally ordered. There is then an induced
total order on $A^\omega$, namely  the lexicographic order.


Let $A= \{ a_1, a_2,\dots, a_d\}$ and let $\Sigma_d$ be the symmetric group on $A$. 
There is an action of $\Sigma_d$ on $A^\ast$ given by 
$\sigma(x_1\dots x_n)=\sigma(x_1)\dots\sigma(x_n)$; 
this action induces an action of $\Sigma_d$ on the tree $T_A$.
Indeed, there is an action of $\Sigma_d$ on $A^\omega$ given by
$$\sigma(x_1 x_2 x_3\dots)=\sigma(x_1)\sigma(x_2)\sigma(x_3)\dots .$$

\begin{notation}
Let $H$ be a subgroup of $\Sigma_d$.
\end{notation}

\begin{definition}
\label{def:symmetric group sim}
If $w_1, w_2\in A^\ast$, then let 
$\sm(w_1A^\omega, w_2A^\omega)$ consist of all homeomorphisms $h\co w_1A^\omega\to w_2A^\omega$
for which there exists $\sigma\in H$ such that
$h(w_1x)=w_2\sigma(x)$ for all $x\in A^\omega$.
Then $\sm$ is the \emph{finite similarity structure for $A^\omega$ determined by $H$.}
\end{definition}

\begin{remark}
\label{rem:sim properties}
Here are some observations related to Definition~\ref{def:symmetric group sim}.
\begin{enumerate}
\item  $\sm$ is a finite similarity structure for $A^\omega$.
	\item The element $\sigma\in H$ is uniquely determined by $h\in\sm(w_1A^\omega, w_2A^\omega)$.
	\item Even though $w_1$ and $w_2$ are not uniquely determined by $h$, the integer $|w_2|-|w_1|$ is the 
	natural logarithm of the 	similarity modulus of $h$ at each point of $w_1A^\omega$. 
	Hence, $|w_2|-|w_1|$ is uniquely determined by $h$. 
	Moreover, $h$ together with either $w_1$ or $w_2$ uniquely determines the other.
	\item If $p, q\in A^*$ are such that $h|\in\sm(w_1pA^\omega, w_2qA^\omega)$, then $h|$ is given by $w_1px\mapsto w_2q\sigma(x)$
	for all $x\in A^\omega$ and $|p| = |q|$. 
	\item $\sm(w_1A^\omega, w_2A^\omega)$ contains the unique order--preserving similarity, which is given by
$w_1x\mapsto w_2x$ for all $x\in A^\omega$.
\end{enumerate}
\end{remark}

\begin{remark}
 \label{claim: Vdh is Gamma} 
If $\Gamma=\Gamma(\sm)$ is the FSS group associated to $\sm$, then
$\Gamma$ is isomorphic to the Nekrashevych--R\"over groups $V_d(H)$.
See Hughes \cite{HugLSHP} for comments about the groups of
Nekrashevych \cite{NekJOT} and  R\"over \cite{Rov99}.
For example,
note that in the special case $H=\{ 1\}$,  the group $V_d(H)$ is   $G_{d,1}$, which is a Higman--Thompson group.
\end{remark}

\section{The similarity complex associated to a finite similarity structure}
\label{sec:The similarity complex}

Throughout this section, $X$ will 
denote a non-empty, compact, 
ultrametric space with a
finite similarity structure $\sm=\sm_X$ on $X$.

Note that  the image of a local similarity embedding $f\co B\to X$, where $B$ is a ball in $X$, is a finite union of mutually disjoint
balls in $X$ (see \cite[Lemma 2.4]{HugLSHP}).

We begin by recalling the zipper  as defined in Hughes \cite{HugLSHP}.
Consider the set 
\begin{align*}
 \mathcal{S} := \{ (f,B) ~|~ &\text{$B$ is a  ball in $X$ and $f\co B\to X$} \\
&\text{is an embedding locally determined by $\sm$}\}.\end{align*}
Define an equivalence relation on $\mathcal S$ by declaring 
that $(f_1, B_1)$ and $(f_2, B_2)$ are {\it equivalent}
provided there exists $h\in\sm(B_1, B_2)$ such
that $f_2h=f_1$ (in particular,
$f_1(B_1)=f_2(B_2)$). The verification that this is an equivalence
relation requires the Identities, Compositions, and Inverses
Properties of the similarity structure.
Equivalence classes are denoted by $[f,B]$.
Let $\mathcal E$ be the set of equivalence classes of pairs $(f,B)\in\mathcal S$.
Thus,
$$\mathcal E := \{ [f,B] ~|~ (f,B)\in\mathcal{S}\}.$$
The \emph{zipper} is 
$$Z := \{[f,B]\in\mathcal E ~|~ f(B) ~\text{is a ball in $X$
and $f\in\sm(B, f(B))$}\}.$$

Note that an element $[f,B]\in Z$ is uniquely determined
by the ball $f(B)$. In fact, $[f,B] = [\incl_{f(B)},f(B)]$,
where $\incl_Y\co Y\to X$ denotes the inclusion map.
Thus,
$$Z = \{[\incl_B, B]\in{\mathcal E} ~|~ B ~\text{is a ball in $X$}\}.$$
In particular, $Z$ can be identified with the collection of all balls in $X$.

We now begin the construction of a complex on which $\Gamma$ acts.

\begin{definition}
Let $k$ be a positive integer. 
A \emph{pseudo-vertex $v$ of height $k$}  is a set
$$v = \{ [f_i, B_i] ~|~ 1\leq i\leq k \},$$ where $[f_i, B_i]\in\mathcal E$ for each $i= 1,\dots, k$
and such that $\{f_i(B_i)\}_{i=1}^k$ is a collection of disjoint subsets of $X$.
The height of $v$ is denoted $\Vert v\Vert = k$.
The \emph{image of $v$} is $\image(v) :=\bigcup_{i=1}^k f_i(B_i)\subseteq X$.
\end{definition}

Note that the image of a pseudo-vertex $v$ is well-defined.
Note also that the set of pseudo-vertices of height $1$ 
is $\{ \{ [f,B]\} ~|~ [f,B]\in\mathcal{E}\}$. 
That is, with a slight abuse of notation, $\mathcal E$ is the set of pseudo-vertices of height $1$.

\begin{definition} \label{def:simequivalenceclass}
The \emph{$\sm$--equivalence class} of a ball $B$ in $X$ is
$$[B] := \{ A\subseteq X ~|~ \text{$A$ is a ball and $\sm(A,B)\not=\emptyset$}\}.$$
\end{definition}

The Identities, Inverses, and Compositions Properties imply that $\sm$--equivalence
is an equivalence relation on the set of all balls in $X$.

\begin{definition}
The \emph{second coordinate} of a pseudo-vertex $v=\{[f,B]\}$ of height $1$ is the 
$\sm$-equivalence class $[B]$. 
The \emph{set of second coordinates} of a pseudo-vertex $v=\{[f_i,B_i] ~|~ 1\leq i\leq k\}$ of height $k$ is
the set $\{ [B_i] ~|~ 1\leq i\leq k\}$.
\end{definition}

Note that this is well-defined; that is, if $[f,B] = [f',B']$, then $[B]=[B']$.

\begin{definition}
A \emph{vertex $v$ of height $k$}  is a pseudo-vertex
$$v = \{ [f_i, B_i] ~|~ 1\leq i\leq k \}$$ 
of height $k$ such that 
$X=\coprod_{i=1}^k f_i(B_i)$,
where $\coprod$ denotes disjoint union.
The set of all vertices of all heights is denoted $K^0$.
\end{definition}

Note that a pseudo-vertex $v$ is a vertex if and only if $\image(v) = X$.
Note also that every homeomorphism $ \gamma\co X\to X$ locally determined by $\sm$ 
represents a vertex $[\gamma ,X]$ of height $1$.

\begin{definition}
A pseudo-vertex $v$ is \emph{positive} if each element of $v$ is in the zipper $Z$.
\end{definition}

\begin{remark}
As noted above, there is a bijection from the zipper $Z$ to the set of balls in $X$.
That bijection induces a bijection from the set of positive vertices to the set of
partitions of $X$ into balls. This bijection sends a positive vertex
$v= \{ [f_i,B_i] \}_{i=1}^k$ to the partition $\{f_i(B_i) \}_{i=1}^k$.
The inverse of this bijection sends a partition $\{ B_i\}_{i=1}^k$ of $X$ into balls
to the positive vertex $\{ [\incl_{B_i},B_i] \}_{i=1}^k$.
\end{remark}

\begin{definition}
If $v$ is a pseudo-vertex and $[f,B]\in v$ with $B$ containing more than one point, then the \emph{simple expansion of $v$ at $[f,B]$}
is the pseudo-vertex 
$$w= \{[g,A]\in v ~|~ [g,A]\not= [f,B]\} \cup \{[f|A, A] ~|~ A \text{~is a maximal proper sub-ball of $B$}\}.$$
Moreover, $v$ is the \emph{simple contraction of $w$ at} 
$$\{[f|A, A] ~|~ A \text{~is a maximal proper sub-ball of $B$}\}.$$
In this situation, we write $v\nearrow w$ and $w\searrow v$.
\end{definition}

If $v$ is a pseudo-vertex and $[f,B]\in v$ with $B$ containing exactly one point (which is to say, $B$ does not contain a proper sub-ball),
then the expansion of $v$ at $[f,B]$ is not defined.

\begin{remark}
If $v$ and $w$ are pseudo-vertices such that $v\nearrow w$, then the following hold.
\begin{enumerate}
	\item $||v|| <||w||$.
	\item $v$ is a vertex if and only if $w$ is a vertex.
	\item If $v$ is positive, then $w$ is positive.
\end{enumerate}
\end{remark}

\begin{remark}\label{rem: well-defined expansions} Simple expansions  are well-defined in the following sense.
If $[f_1,B_1] = [f_2, B_2]\in v$, then
\begin{align*}
&\{[f_1|A_1, A_1] ~|~ A_1 \text{~is a maximal proper sub-ball of $B_1$}\} =\\
&\{[f_2|A_2, A_2] ~|~ A_2 \text{~is a maximal proper sub-ball of $B_2$}\}.
\end{align*}
(This follows from the fact that a surjective similarity $B_1\to B_2$ carries maximal proper sub-balls of $B_1$ to maximal proper sub-balls of $B_2$
and from the Restrictions property of $\sm$.)
The converse need not be true. That is, if $w$ is a pseudo-vertex and $u\subseteq w$, then it might be the case
that there is more than one pseudo-vertex that is a simple contraction of $w$ at $u$. However, if $v$ is a simple contraction
of $w$ at $u$, then $u$ is uniquely determined: if $v$ is also a simple contraction of $w$ at $u'$, then $u = u'$.
\end{remark}

\begin{remark}
\label{rem:unions of pseudo-vertices}
Let $v$ and $w$ be pseudo-vertices such that $\image(v)\cap\image(w)=\emptyset$.
The following observations are immediate.
\begin{enumerate}
	\item $v\cup w$ is a pseudo-vertex and $||v\cup w||= ||v||+||w||$.
	\item If $v\nearrow v'$, then $\image(v')\cap\image(w)=\emptyset$ 
	and $v\cup w\nearrow v'\cup w$.
	\item If $v$ and $w$ are positive, then so is $v\cup w$.
\end{enumerate}
\end{remark}

\begin{definition} If $v$ and $w$ are pseudo-vertices, then write $v\leq w$ if and only if there is a finite sequence of simple
expansions $v=v_1\nearrow v_2  \cdots \nearrow v_n= w.$ The pseudo-vertex $w$ is an \emph{expansion} of $v$
and $v$ \emph{expands to} $w$.
\end{definition}

\begin{lemma} The set of pseudo-vertices is partially ordered by $\leq$.
\end{lemma}

\begin{proof}
The relation is clearly reflexive. 
It is antisymmetric because if $w$ is an expansion of $v$, then $||v|| < ||w||$.
The relation is transitive because it is defined to be the transitive closure of a reflexive, antisymmetric relation.
\end{proof}

The following remark is an immediate consequence of Remark~\ref{rem:unions of pseudo-vertices}(2) and the definitions.

\begin{remark}
If $v,w, v', w'$ are pseudo-vertices such that $\image (v) \cap \image (w) =\emptyset$,
$v\leq v'$, and $w\leq w'$, then $\image(v') \cap \image(w') =\emptyset$
and  $v\cup w\leq v'\cup w'$.
\end{remark}

\begin{remark}
The only pseudo-vertices that are maximal with respect to $\leq$ are those of the form
$\{[f_i,B_i] ~|~ 1\leq i\leq k\}$, where $B_i$ is a singleton for each $i=1,\dots,k$.
In particular, if $X$ has no isolated points, then there are no maximal pseudo-vertices.
\end{remark}

\begin{remark}
\label{rem: unique containment}
If $v, w$ are pseudo-vertices, $v\leq w$, and $[f,B]\in w$, then there exists a unique $[g,A]\in v$
such that $f(B)\subseteq g(A)$.
\end{remark}

\begin{definition} \label{def: complete expansion}
If $v=\{[f_i,B_i] ~|~ 1\leq i\leq k\}$ is a pseudo-vertex, then the \emph{complete expansion of $v$}
is the pseudo-vertex
\begin{align*}
\expa(v) := \{ [f_i|A,A] ~|~ \text{$1\leq i\leq k$ and~} & \text{$A$ is a maximal, proper sub-ball of $B_i$,}\\
                             &\text{or $A=B$ if $B_i$ is a singleton}\}.
\end{align*}                             
\end{definition}

\begin{remark}
\label{rem:expansion}
If $v$ is a pseudo-vertex of height $1$, then $v\nearrow\expa(v)$.
It follows from Remark~\ref{rem:unions of pseudo-vertices}(2) that if 
$v=\{ [f_i, B_i] ~|~ 1\leq i\leq k\}$ is a pseudo-vertex of height $k$, then
\begin{align*}
v &\nearrow \expa\{ [f_1,B_1]\} \cup \{ [f_i, B_i] ~|~ 2\leq i\leq k\} \nearrow \cdots \\
& \cdots \nearrow \bigcup_{i=1}^k\expa\{[f_i,B_i]\} =\expa(v).
\end{align*}
In particular, $v\leq\expa(v)$.
\end{remark}

\begin{definition}
Let $B$ be a ball in $X$. Inductively define a sequence $\{ \mathcal{B}_i\}_{i=0}^\infty$ of partitions of $B$ into sub-balls
as follows. First, $\mathcal{B}_0=\{ B\}$. 
Assuming $i>0$ and $\mathcal{B}_i$ has been defined, a sub-ball $A$ of $B$ is in $\mathcal{B}_{i+1}$ if and only if
there exists a ball $C\in\mathcal{B}_i$ such that $A$ is a maximal proper sub-ball of $C$, or $C$ is a singleton and $A=C$.
The sequence $\{ \mathcal{B}_i\}_{i=0}^\infty$ is the \emph{ball hierarchy of $B$}.
\end{definition}

Suppose $(f,B)\in\mathcal{S}$ and let $\{ \mathcal{B}_i\}_{i=0}^\infty$ be the ball hierarchy of $B$. 
Observe that if $i\geq 1$ and $A\in\mathcal{B}_i$, then the Restrictions property implies $(f|A,A)\in\mathcal{S}$.
For each $x\in B$, let $D((f,B),x)$ denote the smallest nonnegative integer $i$ such that
there exists $A\in\mathcal{B}_i$ with $x\in A$, $f(A)$ a ball, and $f|A\in\sm(A,f(A))$.
The integer $D((f,B),x)$ is called the \emph{depth of $(f,B)$ at $x$}.
Note that if $y\in A$, then $D((f,B),y) = D((f,B),x)$ (since any two balls are either disjoint or one contains the other).
Thus, $D((f,B),\cdot)$ is a locally constant function on $X$.

\begin{definition}
If $(f,B)\in\mathcal{S}$, then the \emph{depth of $[f,B]\in\mathcal{E}$} is
$$D[f,B] := \max\{ D((f,B),x) ~|~ x\in B\}.$$
\end{definition}

Note that $D[f,B]$ is well-defined; that is, it is independent of the representative in $\mathcal{S}$ of $[f,B]\in\mathcal{E}$.

\begin{definition}
If $v=\{ [f_i,B_i] ~|~ 1\leq i \leq k\}$ is a pseudo-vertex, then the \emph{depth of v} is
$$\depth(v) := \max \{ D[f_i,B_i] ~|~ 1\leq i \leq k\}.$$
\end{definition}

\begin{remark}
\label{rem:depth}
If $v$ is a pseudo-vertex, then the following hold.
\begin{enumerate}
	\item $\depth(v) = 0$ if and only if $v$ is positive.
	\item $\depth(\expa(v))\leq\depth(v)$, with equality if and only if $\depth(v)=0$.
\end{enumerate}
\end{remark}

\begin{lemma}
\label{lem:positive expansion} Every pseudo-vertex expands to a positive pseudo-vertex. 
In particular, for every vertex $v\in K^0$, there exists a positive vertex $w$ 
such that $v\leq w$.
\end{lemma}

\begin{proof}
If $k=\depth(v)$, then it follows from Remarks~\ref{rem:expansion} and \ref{rem:depth}
that $v:=v_0\leq v_1\leq\cdots\leq v_k$, where
$v_i:=\expa(v_{i-1})$ for $1\leq i\leq k$
and that $\depth(v_k)=0$.
Thus, $w:=v_k$ is the desired positive pseudo-vertex. 

The second statement of the lemma follows from the first together with  the observation that the expansion
of a vertex is a vertex.
\end{proof}

\begin{lemma}
\label{lem:partition expansion}
If $B$ is a ball in $X$ and $\mathcal{P}$ is a partition of $B$ into sub-balls,
then the positive pseudo-vertex $\{ [\incl_B,B]\}$ expands to the positive pseudo-vertex
$\{ [\incl_A, A] ~|~ A\in\mathcal{P}\}$.
\end{lemma}

\begin{proof}
Observe first that if $B'$ is a sub-ball of $B$, and $B'' \in \mathcal{P}$ is a sub-ball of $B'$, then there is
$\mathcal{P}' \subseteq \mathcal{P}$ partitioning $B'$. 
The proof of the lemma is by induction on the cardinality of $\mathcal P$.
If $|\mathcal P| = 1$, then $\mathcal{P} = \{ B\}$ and there is nothing to prove.
Assume $|\mathcal P|>1$ and that the statement is true for partitions of smaller cardinality.
Let $\{ \mathcal{B}_i\}_{i=0}^\infty$ be the ball hierarchy of $B$ and 
let
$N=\max\{ i>0 ~|~ \mathcal{P}\cap\mathcal{B}_i\not=\emptyset\}$
and choose $C\in\mathcal{P}\cap\mathcal{B}_N$. Note $C\not= B$.
Let $D$ be the smallest sub-ball of $B$ such that $C\not= D$ and $C\subseteq D$.
Note that $C$ is a maximal proper sub-ball of $D$.
By the observation above, $\mathcal P$ contains a partition $\mathcal{P}_D$ of $D$.
By the definition of $N$, $\mathcal{P}_D$ is the partition of $D$ into maximal proper sub-balls.
Clearly, $C\in\mathcal{P}_D$ and $|\mathcal{P}_D|>1$.
Let $\mathcal{P}'=\mathcal{P}\setminus\mathcal{P}_D\cup\{ D\}$.
Since $\mathcal{P}'$ is a partition of $B$ by balls and 
$|\mathcal{P}'|<|\mathcal{P}|$, the inductive assumption implies that
$\{ [\incl_B,B]\}$ expands to the pseudo-vertex
$w=\{ [\incl_A, A] ~|~ A\in\mathcal{P}'\}$.
The proof is now complete upon observing that the simple expansion of $w$
at $[\incl_D,D]$ is the pseudo-vertex
$\{ [\incl_A, A] ~|~ A\in\mathcal{P}\}$.
\end{proof}

\begin{definition}
The \emph{similarity complex associated to $\sm$} is the simplicial complex $K=K_\sm$ obtained from $(K^0, \leq)$. Thus, an $n$-simplex of $K$ is an
ascending chain $(v_0, v_1, \dots, v_n)$ of distinct vertices  $v_0 < v_1 <\cdots < v_n$. 
\end{definition}

Note that the vertices of an $n$-simplex of $K$ are totally ordered by  $\leq$. 
Note also that $K\not=\emptyset$ because it contains the positive vertex $\{ [\id_X,X]\}$ of height $1$.

\begin{proposition}
\label{prop: contractible}
The partially ordered set $(K^0, \leq)$ is a directed set. Hence, $K$ is contractible.
\end{proposition}

\begin{proof}
By Lemma~\ref{lem:positive expansion} $(K^0, \leq)$ is a directed set if any two positive vertices have an upper bound.
If $v_1$ and $v_2$ are positive vertices, then there are partitions $\mathcal{P}_1$ and $\mathcal{P}_2$ of $X$ into balls
such that 
$v_i=\{ [\incl_B,B] ~|~ B\in\mathcal{P}_i\}$ for $i=1,2$.
Let $\mathcal{P} =\{ B_1\cap B_2 ~|~ \text{$B_1\in\mathcal{P}_1$, $B_2\in\mathcal{P}_2$, and $B_1\cap B_2\not=\emptyset$}\}.$
Thus, $\mathcal P$ is a common refinement of $\mathcal{P}_1$ and $\mathcal{P}_2$ and $\mathcal P$ is a partition of 
$X$ into balls. Moreover, $\mathcal P$ contains a partition of any ball in $\mathcal{P}_1$ or in $\mathcal{P}_2$.
Lemma~\ref{lem:partition expansion} implies that if $i=1$ or $2$ and $B\in\mathcal{P}_i$,
then the pseudo-vertex $\{[\incl_B,B]\}$ expands to the pseudo-vertex
$\{[\incl_A,A] ~|~ \text{$A\in\mathcal{P}$ and $A\subseteq B$}\}$ for $i=1,2$.
Remark~\ref{rem:unions of pseudo-vertices} implies that both $v_1$ and $v_2$ expand to the vertex
$ \{[\incl_A,A] ~|~ \text{$A\in\mathcal{P}$}\}.$
This completes the proof of the first statement of the proposition.
The second statement follows from the well-known fact that the complex obtained from a directed, partially ordered set
is contractible (see Geoghegan \cite[Proposition 9.3.14, page 210]{Geoghegan}).
\end{proof}

\begin{example} Let $X=\{x_1, \dots, x_n\}$ be a finite ultrametric space in which the distance between any two distinct points
is $1$. Note that $\{ x_{i} \}$ (for $i \in \{ 1, \ldots, n \}$) and $X$ itself are the only balls in $X$. For a pair of balls
$B_{1}$, $B_{2} \subseteq X$, we define $\sm_{X}(B_{1},B_{2})$ as follows:
\begin{enumerate}
\item $\sm_{X}( \{ x_{i} \}, \{ x_{j} \}) = \{ \phi_{ij} \}$, where $\phi_{ij}$ is the only possible map
$\phi_{ij}: \{ x_{i} \} \rightarrow \{ x_{j} \}$;
\item \label{choice} $\sm_{X} (X,X) = \{ \id_{X} \}$;
\end{enumerate}
It is straightforward to check that $\sm_{X}$ is a finite similarity  structure, and that $\Gamma(\sm_{X}) = \Sigma_{X}$, the symmetric group
on $X$. There are exactly $n! + 1$ vertices:
\begin{itemize}
\item If $\phi \in \Sigma_{X}$, then $\{ [\phi, X] \}$ is a vertex of height $1$. Since $\sm_{X}(X,X) = \{ \id_{X} \}$, 
$\{ [\phi_{1}, X] \} \neq \{ [\phi_{2}, X] \}$ if $\phi_{1} \neq \phi_{2}$, so there are $n!$ vertices of this type.
\item The remaining vertex is $\{ [\phi_{ii}, \{ x_{i} \}]  ~|~ 1\leq i\leq n \}$.
\end{itemize}
Every vertex of the form $\{ [\phi, X] \}$ expands to $\{ [\phi_{ii}, \{ x_{i} \}] \}$. It follows that $K$ may be identified with
the cone on $\Sigma_{X}$; that is,
$$ K = ( \Sigma_{X} \times I )/ \sim,$$
where $I$ denotes the unit interval and $(\phi_{1}, t_{1}) \sim (\phi_{2}, t_{2})$ if $t_{1} = t_{2} = 0$. The action of 
$\Gamma(\sm_{X})$ on $K$ under this identification is the same as the natural action of $\Sigma_{X}$ on its cone.

On the other hand, we might set $\sm_{X}(X,X) = \Sigma_{X}$ 
(in place of (\ref{choice}) above). The result is still a finite similarity structure. In
this case, there are just two vertices, $\{ [\id_{X}, X] \}$ and $\{ [\phi_{ii}, \{ x_{i} \}] \mid i \in \{ 1, \ldots, n \} \}$,
and $K$ may be identified with the unit interval. We still have $\Gamma(\sm_{X}) = \Sigma_{X}$, but the action of $\Gamma(\sm_{X})$
on $K$ is now trivial.

Various intermediate constructions are possible, depending on the size of the group $\sm_{X}(X,X)$. 
\end{example}

Note that up to this point we have not used the Finiteness property of the $\sm$ structure.

\section{Local finiteness of the sub-level complexes} \label{sec:localfiniteness}

We continue to use the same notation as in the previous section.
In particular, $X$ denotes a non-empty, compact ultrametric space with a finite similarity structure $\sm$. 
Moreover, $K$ denotes the similarity complex associated to $\sm$.

The  goal of this section is to filter $K$ by subcomplexes that are locally finite if the set of $\sm$--equivalence classes of balls in $X$ is assumed to be finite (see Proposition~\ref{prop: locally finite filtration}).

\begin{definition}
For $n\in\bn$, the \emph{sub-level complex $K_{\leq n}$} is the subcomplex of $K$ spanned by all vertices of height less than or equal to $n$.
\end{definition}

\begin{lemma}
\label{lem:finite contractions} 
Suppose that  $B$ is a ball in $X$, $w$ is a  pseudo-vertex,
and $P_{w,B}$ denotes the set of all pseudo-vertices $v$ of height $1$
such that the second coordinate of $v$ is $[B]$ and such that $v\nearrow w$.
Then $P_{w,B}$ is finite.
\end{lemma}

\begin{proof}
Write
$w=\{ [f_{i}, B_i] ~|~ 1\leq i\leq k\}$.
We may assume that $P_{w,B}$ is not empty so that there is an element in $P_{w,B}$ of the form $[f,B]$.
The fact that $[f,B]\nearrow w$ implies that there are exactly $k$ maximal proper sub-balls of $B$,
say $\widehat{B}_1,\dots,\widehat{B}_k$, indexed so that if $\widehat{f}_i= f|\widehat{B}_i$, then
$[\widehat{f}_i,\widehat{B}_i]= [f_{i}, B_i]$ for $i=1,\dots, k$.
Let $\mathcal{S}_{w,B} = \{ (g,B) \in \mathcal{S} ~|~ [g,B]\in P_{w,B}\}.$
Since the function $\mathcal{S}_{w,B}\to P_{w,B}$ defined by $(g,B)\mapsto [g,B]$ is surjective, it suffices
to show that $\mathcal{S}_{w,B}$ is finite.
Let $\Sigma_k$ be the set of all permutations of $\{1,\dots, k\}$.
The proof will be completed by defining an injection 
$$\Psi\co \mathcal{S}_{w,B}\to\coprod_{\sigma\in\Sigma_k}\prod_{i=1}^k\sm(\widehat{B}_i, B_{\sigma(i)}).$$
Given 
$(g,B)\in\mathcal{S}_{w,B}$,
we know that $\{ [g_i,\widehat{B}_i] ~|~ 1\leq i\leq k\} = \{ [f_{i}, B_i] ~|~ 1\leq i\leq k\}$,
where $g_i=g|\widehat{B}_i$.
It follows that there exists a unique $\sigma\in\Sigma_k$ such that $[g_i,\widehat{B}_i] = [f_{{\sigma(i)}}, B_{\sigma(i)}]$
for $i=1,\dots, k$.
Thus, $f_{\sigma(i)}^{-1}g_i\in\sm(\widehat{B}_i,{B}_{\sigma(i)})$ and we can define
$\Psi(g,B) = (f_{\sigma(1)}^{-1}g_1,\dots, f_{\sigma(k)}^{-1}g_k)\in\prod_{i=1}^k \sm(\widehat{B}_i, B_{\sigma(i)}).$
To see that $\Psi$ is injective, suppose we have another element 
$(h,B)\in\mathcal{S}_{w,B}$ and $\Psi(h,B)=\Psi(g,B)$.
It follows that  $f_{\sigma(i)}^{-1}g_i = f_{\sigma(i)}^{-1}h_i$ for each $i=1,\dots, k$,
where $h_i = h|\widehat{B}_i$.
Thus, $g=h$ and $(g,B)= (h,B)$.
\end{proof}

\begin{remark}
Note that the previous argument relied on the Finiteness  property of the similarity structure.
\end{remark}

\begin{lemma}
\label{lem:finite immediate successors} 
If $v$ is a pseudo-vertex, then $v$ has only finitely many immediate successors.
\end{lemma}

\begin{proof} This is clear because $v$ contains only finitely many elements at which a simple
expansion may be performed.
\end{proof}

In the next result, we will begin using the assumption that the set of $\sm$-equivalence classes of balls in $X$ is finite.
This assumption will be required for the main result, Theorem \ref{thm: main in body}.

\begin{lemma}
\label{lem:finite immediate predecessors}
If $w$ is a pseudo-vertex and the set of $\sm$--equivalence classes of balls in $X$ is finite, then $w$ has only finitely many
immediate predecessors. 
\end{lemma}

\begin{proof} 
An immediate predecessor of $w$ is a pseudo-vertex $v$ such that there is an elementary expansion $v\nearrow w$.
Thus, there is a subset $w'\subseteq w$ and a pseudo-vertex $v'\subseteq v$
of height $1$ such that $v'\nearrow w'$ and $w\setminus w'=v\setminus v'$.
There are only finitely many possibilities for $w'$ (since $w$ has only finitely many subsets).
Once $w'$ is fixed, there are only finitely many possibilities for the second coordinate of $v'$
(by the assumption of the finiteness of the set of $\sm$--equivalence classes of balls).
Finally, once $w'$ and the second coordinate of $v'$ are fixed, there are only finitely many possibilities for $v'$ by
Lemma~\ref{lem:finite contractions}.
\end{proof}

\begin{proposition} 
\label{prop: locally finite filtration}
If the set of $\sm$--equivalence classes of balls in $X$ is finite and $n\in\bn$, then the sub-level complex $K_{\leq n}$
is locally finite.
\end{proposition}

\begin{proof}
It follows from Lemmas~\ref{lem:finite immediate successors} and \ref{lem:finite immediate predecessors}
that any vertex $v$ of $K_{\leq n}$ is contained in at most finitely many ascending chains of vertices in $K_{\leq n}$.
That is to say, $v$ is in only finitely many simplices of $K_{\leq n}$.
\end{proof}

\begin{remark} \label{rem: K not locally finite}
The complex $K$ is usually not locally finite. 
In fact, the following are equivalent:
\begin{enumerate}
	\item $K$ is finite.
	\item $K$ is locally finite.
	\item $X$ is finite.
\end{enumerate}
\end{remark}
\begin{proof} 
If $X$ is not finite, then since $X$ is compact there exists a sequence of balls
$X=B_1\supseteq B_2 \supseteq B_3 \supseteq \cdots$ such that $B_{i+1}$ is a maximal proper sub-ball of $B_i$
for each $i\in\bn$.
Define vertices $v_1 < v_2 < v_3< \cdots$ inductively as follows. Let $v_1 =\{[\incl_{B_1}, B_1]\}$. If $i>1$ and $v_i$
has been defined so that $[\incl_{B_i}, B_i] \in v_i$, let $v_{i+1}$ be obtained from $v_i$ by a simple expansion at 
$[\incl_{B_i}, B_i]$. Then $v_1$ is a vertex of the simplex spanned by $\{v_1,\dots, v_n\}$ for every $n\in\bn$, showing
that $K$ is not locally finite.

On the other hand, if $X$ is finite, then it is rather obvious that $K$ is finite: if $X$ has cardinality $n$, then there are only finitely many 
partitions of $X$ and each has cardinality $\leq n$, there are only finitely many collections of at most $n$ balls, and only
a finite number of functions between any two subsets of $X$. This shows that there are only finitely many vertices of $K$.
\end{proof}

\section{Connectivity of the descending links} \label{sec: connectivityofdownwardlinks}

We continue to use the same notation as in the previous two sections.
In particular, $X$ denotes a non-empty, compact, ultrametric space with a finite similarity structure $\sm$. 
Moreover, $K$ denotes the similarity complex associated to $\sm$.

The goal of this section is to prove, under the assumptions in the Main Theorem~\ref{thm: main},
that the descending link of a vertex in $K$ is highly connected depending on the height of the vertex
(see Corollary~\ref{cor: high connectivity}).
The main technical result is Theorem~\ref{thm: connectivity of downward links}.

\begin{definition}
A  pseudo-vertex $v$ is \emph{contracting} if there exists $[f,B]\in \mathcal{E}$ such that
$v= \{[f|A, A] ~|~ A \text{~is a maximal proper sub-ball of $B$}\}.$
\end{definition}

Note that $v$ is contracting if and only if there exists $[f,B]\in \mathcal{E}$ such that $B$ is not a singleton 
and $v=\expa\{[f,B]\}$.
Note also that every simple contraction of a vertex $v$ takes place at a subset $w$ of $v$, where $w$ is a contracting
pseudo-vertex. 

\begin{definition} \label{def: pairwise disjoint simple contractions}
For $1\leq i \leq k$, let $v_i$ be pseudo-vertices each obtained by simple contractions of a pseudo-vertex $v$
at contracting pseudo-vertices $w_i\subseteq v$.
Then $v_1,\dots, v_k$ are \emph{obtained from $v$ by pairwise disjoint simple contractions} if $w_i \cap w_j =\emptyset$
whenever $i\not= j$.
\end{definition}

We note that, by the final line of Remark \ref{rem: well-defined expansions}, 
the property of being obtained from $v$ by pairwise disjoint simple
contractions is well-defined.

\begin{lemma}
\label{lem: expansion and order}
Suppose $v, w$, and $y$ are pseudo-vertices, $v\leq w$, $[f,B]\in v$, and
$[f,B]\notin w$.
If $\{[f,B]\}\nearrow y$, then $z := \left( v\setminus\{ [f,B]\}\right)\cup y$ is a pseudo-vertex and
$z\leq w$.
\end{lemma}

\begin{proof}
The fact that $f(B) =\image(y)$ implies that $z$ is a pseudo-vertex.
Now choose a sequence of simple expansions $v=v_1\nearrow v_2\nearrow\cdots\nearrow v_n= w$ and let $m$ be the greatest integer such that
$[f,B]\in v_m$. It follows that $v_{m+1} = \left( v_m\setminus\{ [f,B]\}\right)\cup y$
and $v\setminus\{ [f,B]\} \leq v_m\setminus\{ [f,B]\}$.
Thus, $z\leq v_{m+1}\leq w$.
\end{proof}

A pseudo-vertex $\widehat{q}$ is a \emph{maximal lower bound for $v_1,\dots, v_k$} if $\widehat{q}$ is a lower bound for $v_1,\dots, v_k$
and if $\widehat{q} < q$, then $q$ is not a lower bound for $v_1,\dots, v_k$. 
By contrast, $\widehat{q}$ is the \emph{greatest lower bound for $v_1,\dots, v_k$} if $\widehat{q}$ is a lower bound for $v_1,\dots, v_k$
and if $q$ is another lower bound for $v_1,\dots, v_k$, then $q\leq \widehat{q}$. 
A greatest lower bound is maximal, but the converse need not hold in 
arbitrary partially ordered sets. 

\begin{lemma}
\label{lem: max lower bounds}
Let $\widehat{q}$ be a maximal lower bound for $v_1,\dots, v_k$.
If $[g,A]\in \bigcap_{i=1}^k v_i$, then $[g,A]\in \widehat{q}$.
\end{lemma}

\begin{proof}
Remark~\ref{rem: unique containment} implies there exists a unique $[\widehat{g},\widehat{A}]\in \widehat{q}$ such that
$g(A)\subseteq \widehat{g}(\widehat{A})$. We note that, since $\widehat{g}(\widehat{A}) \cap g(A) \neq \emptyset$ and $v_1, \dots, v_k, \widehat{q}$
are pseudo-vertices,
either $[\widehat{g}, \widehat{A}] = [g,A]$, or $[\widehat{g}, \widehat{A}] \notin v_i$, for all $i = 1, \dots, k$. 
Let $y$ be such that $[\widehat{g},\widehat{A}] \nearrow y$.
If $[\widehat{g},\widehat{A}]\notin v_i$ for all $i= 1, \dots, k$, then Lemma~\ref{lem: expansion and order} implies that
$q' := y\cup( \widehat{q} \setminus \{[\widehat{g},\widehat{A}]\})\leq v_i$, for all $i = 1, \dots, k$.
Since $\widehat{q}\nearrow q'$, this contradicts maximality of $\widehat{q}$. 
Therefore, $[\widehat{g},\widehat{A}] = [g,A]$ and $[g,A]\in \widehat{q}$.
\end{proof}

\begin{lemma}
\label{lem: lower bounds}
Let $v$ be a pseudo-vertex containing distinct contracting pseudo-vertices $w_1, \dots, w_k$ 
and let $v_i$ be a pseudo-vertex obtained from a simple contraction of $v$ at $w_i$ for $1\leq i\leq k$.
\begin{enumerate}
	\item The pseudo-vertices $v_1,\dots, v_k$ have a lower bound if and only if  
	$v_1,\dots, v_k$ are obtained from $v$ by pairwise disjoint simple contractions.
	\item If the pseudo-vertices $v_1,\dots, v_k$ have a lower bound, then they have a greatest lower bound.
\end{enumerate}
\end{lemma}

\begin{proof}
For notation that will be used throughout the proof, choose $[f_i, B_i]\in\mathcal{E}$ such that $[f_i, B_i]\in v_i$ 
and if $u_i := \{ [f_i, B_i]\}$, then $u_i\subseteq v_i$ and $u_i\nearrow w_i$
for $1\leq i\leq k$.
Note that $v\setminus w_i = v_i\setminus u_i$ for $1\leq i\leq k$.

To prove the ``if'' part of the first statement, the assumption is that $w_i\cap w_j=\emptyset$
whenever $i\not= j$.
Define a pseudo-vertex
$$\widehat{v} = \left[ v\setminus  \bigcup_{i=1}^k w_i\right] \cup \bigcup_{i=1}^ku_i .$$
It follows that $\widehat{v}\leq v_j$ for $1\leq j\leq k$ as is amply illustrated for the case $j=k$:
$$\widehat{v} \nearrow \left[ v\setminus  \bigcup_{i=2}^k w_i\right] \cup \bigcup_{i=2}^ku_i 
\nearrow \left[ v\setminus  \bigcup_{i=3}^k w_i\right] \cup \bigcup_{i=3}^ku_i 
\nearrow \cdots \nearrow
\left[ v\setminus   w_k\right] \cup u_k = v_k 
,$$
where the  $\ell^{th}$ simple expansion in the sequence above uses $u_\ell\nearrow w_\ell$.

To prove the ``only if'' part of the first statement, it suffices to consider the case $k=2$.
Suppose $z$ is a lower bound of $v_1$ and $v_2$. The goal is to show $w_1\cap w_2=\emptyset$.
Suppose on the contrary that there exists $[f,B]\in w_1\cap w_2$.
Since $u_1\nearrow w_1$ and $u_2\nearrow w_2$, it follows that there exist maximal proper sub-balls,
$\widehat{B}_1\subseteq B_1$ and $\widehat{B}_2\subseteq B_2$,
such that $[f_1|\widehat{B}_1,\widehat{B}_1] =[f,B] = [f_2| \widehat{B}_2, \widehat{B}_2]$.
Since $z\leq v$ and $[f,B]\in v$, it follows from 
Remark \ref{rem: unique containment} that there exists a unique $[h,C]\in z$ such that
$f(B)\subseteq h(C)$. Now, since $z \leq v_{i}$ and $[f_{i}, B_{i}] \in v_{i}$ ($i=1,2$),
it follows from Remark \ref{rem: unique containment} that there are unique $[h_{i}, D_{i}] \in z$
($i=1,2$) such that $f_{i}(B_{i}) \subseteq h_{i}(D_{i})$ ($i=1,2$). Since
$[h_{1}, D_{1}]$, $[h_{2}, D_{2}]$, $[h, C] \in z$ and $z$ is a pseudo-vertex, we must have that any
two of $h_{1}(D_{1})$, $h_{2}(D_{2})$, $h(C)$ are either identical or disjoint. We have
$h_{i}(D_{i}) \cap h(C) \neq \emptyset$ for $i=1,2$, however (since $f(B)$ is a subset of both). It follows
that $h_{1}(D_{1}) = h_{2}(D_{2}) = h(C)$, and so $f_{i}(B_{i}) \subseteq h(C)$ for $i=1,2$. Since
$z$ expands to $v_{i}$ ($i=1,2$) and $f_{i}(B_{i}) \subseteq h(C)$, there exist sub-balls $C_{1}$, $C_{2} \subseteq C$
such that $[h|_{C_{i}}, C_{i}] = [f_{i}, B_{i}]$ ($i=1,2$). Since $v_{i}$ ($i=1,2$) expands to $v$, there exist sub-balls
$\widehat{C}_{1} \subseteq C_{1}$ and $\widehat{C}_{2} \subseteq C_{2}$ such that
$[h|\widehat{C}_1,\widehat{C}_1] = [f,B] = [h|\widehat{C}_2,\widehat{C}_2]$.  
In particular, $h(\widehat{C}_1)= f(B) = h(\widehat{C}_2)$, from which it follows that
$\widehat{C}_1=\widehat{C}_2$.

We will now show that $\widehat{C}_1$ is a maximal proper sub-ball of $C_1$.
There exists $g\in\sm(B,\widehat{B}_1)$ such that $f_1g=f$.
There exists $\widehat{h}\in\sm(B_1,C_1)$ such that $h\widehat{h}=f_1$.
Since $\widehat{B}_1$ is a maximal proper sub-ball of $B_1$, $\widehat{h}(\widehat{B}_1)$ is a 
maximal proper sub-ball of $C_1$.
Now $h\widehat{h}(\widehat{B}_1) = f_1(\widehat{B}_1) = fg^{-1}(\widehat{B}_1) = f(B) = h(\widehat{C}_1).$
Thus, $\widehat{h}(\widehat{B}_1) =\widehat{C}_1$ and $\widehat{C}_1$ is a maximal proper sub-ball of $C_1$
as claimed. 
Likewise, $\widehat{C}_2$ is a maximal proper sub-ball of $C_2$. Since $\widehat{C}_1=\widehat{C}_2$, 
it follows that $C_1=C_2$ (in an ultrametric space a ball is a maximal proper sub-ball of at most one ball).
Therefore, $[f_1,B_1] = [f_2,B_2]$; that is, $u_1=u_2$ and $w_1=w_2$, contradicting the assumption that
$w_1$ and $w_2$ are distinct.

To prove the second statement,
assuming $v_1,\dots, v_k$ have a lower bound (equivalently, they are obtained from $v$ by pairwise disjoint
simple contractions), we will show that the pseudo-vertex $\widehat{v}$ defined above is the greatest lower
bound of $v_1,\dots, v_k$.
Let $\widehat{q}$ be a maximal lower bound for $v_1,\dots, v_k$. 
Let $i \in \{ 1, \ldots, k \}$ be arbitrary. We claim that $u_{i} \subseteq \widehat{q}$. Since $\widehat{q} \leq v_{i}$
and $u_{i} = \{ [f_{i}, B_{i}] \} \subseteq v_{i}$, Remark \ref{rem: unique containment} implies that there is a unique 
$[\widehat{f}_{i}, \widehat{B}_{i}] \in \widehat{q}$ such that $f_{i}(B_{i}) \subseteq \widehat{f}_{i}(\widehat{B}_{i})$.
Suppose, for a contradiction, that $i \neq j$, but $[ \widehat{f}_{i}, \widehat{B}_{i}] \in v_{j}$. Since $w_{i} \in v_{j}$, we have
$$\expa\{ u_i \} = \{ [f_{i}|_{B'_{l}}, B'_{l}] \mid B'_{l} \text{ is a maximal proper subball of } B_{i} \} \subseteq v_{j}.$$
Clearly, each $f_{i}(B'_{l})$ is a proper subset of $\widehat{f}_{i}(\widehat{B}_{i})$. Since $v_{j}$ is a pseudo-vertex 
and $[f_{i}|_{B'_{l}}, B'_{l}]$, $[\widehat{f}_{i}, \widehat{B}_{i}] \in v_{j}$, we have a contradiction. Thus, 
$[\widehat{f}_{i}, \widehat{B}_{i}] \notin v_{j}$ if $i \neq j$. Now, if $[\widehat{f}_{i}, \widehat{B}_{i}] \notin v_{i}$,
then, by Lemma \ref{lem: expansion and order}, the simple expansion $\widetilde{q}_{i}$ of $\widehat{q}$ at 
$[\widehat{f}_{i}, \widehat{B}_{i}]$ satisfies $\widetilde{q}_{i} \leq v_{j}$ for all $j \in \{ 1, \ldots, k \}$, violating
maximality of $\widehat{q}$. Thus, $u_{i} \subseteq \widehat{q}$. It follows that $[f_{i}, B_{i}] \in \widehat{q}$, for
$i = 1, \ldots, k$.  

Lemma~\ref{lem: max lower bounds} implies that $z\subseteq \widehat{q}$.
Since $\image(u_1\cup\cdots\cup u_k\cup z) =\image(\widehat{v})$,
we must have $\widehat{q} = \bigcup_{i=1}^k u_i\cup z= \widehat{v}$.
\end{proof}

\begin{definition}
Let $v$ be a vertex.
\begin{enumerate}
	\item The \emph{descending link of $v$}, denoted $\dlk(v)$, is the subcomplex of $K$
spanned by $\{ v'\in K^0 ~|~ v' < v\}$.
	\item The \emph{complex below $v$}, denoted $B(v)$, is the subcomplex of $K$ 
spanned by $\{ v'\in K^0 ~|~ v' \leq v\}$.	
\end{enumerate}
\end{definition}

Note that the set of vertices of $B(v)$ is a directed set; in fact, it has a greatest element $v$.
Thus, $B(v)$ is contractible.

\begin{definition}
The \emph{nerve complex associated to a pseudo-vertex $v$}, denoted $\mathcal{N}_v$,
is the abstract simplicial complex of which 
a vertex is a pseudo-vertex  obtained from $v$ by a simple contraction and 
a $k$--simplex is a set of the form
$\{ v_0, \dots, v_k\}$, where $v_0,\dots, v_k$ are pseudo-vertices 
obtained from $v$ by pairwise disjoint simple contractions.
\end{definition}

\begin{remark}
\label{rem: nerve}
The reason for the nerve terminology is the following alternative interpretation
of $\mathcal{N}_v$ in the case $v$ is a vertex.
Recall that in general if $\mathcal{U}$ is a cover of a space, then the \emph{nerve of $\mathcal{U}$}
is the simplicial complex, denoted $N(\mathcal{U})$, whose vertices are the elements of $\mathcal{U}$ and such that a collection $\{U_0,\dots, U_n\}$
of vertices spans an $n$-simplex of $N(\mathcal{U})$ if and only if $\bigcap_{i=0}^nU_i\not=\emptyset$.
Let $v_1, \dots, v_n$ be the complete list of distinct vertices that can be obtained from $v$
by simple contractions; that is, $v_i\nearrow v$ for $1\leq i\leq n$.
Note that $\mathcal{U} = \{ B(v_1) , \dots, B(v_n)\}$ is a cover of $\dlk(v)$
by subcomplexes.
Moreover, a $k$--element subset  $\{ B(v_{i_1}) , \dots, B(v_{i_k})\}$ 
of $\mathcal{U}$ has a non-empty intersection if and only if $v_{i_1} , \dots, v_{i_k}$ have a lower bound.
By Lemma~\ref{lem: lower bounds}, this means that  $\{ B(v_{i_1}) , \dots, B(v_{i_k})\}$ 
has a non-empty intersection if and only if $v_{i_1} , \dots, v_{i_k}$
are obtained from $v$ by pairwise disjoint simple contractions.
Therefore, $\mathcal{N}_v$ is the nerve of the cover $\mathcal{U}$.
\end{remark}

\begin{proposition}
If $v$ is a vertex, then $\dlk(v)$ is homotopy equivalent to
$\mathcal{N}_v$.
\end{proposition}

\begin{proof}
Let $v_1, \dots, v_n$ be the complete list of distinct vertices that can be obtained from $v$
by simple contractions.
Using the alternative interpretation of $\mathcal{N}_v$ in Remark~\ref{rem: nerve} and a
standard fact about nerves of covers, (which may be found in Geoghegan \cite[Proposition 9.3.20]{Geoghegan}),
it suffices to show that $\bigcap_{j=1}^k B(v_{i_j})$ is contractible whenever it is non-empty.
The intersection is non-empty precisely when the vertices $v_{i_1}, \dots, v_{i_k}$ have a lower bound.
In that case, Lemma~\ref{lem: lower bounds} implies that the vertices have a greatest lower bound.
That is to say, $\bigcap_{j=1}^k B(v_{i_j})^{0}$ has a greatest element, in particular, it is a directed set.
Therefore, the intersection $\bigcap_{j=1}^k B(v_{i_j})$ is contractible.
\end{proof}

Recall that a simplicial complex $M$ is a \emph{flag complex} if every finite subset of vertices of $M$ that is pairwise
joined by edges spans a simplex.

\begin{lemma}
If $v$ is a pseudo-vertex, then $\mathcal{N}_v$ is a flag complex.
\end{lemma}

\begin{proof}
Let $v_0, \dots, v_k$ be vertices of $\mathcal{N}_v$ such that any pair spans a $1$-simplex of $\mathcal{N}_v$.
Thus, $v_0, \dots, v_k$ are pseudo-vertices obtained from $v$ by pairwise disjoint simple contractions.
That is to say, $\{v_0, \dots, v_k\}$ is a $k$-simplex of $\mathcal{N}_v$.
\end{proof}

We will need to assume the following property in order to establish our main finiteness result Theorem\ref{thm: main}.

\begin{definition}
\label{def: rich in contractions} 
The space $X$ together with $\sm$ is \emph{rich in simple contractions}
if there exists a constant $C_0 > 0$ such that if $k\geq C_0$ and $v$ is a pseudo-vertex of height $k$,
then there exists a pseudo-vertex $w\subseteq v$ with $\Vert w\Vert > 1$ and a simple contraction of $v$ at $w$.
\end{definition}

Note that the condition $\Vert w\Vert >1$ in the definition above is redundant because it is implied by the definition of a simple contraction.

The  property of Definition~\ref{def: rich in contractions} is the one that we will need in our proof; however, the following property,
which is a bit more cumbersome to state, is easier to
verify and implies rich in simple contractions.

\begin{definition}
\label{def: rich in ball contractions} 
The space $X$ together with $\sm$ is \emph{rich in ball contractions}
if there exists a constant $C_0 > 0$ such that if $k\geq C_0$ and $(B_1,\dots, B_k)$ is a  $k$--tuple of balls of $X$,
then there exists a ball $B\subseteq X$  such that if $\mathcal{M}_B := \{ A ~|~ \text{$A$ is a maximal, proper sub-ball of $B$}\}$,
then $|\mathcal{M}_B|>1$ and there is an injection $\sigma\co\mathcal{M}_B\to \{ (B_i,i) ~|~  1\leq i\leq k\}$
such that $[A]=[B_i]$ whenever $\sigma(A) = (B_i,i)$.
\end{definition}

\begin{proposition}
\label{prop: rich implies rich}
If $X$ together with $\sm$ is rich in ball contractions, then it is rich in simple contractions.
\end{proposition}

\begin{proof} Let $C_0$ be the constant given in Definition~\ref{def: rich in ball contractions}; we will show that 
Definition~\ref{def: rich in contractions} is satisfied with the same constant.
Let $v=\{ [f_i,B_i] ~|~ 1\leq i\leq k\}$ be a pseudo-vertex of height $k\geq C_0$.
Let $B\subseteq X$ be a ball such that $|\mathcal{M}_B| > 1$ and there exists an injection
$\sigma\co \mathcal{M}_B\to \{ (B_i,i)~|~ 1\leq i\leq k\}$.
Let $\sigma_1$ and $\sigma_2$ denote the first and second coordinates of $\sigma$, respectively; that is,
if $\sigma(A)=(B_i,i)$, then $\sigma_1(A)=B_i$ and $\sigma_2(A)=i$.
For each $A\in\mathcal{M}_B$, choose $h_A\in\sm(A,\sigma_1(A))$.
Define $f\co B\to X$ by setting
$f|A = f_{\sigma_2(A)} \circ h_A \co A\to X$ for each $A\in\mathcal{M}_B$.
Let $w=\{ [f_i,B_i] ~|~ i\in\image(\sigma_2)\}$. Then $w\subseteq v$ is a pseudo-vertex
and $\Vert w\Vert > 1$.
Define $u = \{[f,B]\}\cup v\setminus w$.
Clearly, $u$ is obtained from a simple contraction at $w$.
\end{proof}

\begin{example}
\label{example: vdh rich in ball contractions}
We let $A = \{ a_1, \dots, a_d \}$ be a finite alphabet, and consider, for arbitrary $H \leq \Sigma_{d}$, 
the finite similarity structure for $A^{\omega}$ from Definition \ref{def:symmetric group sim}.

We claim that $A^{\omega}$ with the given $\sm$ structure is rich in ball
contractions with $C_{0} = d$. Suppose $k \geq d$ and $(B_{1}, \ldots, B_{k})$ is 
a $k$-tuple of balls in $A^{\omega}$. We can write $(B_{1}, \ldots, B_{k}) = 
(u_{1}A^{\omega}, \ldots, u_{k}A^{\omega})$ for appropriate words $u_{1}, \ldots, u_{k} \in A^{\ast}$.
We consider $\mathcal{M}_{A^{\omega}} = \{ a_{i}A^{\omega} \mid a_{i} \in A \}$.
Let $\sigma: \mathcal{M}_{A^{\omega}} \rightarrow \{ (u_{i}A^{\omega}, i) \mid 1 \leq i \leq k \}$ 
be defined by $\sigma(a_{i}A^{\omega}) = (u_{i}A^{\omega}, i)$. This map is injective, and clearly
$\sm(a_{i}A^{\omega}, u_{i}A^{\omega}) \neq \emptyset$,
so $[a_{i}A^{\omega}] = [u_{i}A^{\omega}]$.
\end{example}

\begin{lemma}
\label{lem: C_1}
If the set of 
$\sm$-equivalence classes of balls in $X$ is finite, then there exists a constant $C_1$ such that $\Vert v\Vert \leq C_1$
whenever $v$ is a contracting pseudo-vertex.
\end{lemma}

\begin{proof}
Let $[B_1],\dots, [B_n]$ be the set of $\sm$-equivalence classes of balls in $X$.
Let $N_i$ be the number of maximal, proper sub-balls of $B_i$ for $1\leq i\leq n$.
Define $C_1:=\max\{ N_i ~|~ 1\leq i\leq n\}$.
If $v$ is a contracting pseudo-vertex, then there exist $i\in\{ 1,\dots, n\}$ and $[f,B_i]\in\mathcal{E}$
such that $v=\expa\{[f,B_i]\}$. Thus, $\Vert v\Vert\leq N_i$.
\end{proof}

\begin{hypothesis}
\label{hyp: hyp on X}
The following two conditions are satisfied.
\begin{enumerate}
	\item There exists at most finitely many $\sm$-equivalence classes of balls of $X$ and $C_1>0$ is the constant given by
	Lemma~\ref{lem: C_1}.
	\item The space $X$ together with $\sm$ is rich in simple contractions and $C_0>0$ is the constant in Definition~\ref{def: rich in contractions}. 
\end{enumerate}
\end{hypothesis}

For the proof of Theorem~\ref{thm: connectivity of downward links} we need the following three results concerning connectivity in simplicial complexes.

Recall that the \emph{star} of a vertex $v$ in a simplicial complex $M$, denoted $\sta(v,M)$, or $\sta(v)$
if $M$ is understood, is the subcomplex of $M$ consisting of all the simplices containing $v$, together with the faces of these simplices.
The \emph{link} of a vertex $v$ in a simplicial complex $M$, denoted $\lk(v,M)$ or $\lk(v)$, consists of all simplices in
$\sta(v,M)$ that do not contain $v$.

A reference for the following well-known result is Bj\"orner \cite[Theorem 10.6, page 1850]{Bjorner}.

\begin{theorem}[Nerve Theorem]
\label{thm: nerve} 
Let $M$ be a simplicial complex and let $\{ M_i\}_{i\in I}$ be a family of subcomplexes such
that $M=\bigcup_{i\in I} M_i$. If every non-empty intersection $M_{i_1}\cap\cdots \cap M_{i_t}$ is $(k-t+1)$-connected, then
$M$ is $k$-connected if and only if the nerve of the cover $\{ M_i\}_{i\in I}$ is $k$-connected. 
\end{theorem}

\begin{lemma}
\label{lem: stars and links}
Suppose $v_1,\dots, v_n$ are vertices in a flag complex $M$.
If 
$$\text{$\bigcap_{i=1}^n \sta(v_i,M) \not=\emptyset$\quad but \quad$\bigcap_{i=1}^n \lk(v_i,M) =\emptyset$,}$$
then $\bigcap_{i=1}^n \sta(v_i,M)$ is a simplex.
\end{lemma}

\begin{proof}
By the flag property, it suffices to show that any two vertices of $\bigcap_{i=1}^n \sta(v_i,M)$
are adjacent. If $u, w$ are vertices of $\bigcap_{i=1}^n \sta(v_i,M)$, then, since the intersection of the links is empty,
$u, w\in \{v_1, \dots, v_n\}$. It follows that $w\in\sta(u,M)$, which is to say $u$ and $w$ are adjacent.
\end{proof}

The following result is due to Farley \cite[Lemma 6]{FarHFP}. We only require the second item;
however, we state both parts in order to clarify the statement in \cite{FarHFP}.

\begin{lemma}[Farley]
\label{lem: Dan's lemma 6}
Let $M$ be a non-empty finite flag complex.
\begin{enumerate}
	\item Assume $k\geq 0$ and for any collection $S$ of vertices of $M$ such that $|S| \geq 2$,
	$$\bigcap_{v\in S} \lk(v) \text{ is $(k-|S|+1)$-connected}.$$
	Then $M$ is $k$-connected.
	\item Assume $n\geq -1$. If $S$ is any collection of vertices of $M$ and $\bigcap_{v\in S} \lk(v)$ is $n$-connected, then
	so is $\bigcap_{v\in S}\sta(v)$.
\end{enumerate}
\end{lemma}

We are now ready for the main technical result of this section.

\begin{theorem}
\label{thm: connectivity of downward links}
If Hypothesis~\ref{hyp: hyp on X} is satisfied, $v$ is a pseudo-vertex, $k\geq -1$ is an integer, and
$$\Vert v\Vert \geq (2k+2)C_1+C_0,$$ 
then $\mathcal{N}_v$ is $k$-connected.
\end{theorem}

\begin{proof}
The proof is by induction on $k$. We begin with the case $k=-1$. Then $\Vert v\Vert\geq C_0$.
Thus, there exist a pseudo-vertex $w\subseteq v$ and a simple contraction of $v$ at $w$.
Let $v_1$ be a pseudo-vertex resulting from such a simple contraction.
Hence, $v_1$ is a vertex of $\mathcal{N}_v$; that is, $\mathcal{N}_v\not=\emptyset$, which is to say
$\mathcal{N}_v$ is $(-1)$-connected.

Now consider the case $k=0$. Then $\Vert v\Vert\geq 2C_1+C_0$.
To show that $\mathcal{N}_v$ is $0$-connected, let $v_1, v_2$ be vertices of $\mathcal{N}_v$.
Thus, there exist pseudo-vertices $w_1, w_2\subseteq v$ such that $v_i$ is obtained from a simple contraction
of $v$ at $w_i$ for $i=1,2$.
Thus, $w_1, w_2$ are contracting pseudo-vertices and
$\Vert w_i\Vert \leq C_1$ for $i=1,2$.
Hence, $\Vert v\setminus (w_1\cup w_2)\Vert \geq C_0$ and so there is a pseudo-vertex $w\subseteq v\setminus (w_1\cup w_2)$
and a pseudo-vertex $v_3'$ resulting from a simple contraction of $v\setminus (w_1\cup w_2)$ at $w$.
It follows that $v_3:=v_3'\cup w_1\cup w_2$ is a pseudo-vertex such that $v_3\nearrow v$.
Therefore, since, $w_1\cap w=\emptyset = w_2\cap w$,
$\{ v_1, v_3\}$ and $\{ v_2, v_3\}$ are $1$-simplices of $\mathcal{N}_v$ showing that 
$v_1$ and $v_2$ are in the same component.

Now suppose $k> 0$ and that the nerve complex $\mathcal{N}_w$ is $\ell$-connected whenever $w$ is a pseudo-vertex, $-1\leq\ell<k$, and
$\Vert w\Vert \geq (2\ell+2)C_1+C_0$.
We continue to let $v$ be a pseudo-vertex with $\Vert v\Vert \geq (2k+2)C_1+C_0$.
We will show that $\mathcal{N}_v$ is $k$-connected by appealing to the Nerve Theorem~\ref{thm: nerve}.
Let $v_1, \dots, v_n$ be the distinct pseudo-vertices obtained from $v$ by simple contractions
(since $\Vert v\Vert\geq C_0$, $n\geq 1$). 
Thus, $v_1,\dots, v_n$ are the vertices of $\mathcal{N}_v$ and 
$\mathcal{N}_v= \bigcup_{i=1}^n\sta(v_i,\mathcal{N}_v)$.
To apply the Nerve Theorem~\ref{thm: nerve}, we must verify the following two items.
\begin{enumerate}
	\item If $\emptyset \not= 
	\{ i_1,\dots, i_t\} \subseteq \{ 1,\dots, n\}$ 
	and $S:= \sta(v_{i_1},\mathcal{N}_v)\cap\cdots\cap\sta(v_{i_t},\mathcal{N}_v)\not=\emptyset$,
	then $S$ is $(k-t+1)$-connected.
	\item The nerve of the cover $\{ \sta(v_i,\mathcal{N}_v)\}_{i=1}^n$ is $k$-connected.
\end{enumerate}

We begin by setting notation. For $1\leq i\leq n$ there is $w_i\subseteq v$ such that $v_i$ is obtained from $v$ by a simple contraction
at $w_i$. By the choice of constants, $\Vert w_i\Vert \leq C_1$ for $1\leq i\leq n$.

We now begin the verification of item (1).
If $t=1$, then $S$ is a star, which is contractible.
Now assume $t\geq 2$. 
If $\lk(v_{i_1},\mathcal{N}_v)\cap\cdots\cap\lk(v_{i_t},\mathcal{N}_v) =\emptyset$, then Lemma~\ref{lem: stars and links} implies
that $S$ is a simplex. 
Hence, we may assume that
$\lk(v_{i_1},\mathcal{N}_v)\cap\cdots\cap\lk(v_{i_t},\mathcal{N}_v) \not=\emptyset$.
Lemma~\ref{lem: Dan's lemma 6}(2) implies that it suffices to show that 
$\lk(v_{i_1},\mathcal{N}_v)\cap\cdots\cap\lk(v_{i_t},\mathcal{N}_v)$ is $(k-t+1)$-connected.
If $t \geq k+2$, then $-1 \geq k-t+1$, and there is nothing to prove (since $\lk(v_{i_1},\mathcal{N}_v)\cap\cdots\cap\lk(v_{i_t},\mathcal{N}_v) \neq \emptyset$
by hypothesis). 
Thus, we may assume $t\leq k+1$.
Define $u := v\setminus (w_{i_1}\cup\cdots\cup w_{i_t})$ and estimate the height of $u$:
$$\Vert u\Vert \geq (2k+2)C_1+C_0 - tC_1  = (2k-t+2)C_1 +C_0.$$
Since $t\leq k+1$, $\Vert u\Vert\geq C_0$ and $\mathcal{N}_u\not=\emptyset.$

We now show that $\lk(v_{i_1},\mathcal{N}_v)\cap\cdots\cap\lk(v_{i_t},\mathcal{N}_v)$ is isomorphic to  $\mathcal{N}_u$.
We begin by showing that $\mathcal{N}_u$ is isomorphic to a subcomplex $\mathcal{N}_u'$ of $\mathcal{N}_v$.
If $y$ is a pseudo-vertex obtained from $u$ by a simple contraction at $w\subseteq u$, then let
$y' := y\cup(v\setminus u)$. Thus, $y'\nearrow v$, so $y'$ is a vertex of $\mathcal{N}_v$.
Define a simplical map $\mathcal{N}_u\to \mathcal{N}_v$ by $y\mapsto y'$. This induces an isomorphism of $\mathcal{N}_u$
onto its image $\mathcal{N}_u'$.
We now show that $\lk(v_{i_1},\mathcal{N}_v)\cap\cdots\cap\lk(v_{i_t},\mathcal{N}_v)=\mathcal{N}_u'$.
Let $v_{m}$ be obtained from $v$ by a simple contraction at $w_{m}$. 
For $1\leq j\leq t$, the vertex $v_m$ of $\mathcal{N}_v$ is in $\lk(v_{i_j},\mathcal{N}_v)$ if and only if
$\{ v_{i_j}, v_m\}$ is a $1$-simplex of $\mathcal{N}_v$. 
Thus, $v_m$  is in $\lk(v_{i_j},\mathcal{N}_v)$ if and only if
$v_{i_j}$ and $v_m$ are obtained from $v$ by disjoint simple contractions. 
Therefore, $v_m\in \lk(v_{i_1},\mathcal{N}_v)\cap\cdots\cap\lk(v_{i_t},\mathcal{N}_v)$ if and only if $w_m\subseteq u$.
It follows that $\mathcal{N}_u'$ and $\lk(v_{i_1},\mathcal{N}_v)\cap\cdots\cap\lk(v_{i_t},\mathcal{N}_v)$ are both subcomplexes of
$\mathcal{N}_v$ with the same sets of vertices. Since they are both full subcomplexes, they are equal.
(Recall that a subcomplex $A$ of a complex $B$ is \emph{full} if $A$ is the largest subcomplex of $B$ having $A^0$ as its $0$-skeleton.
It is obvious that $\mathcal{N}_u'$ is a full subcomplex of $\mathcal{N}_v$. 
In general, the link of a vertex in a flag complex is a full subcomplex. Moreover, intersections of full subcomplexes are full.)

Define $\ell := k-t+1$. To finish the verification of item (1), we need to show that $\mathcal{N}_u$ is $\ell$-connected.
Since $t \geq 2$, we have $\ell < k$. Therefore, we will be able to invoke the inductive hypothesis to conclude that
$\mathcal{N}_u$ is $\ell$-connected if it is true that $\Vert u\Vert \geq (2\ell +2)C_1+C_0$.
We continue from the estimate above:
\begin{align*}
\Vert u\Vert & \geq (2k-t+2)C_1 +C_0  \\ 
& = (2(\ell +t-1)-t+2)C_1 +C_0  \\
& =(2\ell +t )C_1 +C_0 \\ 
& \geq (2\ell +2 )C_1 +C_0,
\end{align*}
which completes the verification of item (1).

For the verification of item (2),
let $M$ denote the nerve of the cover $\{ \sta(v_i,\mathcal{N}_v)\}_{i=1}^n$ of $\mathcal{N}_v$.
To show that $M$ is $k$-connected, it suffices to prove that
the $(k+1)$-skeleton of  $M$ is isomorphic to the $(k+1)$-skeleton of
the $n$-simplex.
Thus, we need to show that if $1\leq t\leq k+2$, then any collection of $t$ vertices of $M$ spans a $(t-1)$-simplex in $M$.
To this end, let 
$\emptyset \not= \{ i_1,\dots, i_t\} \subseteq \{ 1,\dots, n\}$, $1\leq t\leq k+2$, 
and show that $S:= \sta(v_{i_1},\mathcal{N}_v)\cap\cdots\cap\sta(v_{i_t},\mathcal{N}_v)$ is non-empty.
As above, let $u := v\setminus (w_{i_1}\cup\cdots\cup w_{i_t})$.
The estimate of the height if $u$ (using $1\leq t\leq k+2$)
is
$$\Vert u\Vert \geq (2k+2)C_1+C_0 - tC_1  = (2k-t+2)C_1 +C_0 \geq (k+1)C_1 +C_0\geq C_0.$$
This implies that there exist a pseudo-vertex $w\subseteq u$ and a simple contraction
of $u$ at $w$.
Let $y$ be the resulting pseudo-vertex. Thus, $y\nearrow u$ and 
$y\cup(w_{i_1}\cup\cdots\cup w_{i_t})\nearrow v$.
Hence, $\widehat{y} := y\cup(w_{i_1}\cup\cdots\cup w_{i_t})$ is a vertex of $\mathcal{N}_v$.
Since $w$ is disjoint from $w_{i_1}\cup\cdots\cup w_{i_t}$, 
it follows that for $1\leq j\leq t$, $\{v_{i_j}, \widehat{y}\}$ is a $1$-simplex of $\mathcal{N}_v$
and, in particular, $\widehat{y}\in\sta(v_{i_j},\mathcal{N}_v)$.
Thus, $\widehat{y}\in S$ and $S\not=\emptyset$, as desired.
\end{proof}

\begin{corollary}
Suppose Hypothesis~\ref{hyp: hyp on X} is satisfied.
If  $v$ is a vertex of $K$ such that 
$$\Vert v\Vert \geq (2k+2)C_1+C_0,$$ 
then $\lk_\downarrow v$ is $k$-connected.
\end{corollary}

\begin{corollary}
\label{cor: high connectivity} 
Suppose Hypothesis~\ref{hyp: hyp on X} is satisfied.
There is a function $f\co\bn\to\bn\cup\{0, -1\}$ such that 
$\lim_{n\to\infty}f(n)=\infty$ and 
if $v$ is a vertex, then $\lk_\downarrow v$ is $f(\Vert v\Vert)$--connected.
\end{corollary}

\section{The zipper action of an FSS group on the similarity complex} \label{sec: zipper}

Throughout this section, $X$ will 
denote a compact 
ultrametric space  with a finite similarity structure $\sm=\sm_X$
and
$\Gamma=\Gamma(\sm)$ will be the FSS group associated to $\sm$.

The goal of this section is to define an action of $\Gamma$ on the similarity complex and use this action, together with
Brown's finiteness criterion \cite{Brown1987}, to prove the Main Theorem~\ref{thm: main} (see Theorem~\ref{thm: main in body} below).
We also show that the similarity complex $K$ is a model for $\underline{E}\Gamma$, the classifying space for proper $\Gamma$ actions
(see Proposition~\ref{prop: E fin}).

We begin by recalling the action of $\Gamma$ on $\mathcal{E}$ as defined in Hughes \cite{HugLSHP}.
The \emph{zipper action} is the left action 
$\Gamma\curvearrowright\mathcal E$ defined by
$\gamma [f,B] = [\gamma f,B]$. The fact that $[\gamma f,B]\in\mathcal{E}$ follows from the
Compositions and Restrictions Properties of the similarity structure. 

\begin{remark}
\label{rem:action properties}
The zipper action $\Gamma\curvearrowright\mathcal E$ extends to an action of $\Gamma$ on the set of all pseudo-vertices
as follows. If $\gamma\in\Gamma$ and $v= \{ [f_i,B_i] ~|~ 1\leq i\leq k\}$ is a pseudo-vertex, then
$\gamma v := \{ [\gamma f_i,B_i] ~|~ 1\leq i\leq k\}$.
The following facts are easily verified.
\begin{enumerate}
	\item Height is $\Gamma$-invariant; that is, if $\gamma\in\Gamma$ and $v$ is a pseudo-vertex, then
	$||\gamma v|| = ||v||$.
	\item $K^0$ is $\Gamma$-invariant; that is, if $g\in\Gamma$ and $v$ is a vertex, then $gv$ is a vertex.
	\item If $v\nearrow w$, where $v$ and $w$ are pseudo-vertices, and $\gamma\in\Gamma$, then $\gamma v\nearrow \gamma w$.
	\item If $\gamma\in\Gamma$ permutes the vertices of an $n$-simplex $\Delta$ of $K$, then $\gamma$ fixes each vertex of $\Delta$.
\end{enumerate}
It follows that the partial order on pseudo-vertices is preserved by the $\Gamma$--action.
Hence, there is an induced simplicial action $\Gamma\curvearrowright K$. Each of the actions of $\Gamma$ on pseudo-vertices,
on vertices, and on $K$ are called the \emph{zipper action}.
\end{remark}

The next task is to characterize orbits under the zipper action.

\begin{lemma} 
\label{lem: orbits}
Let $v=\{[f_i,B_i] ~|~ 1\leq i\leq k\}$ be a pseudo-vertex. 
If the pseudo-vertex $w=\{[\widehat{f}_i,\widehat{B}_i] ~|~ 1\leq i\leq\ell\}$ is in the
$\Gamma$--orbit of $v$, then
$k=\ell$ and there exists a permutation $\sigma$ of $\{1,\dots, k\}$ such that
$\sm(B_i,\widehat{B}_{\sigma(i)})\not=\emptyset$ for $i=1,\dots, k$.
If $v$ is a vertex, then the converse holds.
\end{lemma}

\begin{proof}
Assume first that $w=\gamma v$ for some $\gamma\in\Gamma$.
The fact from Remark~\ref{rem:action properties} that height is $\Gamma$--invariant 
implies $k=\ell$. Since the sets $\gamma v=\{[\gamma f_i,B_i] ~|~ 1\leq i\leq k\}$
and $w=\{[\widehat{f}_i,\widehat{B}_i] ~|~ 1\leq i\leq k\}$ are equal, there exists a permutation
$\sigma$ such that $[\gamma f_i,B_i] = [\widehat{f}_{\sigma(i)},\widehat{B}_{\sigma(i)}]$ for each $i=1,\dots, k$.
The definition of the equivalence relation immediately implies
$\sm(B_i,\widehat{B}_{\sigma(i)})\not=\emptyset$ for $i=1,\dots, k$.

Conversely, if $v$ is a vertex, choose $h_i\in\sm(B_i,\widehat{B}_{\sigma(i)})$ for each $i=1,\dots, k$.
Define $\gamma\in\Gamma$ by 
$\gamma|f_i(B_i) = \widehat{f}_{\sigma(i)}h_if_i^{-1}\co f_i(B_i)\to \widehat{f}_{\sigma(i)}(B_{\sigma(i)})$.
Since $v$ and $w$ are vertices, $X=\coprod_{i=1}^k f_i(B_i)= \coprod_{i=1}^k \widehat{f}_{\sigma(i)}(\widehat{B}_{\sigma(i)}) $ and so $\gamma$ is a homeomorphism on $X$.
Since $\Gamma$ is the maximal group of homeomorphisms locally determined by $\sm$, it follows that $\gamma\in\Gamma$.
Clearly, $\gamma v=w$.
\end{proof}

We next show that the zipper action has finite vertex stabilizers.

\begin{lemma}
\label{lem: finite isotropy}
The isotropy group of any vertex of $K$ under the zipper action is a finite subgroup of $\Gamma$.
\end{lemma}

\begin{proof}
Let $v=\{ [f_i,B_i] ~|~ 1\leq i\leq k\}$ be a vertex of height $k$, where representatives $(f_i,B_i)\in\mathcal{S}$, $1\leq i\leq k$, have been chosen
for each member of $v$.
Let $\Gamma_v$ be the isotropy subgroup of $\Gamma$ fixing $v$.
Let $\Sigma_k$ be the set of permutations of $\{1,\dots, k\}$.
The proof will be completed by defining an injection
$$\Psi\co \Gamma_v\to\coprod_{\sigma\in\Sigma_k}\prod_{i=1}^k \sm(B_i, B_{\sigma(i)}).$$
Given $\gamma\in\Gamma_v$, $v=\gamma v = \{ [\gamma f_i,B_i] ~|~ 1\leq i\leq k\}$ 
implies there exists a unique $\sigma\in\Sigma_k$ such that
$[\gamma f_i, B_i]= [f_{\sigma(i)}, B_{\sigma(i)}]$ for $1\leq i\leq k$.
It follows that $f_{\sigma(i)}^{-1}\gamma f_i\in\sm(B_i,B_{\sigma(i)})$ for $1\leq i\leq k$.
Define 
$$\Psi(\gamma) = (f_{\sigma(1)}^{-1}\gamma f_1, \dots, f_{\sigma(k)}^{-1}\gamma f_k)\in\prod_{i=1}^k\sm(B_i, B_{\sigma(i)}).$$
To see that $\Psi$ is injective, suppose we are given another element $\beta\in\Gamma_k$ and
$\Psi(\gamma)= \Psi(\beta)$. It follows that $\gamma|f_i(B_i) = \beta|f_i(B_i)$ for $1\leq i\leq k$.
Thus, $\gamma=\beta$ since $X=\coprod_{i=1}^k f_i(B_i)$.
\end{proof}

We next show that the zipper action restricted to sub-level sets is cocompact if 
the set of $\sm$--equivalence classes of balls in $X$ is finite.

\begin{proposition}
\label{prop: cocompactness}
If the set of $\sm$--equivalence classes of balls in $X$ is finite and $n\in\bn$, 
then the sub-level set $K_{\leq n}$ is $\Gamma$-finite; that is, $\Gamma\backslash K_{\leq n}$ is a finite complex.	
\end{proposition}

\begin{proof}
By Proposition \ref{prop: locally finite filtration}, it suffices to show 
that $\Gamma \backslash K_{\leq k}^{0}$ is finite for each $k=1,2,3,\ldots$. 
Let $[B_{1}], \ldots, [B_{l}]$ be the distinct $\sm$-equivalence classes of balls in $X$. For a vertex
$v \in K^{0}$, define
$$ n_{[B_{i}]}(v) = | \{ [\widehat{f}, \widehat{B}] \in v \mid [\widehat{B}] = [B_{i}] \}|.$$
Let $\widetilde{\phi}: K^{0} \rightarrow \prod_{i=1}^{l} \left( \mathbb{N} \cup \{ 0 \} \right)$ be defined
by $\widetilde{\phi} = n_{[B_{1}]} \times \ldots \times n_{[B_{l}]}$. By Lemma \ref{lem: orbits},
$\widetilde{\phi}$ descends to a well-defined injection $\phi$ on the quotient; that is,
$\phi: \Gamma \backslash K^{0} \rightarrow \prod_{i=1}^{l} \left( \mathbb{N} \cup \{ 0 \} \right)$. Fixing a height
$k$, we get an injection
$\phi_{k}: \Gamma \backslash K_{\leq k}^{0} \rightarrow \prod_{i=1}^{l} \left( \mathbb{N} \cup \{ 0 \} \right)$, where
the entries of an element in the image of $\phi_{k}$ must add up to $k$. There are $\binom{k+l-1}{k}$ distinct
ordered $l$-tuples of non-negative integers which add up to $k$, so
$| \Gamma \backslash K_{\leq k}^{0} | \leq \binom{k+l-1}{k}$.
\end{proof}

We can now prove the Main Theorem~\ref{thm: main}, which is restated here.

\begin{theorem}[Main Theorem]
\label{thm: main in body}
If Hypothesis~\ref{hyp: hyp on X} is satisfied, then $\Gamma$ is of type $F_\infty$.
\end{theorem}

\begin{proof}
This is a standard application of Brown's criterion for finiteness.
See Geoghegan \cite[Section 7.4]{Geoghegan} for an exposition, and 
\cite[Exercise, page 179]{Geoghegan} for the result we need.

We refer to the original statement from 
Brown \cite[Corollary 3.3, part a]{Brown1987}. Note that the similarity complex $K$ is a contractible 
$\Gamma$-complex (Proposition~\ref{prop: contractible}; Remark~\ref{rem:action properties}),
it is filtered by the $\Gamma$-finite $\Gamma$-complexes $K_{\leq n}$ (Proposition~\ref{prop: cocompactness};  Remark~\ref{rem:action properties}(1)), 
and the stabilizer of each vertex is finite  (Lemma~\ref{lem: finite isotropy}). The final point to check is that the connectivity of the pair
$(K_{\leq n+1}, K_{\leq n})$  tends to infinity as $n$ tends to infinity. We may assume that there are vertices of height $n+1$, so 
$K_{\leq n+1} \neq K_{\leq n}$. The complex $K_{\leq n+1}$, up to homotopy, is $K_{\leq n}$ with a collection $\{ C_i \}_{i \in \mathcal{I}}$
of cones attached along their bases, each of which is homotopy equivalent to the descending link of a vertex of height $n+1$. By 
Corollary~\ref{cor: high connectivity}, the connectivity of such descending links tends to infinity with $n$, 
and it follows (from elementary Mayer-Vietoris
and van Kampen arguments) that 
the connectivity of $(K_{\leq n+1}, K_{\leq n})$ tends to infinity as well. Therefore, $\Gamma$ has type $F_{\infty}$.      

See Farley~\cite{FarHFP} for an illustration of how to put these ingredients together in a related context.
\end{proof} 

\begin{corollary}
\label{cor: Nekrashevych-Roever examples F_infinity}
The groups $V_{d}(H)$ have type $F_{\infty}$, for all $d \in \mathbb{N}$ and $H \leq \Sigma_{d}$.
\end{corollary}

\begin{proof}
We fix $A = \{ a_{1}, \ldots, a_{d} \}$. Recall that $\Sigma_{d}$ denotes the symmetric group on $A$. We choose $H \leq \Sigma_{d}$. 
We equip the space $A^{\omega}$ with the finite similarity structure $\sm$ 
from Definition~\ref{def:symmetric group sim}. 
By Remark~\ref{claim: Vdh is Gamma}, $V_{d}(H)$ is the FSS group associated to $\sm$. 
By Example~\ref{example: vdh rich in ball contractions},
$A^{\omega}$ with the given finite similarity structure is rich in ball contractions with 
constant $C_{0} = d$. Since there is only one
$\sm$-equivalence class $[B]$ of balls, and $B$ has $d$ maximal 
proper sub-balls, Hypothesis~\ref{hyp: hyp on X}(1) is satisfied with $C_{1} = d$.
Thus, $V_{d}(H)$ has type $F_{\infty}$ by Theorem \ref{thm: main in body}.   
\end{proof}

\begin{example}
The action $\Gamma\curvearrowright K$ is usually not free.
In fact, a vertex $v= \{ [f_{i}, B_{i}] \mid 1 \leq i \leq m \}$ will have a non-trivial stabilizer in either
of the following cases:
\begin{enumerate}
\item \label{ex: nonfree1} if $[B_{i}] = [B_{j}]$ for some $[f_{i}, B_{i}] \neq [f_{j}, B_{j}]$ with $i,j \in \{ 1, \ldots, m \}$, or
\item \label{ex: nonfree2} if the group $\sm_{X}(B_{i}, B_{i}) \neq \{ \id_{B_{i}} \}$ for some $[f_{i}, B_{i}] \in v$.
\end{enumerate}
For (\ref{ex: nonfree1}), suppose $[B_{i}] = [B_{j}]$ and $[f_{i}, B_{i}] \neq [f_{j}, B_{j}]$. We choose 
$h \in \sm_{X}(B_{i}, B_{j})$ and define a homeomorphism $g: X \rightarrow X$ as follows. If $k \neq i,j$, then 
$g|_{f_{k}(B_{k})} = \id_{f_{k}(B_{k})}$. We set $g|_{f_{i}(B_{i})} = f_{j}hf_{i}^{-1}$
and $g|_{f_{j}(B_{j})} = f_{i}h^{-1}f_{j}^{-1}$. These assignments completely determine $g$ on all of $X$, since
$\{ f_{1}(B_{1}), \ldots, f_{m}(B_{m}) \}$ is a partition of $X$. The map $g: X \rightarrow X$ is continuous
since the partition $\{ f_{1}(B_{1}), \ldots, f_{m}(B_{m})\}$ is made up of open (and, therefore, also closed) sets,
and $g$ is continuous on each piece. The map $g$ is bijective since it induces a bijection on the partition
$\{ f_{1}(B_{1}), \ldots, f_{m}(B_{m}) \}$, and $g$ also maps any element of the partition bijectively to another
such element. Lastly, $g$ is locally determined by $\sm_{X}$ since it is locally determined by $\sm_{X}$ on each
piece $f_{i}(B_{i})$, $i = 1, \ldots, m$. It follows that $g \in \Gamma(\sm_{X})$. One easily checks that 
$[gf_{k}, B_{k}] = [f_{k}, B_{k}]$ for $k \neq i,j$, $[gf_{i}, B_{i}] = [f_{j}, B_{j}]$
and $[gf_{j}, B_{j}] = [f_{i}, B_{i}]$. Thus $g \cdot v = v$. On the other hand, $g$ is not the identity, since
$g(f_{i}(B_{i})) \cap f_{i}(B_{i}) = f_{j}(B_{j}) \cap f_{i}(B_{i}) = \emptyset$.

For (\ref{ex: nonfree2}), suppose $\psi \in \sm_{X}(B_{i}, B_{i})$, where 
$\psi \neq \id_{B_i}$. We define $g \in \Gamma(\sm_{X})$ such that $g|_{f_{k}(B_{k})} = \id_{f_{k}(B_{k})}$ when
$k \neq i$, and $g|_{f_{i}(B_{i})} = f_{i}\psi f_{i}^{-1}$. By reasoning similar to that from Case (\ref{ex: nonfree1}),
$g$ is a non-trivial element of $\Gamma(\sm_{X})$, $g \cdot v = v$, and $g \neq \id_{X}$ since
$g|_{f_{i}(B_{i})} = f_{i} \psi f_{i}^{-1} \neq \id_{f_{i}(B_{i})}$.
\end{example}

\begin{example}
The quotient $\Gamma\backslash K$ is usually not locally finite.
In fact, the following are equivalent:
\begin{enumerate}
	\item \label{ex: nlf1} $\Gamma \backslash K$ is finite.
	\item \label{ex: nlf2} $\Gamma \backslash K$ is locally finite.
	\item \label{ex: nlf3} $X$ is finite.
\end{enumerate}
\begin{proof} 
It is clear that (\ref{ex: nlf1}) implies (\ref{ex: nlf2}). If $X$ is finite, then $K$ is finite by Remark 
\ref{rem: K not locally finite}, so $\Gamma \backslash K$
will also be finite. If $X$ is infinite, then the argument from Remark \ref{rem: K not locally finite} shows that
there is an infinite chain of vertices $v_0 < v_1 < v_2 < \ldots $. Any two of these vertices are adjacent in $K$,
and at different heights. Since the action of $\Gamma$ preserves height by Remark \ref{rem:action properties},
the vertex $v_0$ is adjacent to infinitely many vertices in the quotient $\Gamma \backslash K$. Thus $\Gamma \backslash
K$ is not locally finite.  
\end{proof}
\end{example}

\subsection*{The similarity complex as a classifying space}
We now show that the similarity complex $K$ is a classifying space with finite isotropy; that is,
$K$ is a model for $E_{\text{Fin}}\Gamma$, where $\Gamma$ is the FSS group associated to the given finite similarity structure and
$\text{Fin}$ denotes the family of finite subgroups of $\Gamma$. 

\begin{definition} \label{def: family}
If $\Gamma$ is any group, then a \emph{family} $\mathcal{F}$ of subgroups
of $\Gamma$ is a non-empty collection of subgroups that is closed under conjugation by elements of $\Gamma$
and passage to subgroups.  If $\Gamma$ is any group, then we let ${Fin}$
denote the family of finite groups.
\end{definition}

\begin{definition} \label{def: EFin}
Let $X$ be a $\Gamma$-CW complex.  Suppose that, if $c \subseteq X$ is a cell of $X$, then $\gamma \cdot c = c$ if and only if
$\gamma$ fixes $c$ pointwise.  Let $\mathcal{F}$ be a family of subgroups of $\Gamma$.  We say that $X$ is an \emph{$E_{\mathcal{F}}\Gamma$-complex} if
\begin{enumerate}
\item $X$ is contractible;
\item whenever $H \in \mathcal{F}$, the fixed set $Fix(H) = \{ x \in X \mid \gamma \cdot x = x \text{ for all } \gamma \in H \}$ is contractible;
\item whenever $H \not \in \mathcal{F}$, $Fix(H)$ is empty.
\end{enumerate}
\end{definition}

\begin{proposition}
\label{prop: E fin}
$K$ is a model for $E_{Fin}= \underline{E}\Gamma$; that is, the fixed set by the action on $K$ of a subgroup $G$ of $\Gamma$ is empty if $G$ is infinite and
	contractible if $G$ is finite. 
\end{proposition}

\begin{proof}
It follows from Lemma~\ref{lem: finite isotropy} that the fixed set of an infinite subgroup of $\Gamma$ is empty.
Assume that $G$ is a finite subgroup of $\Gamma$. We first claim that there is a positive vertex $\widehat{v}$ such that
the orbit $G \cdot \widehat{v}$ contains only positive vertices. For a vertex $v$, we let $\expa^{k}(v)$ denote the result of
applying the $\expa$ function (Definition~\ref{def: complete expansion}) to $v$ $k$ times. By Lemma~\ref{lem:positive expansion}, 
there is, for each vertex $v_{g} = \{ [g, X] \}$ ($g \in G$), a positive integer $n_{g}$ such that $\expa^{n_{g}}(v_{g})$ is positive.
Since $G$ is finite, it follows that there is $N \in \mathbb{N}$ such that $\expa^{N}(v_{g})$ is positive for all $g \in G$. This immediately
implies that the orbit $G \cdot \expa^{N}(v_{id_{X}})$ consists of positive vertices, proving the claim with 
$\widehat{v} =  \expa^{N}(v_{id_{X}})$.

The usual partial order $\leq$ on vertices has the property that any two positive vertices $v_{1}, v_{2}$ have a least upper bound; that is,
there is $\widetilde{v} \in K^{0}$ such that $\widetilde{v} \geq v_{1}, v_{2}$ and if $v' \in K^{0}$ is such that $v' \geq v_{1}, v_{2}$, then 
$v' \geq \widetilde{v}$. (In fact, if $v_{1} = \{ [ \incl_{B_{i}}, B_{i}] \mid 1 \leq i \leq m \}$ and 
$v_{2} = \{ [ \incl_{\widehat{B}_{j}}, \widehat{B}_{j}] \mid 1 \leq j \leq n \}$, then 
$\widetilde{v} = \{ [ \incl_{B_{i} \cap \widehat{B}_{j}}, B_{i} \cap \widehat{B}_{j}] \mid B_{i} \cap \widehat{B}_{j} \neq \emptyset, 1 \leq i \leq m, 1 \leq j
\leq n \}$ is the required vertex.) The least upper bound is necessarily unique.

It follows from an entirely straightforward argument that any finite collection of positive vertices has a (unique) least upper bound in $K^{0}$.
Now since $G \cdot \widehat{v}$ consists of positive vertices, $G$ must fix the least upper bound of $G \cdot \widehat{v}$ by the uniqueness of the
least upper bound. Therefore, the fixed set of $G$ is non-empty.  

We now show that the fixed set of $G$ is contractible. It is enough to show that the set of fixed vertices 
is directed. Note that if $v,w$ are vertices, $g \in \Gamma$,
$gv = v$, and $\expa(v) = w$, then $gw = w$. Thus, given vertices $v_{1}, v_{2}$ such that $g v_{i} = v_{i}$ ($i=1,2$, $g \in G$), we can use 
Lemma~\ref{lem:positive expansion} to find positive vertices
$v'_{1}, v'_{2}$ such that $gv'_{i} = v'_{i}$ and $v_i \leq v'_{i}$ ($i=1,2$, $g \in G$). We let $\widetilde{v}$ be the least upper bound of $\{ v'_{1}, v'_{2} \}$.
Since $v'_{1}$ and $v'_{2}$ are fixed by $G$, $\widetilde{v}$ must also be fixed by $G$ due to the uniqueness of the least upper bound. Thus
$v_{1}, v_{2} \leq \widetilde{v}$, and all three vertices are fixed by $G$, so the set of fixed vertices is directed.
\end{proof}
 


\section{Isomorphism classes of groups defined by finite similarity structures} \label{sec:isomorphismtype}

In this section, we will attempt to distinguish between FSS groups by analyzing the germ groups of their associated actions.
The main theoretical tool is Rubin's Theorem (Theorem \ref{thm:Rubin}).

\subsection*{Recollections on Rubin's theorem}
We recall Rubin's theorem \cite{Rubin1996, Rubin1989} as exposited by Brin \cite[Section 9]{BrinHDTG}.

\begin{definition}
If $X$ is a topological space, then a subgroup $F$ of the group of homeomorphisms of $X$ is \emph{locally dense}
if for every $x\in X$ and open neighborhood $U\subseteq X$ of $x$, the closure of
$$\{ f(x) ~|~ f\in F, f|_{(X\setminus U)} = \id_{(X\setminus U)}\}$$
contains a nonempty open set.
\end{definition}

\begin{theorem}[Rubin]
\label{thm:Rubin}
Let $X$ and $Y$ be locally compact, Hausdorff spaces without isolated points.
Let $F$ and $G$ be subgroups of the groups of homeomorphisms of $X$ and $Y$, respectively.
If $F$ and $G$ are locally dense and $\phi\co F\to G$ is an isomorphism, then
there exists a unique homeomorphism $h\co X\to Y$ such that $\phi(f) = hfh^{-1}$
for every $f\in F$.
\end{theorem}

\begin{definition} Let $X$ be a topological space, and let $\Gamma$ be a group acting on $X$.
For each $x\in X$, we let $G_x$ be the \emph{group of germs of the action
$\Gamma\curvearrowright X$ at $x$}.
That is, let $\Gamma_x\leq\Gamma$ be the isotropy subgroup consisting of all elements of $\Gamma$ that fix $x$.
Let $N_x\triangleleft \Gamma_x$ be the normal subgroup consisting of the elements $\gamma$ for which there
is an open neighborhood $U_{\gamma}$ of $x$ such that $\gamma|_{U_{\gamma}} = \id_{U_{\gamma}}$.
Then $G_x := \Gamma_x/N_x$.
\end{definition}

Brin \cite{BrinHDTG} used the first part of the following corollary to show that
$2V$ and $V$ are not isomorphic. Bleak and Lanoue \cite{BlL} used the second part
to show $nV$ is not isomorphic to $mV$ if $n\not= m$.

\begin{corollary} \label{corollary: applications from Rubin}
Assume the notation of Rubin's Theorem~\ref{thm:Rubin}. 
\begin{enumerate}
\item For every $f\in F$ and for every $x\in X$, $h$ induces a bijection from
the orbit $\{f^nx ~|~ n\in\bz\}$ to the orbit $\{ (\phi(f))^n(h(x)) ~|~ n\in \bz\}$.
	\item 
For every $x\in X$, $h$ induces an isomorphism from the group
of germs of the action $F\curvearrowright X$ at $x$ to the group of germs of the action 
$G \curvearrowright Y$ at $h(x)$.
		\item
For every $x\in X$, $h$ induces a bijection from the orbit $Fx$ to the orbit $Gh(x)$.
\end{enumerate}
\end{corollary}

\subsection*{The Case of the Nekrashevych-R\"{o}ver groups}
For the remainder of this section, we will consider the Nekrashevych-R\"{o}ver groups $V_{d}(H)$.
Thus, we choose a finite alphabet $A = \{ a_{1}, \ldots, a_{d} \}$ $(d \geq 2)$ and a subgroup
$H \leq \Sigma_{d}$, where $\Sigma_{d}$ is the group of permutations of $A$. These choices determine
a finite similarity structure (as specified in Definition~\ref{def:symmetric group sim}), and, therefore, 
a group $\Gamma$, which is isomorphic to $V_{d}(H)$ by Remark~\ref{claim: Vdh is Gamma}.

\begin{lemma} \label{lemma: Rubin's hypothesis is satisfied}
The space $A^{\omega}$ is locally compact, Hausdorff, and has no isolated points. The action of $\Gamma$
on $A^{\omega}$ is locally dense.
\end{lemma}

\begin{proof}
The space $A^{\omega}$ is compact metric, so it is locally compact and Hausdorff. It is straightforward to
check that $A^{\omega}$ has no isolated points using the description of balls in Remark~\ref{rem: characterization of balls}.

We turn to the proof that the action of $\Gamma$ is locally dense. As a first step, we show that the orbit of any $x \in A^{\omega}$
is dense. Let $x,y \in A^{\omega}$ and let $U$ be an open ball containing $y$. The ball $U$ must have the form $uA^{\omega}$, for some $u \in A^{\ast}$, by 
Remark~\ref{rem: characterization of balls}. Let $\mathcal{P} = \{ vA^{\omega} \mid |v| = |u| \}$. The set $\mathcal{P}$ is a partition
of $A^{\omega}$ into balls. We let $v_{x}A^{\omega} \in \mathcal{P}$ be the ball containing $x$. Now choose a bijection $\phi: \mathcal{P} \rightarrow \mathcal{P}$
such that $\phi(v_{x}A^{\omega}) = uA^{\omega}$. We define $\gamma \in \Gamma$ as follows. For each ball $vA^{\omega} \in \mathcal{P}$, we choose
an arbitrary $h_{v} \in \sm (vA^{\omega}, \phi(vA^{\omega}))$, and let $\gamma|_{vA^{\omega}} = h_{v}$. These choices determine a unique $\gamma \in \Gamma$,
and $\gamma(x) \in U$. It follows that the orbit of $x$ is dense.

Now we can prove that the action of $\Gamma$ is locally dense. Let $x \in A^{\omega}$, and let $U$ be an open neighborhood of $x$. We can find an open
ball $uA^{\omega} \subseteq U$ such that $x \in uA^{\omega}$. Let $\mathcal{P} = \{ vA^{\omega} \mid |v| = |u| \}$. For each $\gamma \in \Gamma$, we define
an element $\gamma_{u} \in \Gamma$ as follows. If $v \neq u$, then 
$\gamma_{u}|_{vA^{\omega}} = \id_{vA^{\omega}}$. Let $h_{u} \in \sm(uA^{\omega}, A^{\omega})$ denote the unique order-preserving similarity
in $\sm(uA^{\omega}, A^{\omega})$. We set $\gamma_{u}|_{uA^{\omega}} = h^{-1}_{u} \gamma h_{u}$. The fact that $\gamma_{u} \in \Gamma$ follows from the
fact that $\sm$ is closed under compositions and restrictions. Since the action of $\Gamma$ has dense orbits in $A^{\omega}$, the set
$$ \{ \gamma_{u}(x) \mid \gamma \in \Gamma \}$$
is dense in $uA^{\omega}$. Since $\gamma_{u}|_{A^{\omega} \backslash uA^{\omega}} = \id|_{A^{\omega} \backslash uA^{\omega}}$,
the local denseness of the action follows.
\end{proof}   

\begin{definition}\label{def: local representative}
Let $\gamma \in \Gamma_{x}$. There exists a pair of balls $uA^{\omega}, vA^{\omega}$ such that $x \in uA^{\omega} \cap vA^{\omega}$ and
$\gamma|_{uA^{\omega}} = h$, for some $h \in \sm(uA^{\omega}, vA^{\omega})$. We say that $h$ is a \emph{local representative for $\gamma$ at
$x$}, or that $h$ \emph{locally represents $\gamma$ at $x$}.
\end{definition}

\begin{proposition}\label{prop: each local representative represents something}
If $x \in uA^{\omega} \cap vA^{\omega}$, $u,v \neq 1$, and $h \in \sm(uA^{\omega}, vA^{\omega})$ fixes $x$, then
$h$ locally represents some $\gamma \in \Gamma$ at $x$. For any $\gamma \in \Gamma_{x}$, the collection $\{ h \mid  h \text{ locally represents }
\gamma \}$ is closed under restriction to open ball neighborhoods of $x$, and, given two local representatives of $\gamma$, one is the 
restriction of the other.
\end{proposition}

\begin{proof}
We prove the first statement. Thus, suppose $h \in \sm(uA^{\omega}, vA^{\omega})$ fixes $x$, and $u,v \neq 1$. We let $\mathcal{P}_{1}$, $\mathcal{P}_{2}$
be partitions of $A^{\omega}$ into balls, such that $uA^{\omega} \in \mathcal{P}_{1}$ and $vA^{\omega} \in \mathcal{P}_{2}$. In general, whenever $\mathcal{P}$
is a partition of $A^{\omega}$ into balls, then one can obtain another such partition $\mathcal{P}'$ by replacing $\widehat{u}A^{\omega} \in \mathcal{P}$
with $\widehat{u}a_{1}A^{\omega}, \ldots, \widehat{u}a_{d}A^{\omega}$; i.e., 
$$ \mathcal{P}' = \left( \mathcal{P} \backslash \{ \widehat{u}A^{\omega} \} \right) \cup \{ \widehat{u}a_{j}A^{\omega} \mid 1 \leq j \leq d \}.$$
Moreover, any partition of $A^{\omega}$ into balls arises from a repeated application of this procedure to the partition $\{ A^{\omega} \}$. It follows
that $|\mathcal{P}_{1}| = 1 + k(d-1)$ and $|\mathcal{P}_{2}| = 1 + \ell (d-1)$, for some $k, \ell \in \mathbb{N}$ ($k, \ell > 0$). Since, in particular,
$|\mathcal{P}_{i}| > 1$ for $i=1,2$, there exist $\widetilde{u}A^{\omega} \in \mathcal{P}_{1}$, $\widetilde{v}A^{\omega} \in \mathcal{P}_{2}$, distinct
from $uA^{\omega}$ and $vA^{\omega}$ (respectively). Now we apply the replacement procedure repeatedly, to $\mathcal{P}_{1}$ $\ell$ times and to $\mathcal{P}_{2}$
$k$ times, always replacing balls other than $uA^{\omega}$ (in the first case) and $vA^{\omega}$ (in the second case). The result is a pair of partitions
$\mathcal{P}'_{1}$, $\mathcal{P}'_{2}$ such that $|\mathcal{P}'_{1}| = |\mathcal{P}'_{2}| = 1 + (k + \ell)(d-1)$, $uA^{\omega} \in \mathcal{P}'_{1}$, 
and $vA^{\omega} \in \mathcal{P}'_{2}$. We choose a bijection $\psi: \mathcal{P}'_{1} \rightarrow \mathcal{P}'_{2}$ such that $\psi(uA^{\omega}) = vA^{\omega}$.
For each $tA^{\omega} \in \mathcal{P}'_{1} \backslash \{ uA^{\omega} \}$, we choose an arbitrary $h_{t} \in \sm(tA^{\omega}, \psi(tA^{\omega}))$. 
We let $\gamma$ be defined by the rule $\gamma|_{tA^{\omega}} = h_{t}$ for $tA^{\omega} \in \mathcal{P}'_{1}$, $t \neq u$, and 
$\gamma|_{uA^{\omega}} = h$. The result is an element of $\Gamma$, and $h$ locally represents $\gamma$.

The closure of $\{ h \mid h \text{ locally represents } \gamma \text{ at } x \}$ under restriction to open ball neighborhoods of $x$ is clear.

The final statement follows easily from the fact that the collection of balls containing $x$ is nested.
\end{proof}

\begin{definition}
Let $x \in A^{\omega}$. Consider $\mathcal{G}_{x} = \{ h \in \sm(uA^{\omega}, vA^{\omega}) \mid uA^{\omega}, vA^{\omega} \subseteq A^{\omega},
x \in uA^{\omega} \cap vA^{\omega}, h(x)=x \}$. 
If $h_{1}, h_{2} \in \mathcal{G}_{x}$, we write $h_{1} \sim h_{2}$ when there is a ball neighborhood of $x$ such that
$h_{1}|_{B} = h_{2}|_{B}$. The set $\mathcal{G}_{x} / \sim$ is called the \emph{abstract germ group at $x$}. The group operation is as follows.
If $[h_{1}], [h_{2}] \in \mathcal{G}_{x} / \sim$, then we choose representatives $h'_{1} \in [h_{1}]$, $h'_{2} \in [h_{2}]$ such that 
$h'_{1} \in \sm(B_{1}, B_{2})$ and $h'_{2} \in \sm(B_{2}, B_{3})$, where $x \in B_{1} \cap B_{2} \cap B_{3}$. We set 
$[h_{2}][h_{1}] = [h'_{2} \circ h'_{1}]$.
\end{definition}

\begin{proposition}
The abstract germ group at $x$, $\mathcal{G}_{x}/\sim$, is isomorphic to the germ group at $x$. The
map $\psi: G_{x} \rightarrow \mathcal{G}_{x}/ \sim$ defined by $\psi( [ \gamma ]) = [h_{\gamma}]$, where
$h_{\gamma}$ is a local representative of $\gamma$ at $x$, is an isomorphism.
\end{proposition}

\begin{proof}
We prove the second statement. It is straightforward to check that $\psi$ is well-defined and injective. Surjectivity follows
directly from Proposition~\ref{prop: each local representative represents something}.

Let $[\gamma_{1}], [\gamma_{2}] \in G_{x}$. We can choose local representatives $h_{1} \in \sm(B_{1}, \widehat{B}_{1})$,
$h_{2} \in \sm(B_{2}, \widehat{B}_{2})$ for $\gamma_{1}, \gamma_{2}$ (respectively). We consider the restrictions
$h'_{2} = h_{2}|_{h_{2}^{-1}(\widehat{B}_{2} \cap B_{1})}$ and $h'_{1} = h_{1}|_{B_{1} \cap \widehat{B}_{2}}$.
We have $h'_{1}, h'_{2} \in \mathcal{G}_{x}$ by the Restrictions property of $\sm$. By the definition of the operation in the abstract germ group,
$$ \psi([\gamma_{1}]) \psi([\gamma_{2}]) = [h'_{1}][h'_{2}] = [h'_{1} \circ h'_{2}].$$
On the other hand, $h'_{1}$ and $h'_{2}$ are still local representatives for $\gamma_{1}$ and $\gamma_{2}$ (respectively). It follows that $h'_{1} \circ h'_{2}$
is a local representative for $\gamma_{1} \circ \gamma_{2}$. That is, $\psi([\gamma_{1} \circ \gamma_{2}]) = [h'_{1} \circ h'_{2}]$,
so $\psi$ is a homomorphism.
\end{proof}

\begin{remark}
If $h \in \sm(uA^{\omega}, vA^{\omega}) \cap \mathcal{G}_{x}$, then $u$ and $v$ must satisfy a certain constraint. Let $x= x_{1}x_{2}\ldots x_{i}\ldots$. 
For each $n \in \mathbb{N}$, let $w_{n} = x_{1}x_{2}\ldots x_{n}$. Since $x \in uA^{\omega}$, we must have $u = w_{n}$, for some $n$. Similarly for $v$.
\end{remark}

\begin{remark} \label{rem: restrictions}
Let $h \in H \leq \Sigma_{d}$. Suppose that $\sigma_{h} \in \sm( w_{m}A^{\omega}, w_{n}A^{\omega})$, 
defined by $\sigma_{h} (w_{m}a_{i_{1}}a_{i_{2}}\ldots) = w_{n}h(a_{i_{1}}) h(a_{i_{2}}) \ldots$, satisfies 
$\sigma_{h}(x) = x$. Each restriction of $\sigma_{h}$ to a ball containing $x$ has the form
$\sigma'_{h} \in \sm(w_{m+j}A^{\omega}, w_{n+j}A^{\omega})$, where 
$\sigma'_{h}(w_{m+j}a_{i_{1}}a_{i_{2}}\ldots) = w_{n+j}h(a_{i_{1}})h(a_{i_{2}}) \ldots$, for some $j \in \mathbb{N}$.
\end{remark}

\subsection*{General analysis of the germ groups}
Lemma~\ref{lemma: Rubin's hypothesis is satisfied} implies that we can use Corollary~\ref{corollary: applications from Rubin}
in order to distinguish between isomorphism types of the groups $V_{d}(H)$, for varying $d>1$ and $H \leq \Sigma_{d}$. In this
subsection, we describe how to determine the isomorphism type of the germ group $G_{x}$, for arbitrary $d>1$, $H \leq \Sigma_{d}$,
and $x \in A^{\omega}$.

\begin{definition}
For each $a\in A$, the \emph{isotropy subgroup of $H$ fixing $a$} is defined by
$$H_a =\{ \sigma\in H ~|~ \sigma(a)=a\}.$$
\end{definition}

\begin{definition}
For each $x=x_1x_2x_3\dots\in A^\omega$, where each $x_i\in A$, the \emph{eventual isotropy subgroup of $H$ at $x$}
is the subgroup of $H$ defined by 
$$H_x := \bigcap\{ H_a ~|~ a=x_i \text{ for infinitely many $i$}\}.$$
\end{definition}

\begin{remark} \label{rem: eventual isotropy group}
Since $A$ is finite, it follows that for each $x=x_1x_2x_3\dots\in A^\omega$, where each $x_i\in A$, 
there exists $N\in\bn$ such that for all $k\geq N$,
the eventual isotropy subgroup of $H$ at $x$
is  
$$H_x = \bigcap\{ H_{x_i} ~|~ i\geq k \}.$$
\end{remark}

\begin{example} If every element of $A$ appears infinitely often in $x\in A^\omega$, then the eventual isotropy group of $H$
at $x$ is $H_x=\{ 1\}$.
\end{example}

\begin{definition} An element $x\in A^\omega$ is \emph{eventually periodic} if there exists $u,v\in A^*$ such that
$x=u\bar{v}$. 
If $x$ is eventually periodic, the \emph{period of $x$} is 
$$\pi(x) := \min\{ \vert v\vert ~|~ x=u\bar{v}, \text{where $u,v\in A^\ast$}\}\in\bn.$$
\end{definition}

\begin{example}
\label{example:eventual} 
If $x\in A^\omega$ is eventually periodic with  $x=u\bar{v}$, where  $u\in A^*$
and $v=y_1y_2\cdots y_n\in A^n$,
then the eventual isotropy group of $H$
at $x$ is 
$$H_x = \bigcap\{ H_{y_i} ~|~ 1\leq i\leq n \}.$$
\end{example}

\begin{theorem} \label{theorem: germ description}
The function $\phi: \mathcal{G}_{x}/\sim \, \rightarrow \mathbb{Z}$ sending
$\sigma_{h} \in \sm(w_{m}A^{\omega}, w_{n}A^{\omega})$ to $n-m$ is
a homomorphism. The image is non-trivial if and only if $x$ is eventually periodic.
The kernel is naturally isomorphic to the eventual isotropy group of $H$ at $x$.

Moreover, assuming that $x$ is eventually periodic, there is a constructive procedure
for determining the image of $\phi$.
\end{theorem}

\begin{proof}
The function $\phi$ is well-defined on equivalence classes by Remark~\ref{rem: restrictions}.
We check the homomorphism condition. Let $\sigma_{1} \in \sm(w_{m}A^{\omega}, w_{n}A^{\omega}) \cap \mathcal{G}_{x}$,
$\sigma_{2} \in \sm(w_{n}A^{\omega}, w_{p}A^{\omega}) \cap \mathcal{G}_{x}$. We have
$$ \phi(\sigma_{2}) + \phi(\sigma_{1}) = (p-n) + (n-m) = p-m = \phi(\sigma_{2} \circ \sigma_{1}),$$
since $\sigma_{2} \circ \sigma_{1} \in \sm(w_{m}A^{\omega}, w_{p}A^{\omega})$. Thus, $\phi$ is a homomorphism.

Suppose that $\image(\phi)$ is
non-trivial. Thus, there is $\sigma_{h} \in \sm(w_{m}A^{\omega}, w_{n}A^{\omega}) \cap \mathcal{G}_{x}$,
where $m \neq n$ and $h \in H$. (Here $\sigma_{h}(w_{m}a_{i_{1}}a_{i_{2}}\ldots) = w_{n}h(a_{i_{1}})h(a_{i_{2}})\ldots$.)
We assume, without loss of generality, that $m<n$. Since $\sigma_{h}(x) = x$, we have
\begin{align*}
x_{1}\ldots x_{m} x_{m+1} \ldots &= \sigma_{h}(w_{m}x_{m+1} x_{m+2} \ldots) \\
 &= x_{1} \ldots x_{n} h(x_{m+1}) h(x_{m+2}) \ldots .
\end{align*}
It follows that $h(x_{m+j}) = x_{n+j}$, for $j \in \mathbb{N}$. Since $h$ must have finite order, say $|h| = k$, we have
$$x_{m+j} = h^{k}(x_{m+j}) = x_{k(n-m)+m+j}$$
for $j \in \mathbb{N}$. It follows that $x = x_{1}\ldots x_{m} \overline{x_{m+1} \ldots x_{k(n-m)+m}}$,
so $x$ is eventually periodic.

Conversely, if $x$ is eventually periodic, we have $x = u\overline{v}$. Let $\sigma \in \sm(uA^{\omega}, uvA^{\omega})$ be
defined by 
$\sigma (ua_{i_{1}}a_{i_{2}}\ldots) = uva_{i_{1}}a_{i_{2}}\ldots$.
Clearly $\sigma \in \mathcal{G}_{x}$ and $\phi(\sigma) = |v| \neq 0$, so the image is nontrivial.

By Remark~\ref{rem: eventual isotropy group}, there is some $N \in \mathbb{N}$ such that 
$H_{x} = \bigcap\{H_{x_{i}} \mid i > N \}$. We claim that  $\psi: H_{x} \rightarrow \sm(w_{N}A^{\omega}, w_{N}A^{\omega}) \cap \mathcal{G}_{x}$
is an isomorphism of groups, where $\psi(h)(w_{N}a_{i_{1}}a_{i_{2}}\ldots) = w_{N}h(a_{i_{1}})h(a_{i_{2}})\ldots$.
It is  straightforward to check that $\psi$ is a well-defined homomorphism. If $\psi(h) = 1$, then $\psi(h)(w_{N}a_{i}a_{i}\ldots)
= w_{N}a_{i}a_{i}\ldots$ for $i \in \{1, \ldots, d \}$. Thus, $h=1$ and $\psi$ is injective. Let 
$\sigma_{h} \in \sm(w_{N}A^{\omega}, w_{N}A^{\omega}) \cap \mathcal{G}_{x}$, where $\sigma_{h}(w_{N}a_{i_{1}}a_{i_{2}}\ldots)
= w_{N}h(a_{i_{1}})h(a_{i_{2}})\ldots$. We must have $h(x_{i}) = x_{i}$ for $i > N$ (since $\sigma_{h}(x) = x$), so $h \in H_{x}$. Thus, $\psi$
is surjective. This proves the claim.

We further claim that the projection $p: \sm(w_{N}A^{\omega}, w_{N}A^{\omega}) \cap \mathcal{G}_{x} \rightarrow \mathcal{G}_{x} / \sim$
is an isomorphism onto its image; that is, $p$ is injective. Thus, suppose $\sigma \in \sm(w_{N}A^{\omega}, w_{N}A^{\omega}) \cap \mathcal{G}_{x}$,
and $p(\sigma) = 1$. There is $h \in H$ such that $\sigma(w_{N}a_{i_{1}}a_{i_{2}}\ldots) = w_{N}h(a_{i_{1}})h(a_{i_{2}})\ldots$. The statement
$p(\sigma) = 1$ means that some restriction of $\sigma$, say $\sigma_{1} \in \sm(w_{N+j}A^{\omega}, w_{N+j}A^{\omega}) \cap \mathcal{G}_{x}$,
is the identity. 
Since $\sigma_{1} (w_{N+j}a_{i_{1}}a_{i_{2}}\ldots) = w_{N+j}h(a_{i_{1}})h(a_{i_{2}})\ldots$ for all possible choices of the $a_{i_{j}}$, we must
have $h=1$. Thus, $\sigma = 1$, so $p$ is injective.

It follows that $p \circ \psi$ is an isomorphism onto its image. We now claim that this image is precisely $\text{Ker}\phi$. Indeed, it
is already clear that $(p \circ \psi)(H_{x}) \subseteq \text{Ker}\phi$. If $[\sigma] \in \text{Ker}\phi$, then we have 
$\sigma \in \sm(w_{M}A^{\omega}, w_{M}A^{\omega}) \cap \mathcal{G}_{x}$ for some $M \in \mathbb{N}$. There is $h \in H$ such
that $\sigma(w_{M}a_{i_{1}}a_{i_{2}} \ldots) = w_{M}h(a_{i_{1}})h(a_{i_{2}})\ldots$. We pick $P > \text{max}\{ M,N \}$ and let 
$\sigma' \in \sm(w_{P}A^{\omega}, w_{P}A^{\omega}) \cap \mathcal{G}_{x}$ be the restriction of $\sigma$ to $w_{P}A^{\omega}$. 
It is now clear that 
$\sigma'$ is the restriction of $\psi(h) \in \sm(w_{N}A^{\omega}, w_{N}A^{\omega}) \cap \mathcal{G}_{x}$, so 
$[\sigma'] = [\psi(h)] = (p \circ \psi)(h)$, and so $\text{Ker} \phi \subseteq (p \circ \psi)(H_{x})$.

Finally, we determine effectively whether a given $n \in \mathbb{Z}$ lies in the image of $\phi$, assuming that
$x$ is eventually periodic. Let $x = u\overline{v}$. We already know that $|v| \in \image(\phi)$, so we need only consider
$k \in \{ 1, \ldots, |v|-1 \}$. Let $v=v_{1}v_{2}\ldots v_{\ell}$ (so $|v| = \ell$). The integer $k \in \image(\phi)$ if and only if
$\sm(uA^{\omega}, uv_{1}\ldots v_{k} A^{\omega}) \cap \mathcal{G}_{x} \neq \emptyset$. There will be 
$\sigma \in \sm(uA^{\omega}, uv_{1}\ldots v_{k} A^{\omega}) \cap \mathcal{G}_{x}$ if and only if there is $h \in H$ such that
$h(v_{n}) = v_{k+n}$, for all $n \in \{ 1, \ldots, \ell \}$, where the subscript $k+n$ is interpreted modulo $\ell$. 
This is a finite list of conditions, and
we need only check each of the finitely many elements of $H$ against them. The claim follows. 
\end{proof}

\begin{corollary} \label{corollary: summary of germ groups}
Let $x \in A^{\omega}$. The germ group at $x$ is isomorphic to $H_{x}$ if $x$ is not eventually periodic. If $x$ is eventually
periodic, then $G_{x} \cong H_{x} \rtimes \mathbb{Z}$, where the action of $\mathbb{Z}$ on $H_{x}$ can be determined constructively.
\end{corollary}

\begin{proof}
The entire statement of the Corollary is an immediate consequence of Theorem~\ref{theorem: germ description}. We describe 
the action of $\mathbb{Z}$ on $H_{x}$ in more detail.

Assume $x$ is eventually periodic. Let $[\sigma] \in \mathcal{G}_{x}/\sim$ be such that $\phi([\sigma]) = \ell$ is the smallest positive 
number in $\image \phi$. Clearly, $\phi([\sigma])$ generates $\image \phi$. It follows from elementary group theory that 
$G_{x} \cong H_{x} \rtimes \langle [\sigma] \rangle$. (Here $H_{x}$ is identified with the image of $p \circ \psi$ as described in
the proof of Theorem~\ref{theorem: germ description}.) We assume without loss of generality that 
$\sigma \in \sm(w_{N}A^{\omega}, w_{N+\ell}A^{\omega}) \cap \mathcal{G}_{x}$, where $N$ is as in Remark~\ref{rem: eventual isotropy group}.
By the definition of $\sm$, $\sigma$ is defined by the rule $\sigma(w_{N}a_{i_{1}}a_{i_{2}}\ldots) = w_{N+\ell}h(a_{i_{1}})h(a_{i_{2}})\ldots$, 
for some $h \in H$. An element $[\widehat{\sigma}] \in (p \circ \psi)(H_{x}) \cong H_{x}$
can be represented by $\widehat{\sigma} \in \sm(w_{N}A^{\omega}, w_{N}A^{\omega}) \cap \mathcal{G}_{x}$
such that $\widehat{\sigma}(w_{N}a_{i_{1}}a_{i_{2}}\ldots) = w_{N}\widehat{h}(a_{i_{1}})\widehat{h}(a_{i_{2}})\ldots$, where
$\widehat{h} \in H_{x}$. Now
$\sigma \circ \widehat{\sigma} \circ \sigma^{-1} \in \sm(w_{N+\ell}A^{\omega}, w_{N+\ell}A^{\omega}) \cap \mathcal{G}_{x}$
is defined by $(\sigma \circ \widehat{\sigma} \circ \sigma^{-1})(w_{N+\ell}a_{i_{1}}a_{i_{2}}\ldots) = 
w_{N+\ell}h\widehat{h}h^{-1}(a_{i_{1}})h\widehat{h}h^{-1}(a_{i_{2}})\ldots$. It follows that,
modulo the identification of $H_{x}$ with $(p \circ \psi)(H_{x})$, $[\sigma][\widehat{\sigma}][\sigma]^{-1} = h\widehat{h}h^{-1}$. 
\end{proof}

\subsection*{Examples of germ groups and nonisomorphism results}
We now use the results of the previous subsection to compute the germ groups in some examples and
to show that certain pairs of groups $V_{d'}(H')$, $V_{d}(H)$ are non-isomorphic.
We also give examples that demonstrate the limitations of our methods.

\begin{example}\label{example: compute the germs}
Let $A = \{ 1, 2, 3, 4 \}$, and let $H = S_{4}$. We compute a few of the germ groups $G_{x}$.

Let $x = \overline{12}$. The eventual isotropy group $H_{x}$ is $H_{1} \cap H_{2} = \langle (34) \rangle$.
We note that $v = 12$, so $v_{1} = 1$ and $v_{2} = 2$ (in the notation of the proof of Theorem~\ref{theorem: germ description}).
To determine whether $1 \in \image \phi$, we need to determine whether there is $h \in S_{4}$ such that $h(v_{n}) = v_{n+1}$, where 
$n \in \{ 1, 2 \}$ and the subscript $n+1$ is interpreted modulo $2$. Clearly, $h = (12)$ satisfies the given requirements. It follows
from Corollary~\ref{corollary: summary of germ groups} that $G_{x} \cong \langle (34) \rangle \oplus \mathbb{Z}$, since the action
of $h = (12)$ on $(34)$ by conjugation is trivial.

Let $x = \overline{122}$. We again have $H_{x} = \langle (34) \rangle$. If $1$ were in $\image \phi$, then there would be $h \in S_{4}$
such that $h(1) = 2$, $h(2) = 2$, and $h(2) = 1$ (since $1$, $2$, and $2$ are, respectively, $v_{1}$, $v_{2}$ and $v_{3}$). This is clearly
impossible. If $2$ were in $\image \phi$, we would similarly have $h(1) = 2$, $h(2) = 1$, and $h(2) = 2$ (respectively) for some $h \in S_{4}$.
This is again impossible. It is clear that $3 \in \image \phi$ since the required conditions are satisfied with $h=1$. It follows
that $G_{x} \cong \langle (34) \rangle \oplus \mathbb{Z}$.

Let $x = \overline{123}$. We have $H_{x} = H_{1} \cap H_{2} \cap H_{3} = \langle (1) \rangle$. To determine if $1 \in \image \phi$, we must
determine whether there exists $h \in S_{4}$ such that $h(1) = 2$, $h(2) = 3$, and $h(3) = 1$. Clearly, we can take $h = (123)$. The element
$\sigma_{h} \in \sm(123A^{\omega}, 1231A^{\omega})$ represents a generator of $G_{x}$, which is isomorphic to $\mathbb{Z}$.
\end{example}

\begin{remark}\label{rem: germ groups are sometimes direct products}
Let $A = \{ 1, \ldots, d \}$.
If $H = \Sigma_{d}$, then the germ group $G_{x}$ is isomorphic to $H_{x} \oplus \mathbb{Z}$ if $x$ is eventually periodic (or
to $H_{x}$ if not). This is because the conditions on $h \in \Sigma_{d}$ are either inconsistent, or they can be satisfied using only symbols
from the string $v$. (Here we assume $x = u\overline{v}$, and $h \in \Sigma_{d}$ is as in Example~\ref{example: compute the germs}.)
Since  each $\widehat{h} \in H_{x}$ fixes all of the symbols from $v$ by definition, $\widehat{h}$ and $h$ commute, so the action
of $h$ by conjugation is trivial.
\end{remark}

\begin{example}\label{example: semidirect product germ}
Consider $A = \{ 1, 2, 3, 4, 5 \}$ and let $H = A_{5}$ (the group of even permutations of $A$). We compute the
germ group $G_{x}$ for $x = \overline{12}$. First, we note that $H_{x} = H_{1} \cap H_{2} = \langle (345) \rangle$.
Next, we determine whether $1 \in \image \phi$. Thus we must determine whether there is $h \in A_{5}$ such that $h(1) = 2$ and
$h(2) = 1$. Clearly, we can let $h = (12)(34)$. (There are other possibilities, but we cannot let $h = (12)$, since $(12) \notin A_{5}$.)
Thus $1 \in \image \phi$. 

We conclude that $G_{x} \cong \langle (345) \rangle \rtimes \mathbb{Z}$, where the action of $\mathbb{Z}$ is conjugation by $(12)(34)$. This
example shows how a non-trivial action by $\mathbb{Z}$ can arise in a germ group $G_{x}$. 
\end{example}

We can now offer a sample application of the ideas in this section. Many other statements are possible.

\begin{proposition}\label{prop: nonisomorphism}
If $d \neq d'$, $d, d' > 1$, then $V_{d}(\Sigma_{d})$ is not isomorphic to $V_{d'}(\Sigma_{d'})$. 
\end{proposition}

\begin{proof}
If $x \in A^{\omega}$, where $A = \{ 1, \ldots, d \}$, then the germ group $G_{x}$ (for $\Gamma = V_{d}(\Sigma_{d})$)
must be $H_{x}$ or $H_{x} \oplus \mathbb{Z}$, by Corollary~\ref{rem: eventual isotropy group} and 
Remark~\ref{rem: germ groups are sometimes direct products}. The only possibilities for $H_{x}$, up to isomorphism, are
$\Sigma_{0}, \Sigma_{1}, \ldots, \Sigma_{d-1}$ (respectively), according to  
the number, $d, d-1, \ldots, 2$ or $1$, of symbols  occuring infinitely
often in $x$. 
(Here both $\Sigma_0$ and $\Sigma_1$ are the group with one element.)
The given statement now follows immediately from Corollary~\ref{corollary: applications from Rubin}.
\end{proof}

\begin{example}
Suppose that $H \leq \Sigma_{d}$ acts freely on $A = \{ 1, \ldots, d \}$. 
Since $H_{x}$ can be described as an intersection of stabilizer subgroups of symbols in $A$, we must have $H_{x} = 1$. Thus, by 
Corollary~\ref{corollary: summary of germ groups}, $G_{x}$ is either $\mathbb{Z}$ or $1$, according to whether $x$ is eventually periodic or not. 

Thus, for instance, if $H = 1$ or if $d$ is prime and $H \leq \Sigma_{d}$ is cyclic of order $d$, then the germ groups are either $\mathbb{Z}$
or $1$. We are therefore unable to distinguish such groups from each other using only Corollary~\ref{corollary: applications from Rubin}(2), 
although one expects many differences in isomorphism type.
\end{example}

\section{Simplicity of some FSS groups} \label{sec:simplicity}

In \cite{NekJOT}, Nekrashevych defined the group $V_{d}'(H)$, which is a subgroup of $V_{d}(H)$. Here we will
show that $V_{d}'(H)$ has a simple commutator subgroup and describe this commutator subgroup explicitly.
We fix $d>1$ and $H \leq \Sigma_{d}$, the group of permutations of $\{ 1, \ldots, d \}$, for the remainder of the section.
The alternating subgroup of $\Sigma_d$ is denoted $A_d$. When $H$ is the group with one element, we use the notation
$V_d = V_d(H)$. 

\subsection*{The Abelianization of $V'_{d}(H)$}

\begin{definition} \cite{NekJOT} \label{def: tables}
Let $g \in V_{d}(H)$. There are partitions $\mathcal{P}_{1} = \{ v_{1}A^{\omega}, \ldots, v_{m}A^{\omega} \}$,
$\mathcal{P}_{2} = \{ u_{1}A^{\omega}, u_{2}A^{\omega}, \ldots, u_{m}A^{\omega} \}$, elements $h_{1}, \ldots, h_{m} \in H$,
and a permutation $\sigma \in \Sigma_{m}$ such that, for each $v_{i}A^{\omega} \in \mathcal{P}_{1}$, 
$$g(v_{i}x_{1}x_{2} \ldots) = u_{\sigma(i)}h_{i}(x_{1})h_{i}(x_{2})h_{i}(x_{3})\ldots.$$
All of this information can be summarized in a $3 \times m$ matrix, called a \emph{table} for $g$.
The first row of the table is $(v_{1}, \ldots, v_{m})$, the second is 
$(h_{1}, \ldots, h_{m})$, and the third is 
$(u_{\sigma(1)}, \ldots, u_{\sigma(m)})$. If $g \in V_{d}$, then all of the $h_{i}$ are the identity, so we omit the middle row.

For the moment, let us assume, without loss of generality, that $u_{1} \leq u_{2} \leq \ldots \leq u_{m}$
and $v_{1} \leq v_{2} \leq \ldots \leq v_{m}$ in the lexicographic ordering. We say that the table for the above element $g$ is 
\emph{even} or \emph{odd} when $\sigma$ is an even or odd permutation (respectively).

We will often try to avoid  writing down a $3 \times m$ matrix when we describe an element $g \in V_{d}(H)$.
Instead we write $g = \Sigma_{i=1}^{m} S_{u_{\sigma(i)}}h_{i}S^{\ast}_{v_{i}}$ in place of the table described above.
\end{definition}

(Note that we do not assume, in general, that the $u_{i}$ ($i = 1, \ldots, m$) satisfy $u_{1} \leq u_{2} \leq \ldots \leq u_{m}$.) 

\begin{remark} In \cite{NekJOT},  $\Sigma_{i=1}^{m} S_{u_{\sigma(i)}}h_{i}S^{\ast}_{v_{i}}$ is an element of a certain Cuntz-Pimsner algebra.
We will not use this interpretation in what follows, and refer the interested reader to \cite{NekJOT} for details.
\end{remark}

\begin{remark} Let $h \in H \leq \Sigma_{d}$. Such an $h$ acts on $A^{\omega}$ by the rule $h \cdot x_{1}x_{2} \ldots = h(x_{1})h(x_{2})\ldots$.
We may thus identify $H$ with the indicated subgroup of $V_{d}(H)$ and we do this freely in what follows.
\end{remark}

\begin{definition} \cite{NekJOT}
Every element $g \in V_{d}(H)$ can be represented by infinitely many different tables. If the column $(v_{i}, h_{i}, u_{i})$ appears in a table for
$g$, then we obtain another table for $g$ by striking out that column and replacing it with the $d$ columns $(v_{i}j, h_{i}, u_{i}h_{i}(j))$
for $1 \leq j \leq d$. We say that the latter table is the result of \emph{splitting} the former at the given column.

For odd $d$ the parity of a table for $g \in V_{d}$ is preserved, so there is a subset $V_{d}'$
of even elements. The set $V_{d}'$ is in fact a subgroup of index two in $V_{d}$.  
In the groups $V_{d}$, for even $d$, splitting a table changes its parity, and we simply set $V_{d}' = V_{d}$.
We note that, in either case, $V'_{d}$ is a simple nonabelian group \cite{NekJOT}.

More generally, we set $V_{d}'(H) = V_{d}(H)$ if $d$ is even or if $H$ contains an odd permutation. If $d$ is odd and $H$ contains only
even permutations, then the parity of a table for $g \in V_{d}(H)$ is preserved under splitting and the set $V_{d}'(H)$ consisting of the even elements
of $V_{d}(H)$ is a subgroup of index $2$ in $V_{d}(H)$. 
\end{definition}

In certain cases, we will need to track the change in the parity of a table after splitting when $d$ is odd. The proof of the following lemma
is routine.

\begin{lemma} \label{lemma: oddeventables}
Suppose $d$ is odd.
\begin{enumerate}
\item Let $h \in A_{d}$.  If $h$ can be expressed in the form $\displaystyle h = \sum_{i=1}^{m} S_{u_{i}}hS_{v_{i}}^{\ast}$, then the table
$$ \left( \begin{matrix} v_{1} & \ldots & v_{m} \\ u_{1} & \ldots & u_{m} \end{matrix} \right)$$
is even.
\item Let $h \in \Sigma_{d} \backslash A_{d}$.  If $\displaystyle h = \sum_{i=1}^{m} S_{u_{i}}hS_{v_{i}}^{\ast}$, then $m = 1 + (d-1)k$ for
some $k \in \mathbb{N} \cup \{ 0 \}$.  The table
$$ \left( \begin{matrix} v_{1} & \ldots & v_{m} \\ u_{1} & \ldots & u_{m} \end{matrix} \right)$$
is even if and only if $k$ is even.
\end{enumerate} \qed
\end{lemma}

\begin{definition}
Let $r \in A^{\ast}$. Let $g \in V_{d}(H)$. 
We define $\Lambda_{r}(g): A^{\omega} \rightarrow A^{\omega}$ by the rule
$\Lambda_{r}(g)(x) = x$ if $x \notin rA^{\omega}$ and $\Lambda_{r}(g)(rx_{1}x_{2}\ldots) = rg(x_{1}x_{2}\ldots)$ otherwise.
\end{definition}

It is clear that $\Lambda_{r}(g) \in V_{d}(H)$. If $g \in V_{d}'(H)$ then $\Lambda_{r}(g) \in V_{d}'(H)$ as well.
The maps $\Lambda_{r}: V_{d}(H) \rightarrow V_{d}(H)$ and $\Lambda_{r}: V'_{d}(H) \rightarrow V'_{d}(H)$ are injective homomorphisms.

\begin{lemma} \label{lemma: conjugate}
Given $r,s \in A^{+}$, there is some $v \in V_{d}'$ such that
$v(rA^{\omega}) = sA^{\omega}$ and $v(rx_{1}x_{2}\ldots) = sx_{1}x_{2}\ldots$ for all $x_{1}x_{2}\ldots \in A^{\omega}$.
In particular, for any $h \in H$ and $r, s \in A^{+}$, the elements $\Lambda_{r}(h)$ and $\Lambda_{s}(h)$
are conjugate in $V_{d}'(H)$ (and, therefore, in $V_{d}(H)$).
\end{lemma}

\begin{proof}
The first statement is straightforward to check and the second statement is a simple consequence of the first.
\end{proof}

\begin{proposition} \label{prop: necessaryprop}
Let $H \leq \Sigma_{d}$, let $A$ be an abelian group, and let $\widehat{\phi}: V_{d}'(H) \rightarrow A$ be  a homomorphism.
\begin{enumerate}
\item If $d$ is even, then $(d-1)\widehat{\phi}\Lambda_{1}(h) = 0$ for all $h \in H$.
\item If $d$ is odd, then $(d-1)\widehat{\phi}\Lambda_{1}(h) = 0$ for all $h \in H \cap A_{d}$.
\end{enumerate}
\end{proposition}

\begin{proof}
There are two cases.
\begin{enumerate}
\item We have $\Lambda_{1}(h) = \Sigma_{i=1}^{d} S_{1h(i)}hS_{1i}^{\ast} + \Sigma_{i=2}^{d} S_{i}S_{i}^{\ast}$.
We set $\widehat{h} = \Sigma_{i=1}^{d} S_{1h(i)}S_{1i}^{\ast} + \Sigma_{i=2}^{d} S_{i}S_{i}^{\ast}$.
Thus $\widehat{h} \in V_{d} \leq V_{d}'(H)$. Since $V_{d}$ is simple and nonabelian, we conclude that $\widehat{\phi}(\widehat{h})=0$.
Therefore
$$ \widehat{\phi}\Lambda_{1}(h) = \widehat{\phi}(\widehat{h}^{-1}\Lambda_{1}(h)) = \widehat{\phi}\left(\prod_{i=1}^{d} \Lambda_{1i}(h)\right) =
d(\widehat{\phi}\Lambda_{1}(h)),$$
where the final equation follows from Lemma \ref{lemma: conjugate}.

\item This works as in the first case, but with one minor difference. We define $\widehat{h}$ as above. This time we observe that 
$\widehat{h} \in V_{d}'$, by Lemma \ref{lemma: oddeventables}, since $h \in H \cap A_{d}$. Because $V_{d}'$ is simple  and nonabelian, we conclude
that $\widehat{\phi}(\widehat{h}) = 0$, and the rest of the proof goes through unchanged. 
\end{enumerate}
\end{proof}

\begin{proposition} \label{prop: sufficientprop}
Let $A$ be an abelian group.
\begin{enumerate}
\item Let $d$ be even. If $\phi: H \rightarrow A$ is a homomorphism satisfying $(d-1)\phi(h) = 0$ for all $h \in H$, then there is a unique
homomorphism $\widehat{\phi}:  V_{d}'(H) \rightarrow A$ such that $\left( \widehat{\phi}  \Lambda_{1} \right)|_{H} = \phi$.
\item Let $d$ be odd. If $\phi: H \rightarrow A$ is a homomorphism satisfying $(d-1)\phi(h) = 0$  for all $h \in H \cap A_{d}$, then there is a unique
homomorphism $\widehat{\phi}: V_{d}'(H) \rightarrow A$ such that $\left( \widehat{\phi}  \Lambda_{1} \right)|_{H} = \phi$.
\end{enumerate}
Moreover, in both cases, $\widehat{\phi}(V_{d}'(H)) = \phi(H)$.
\end{proposition}

\begin{proof}
We first prove uniqueness (for both cases). Thus suppose that $\widehat{\phi}: V_{d}'(H) \rightarrow A$ satisfies $  \widehat{\phi}\Lambda_{1} = \phi$, for all
$h \in H$. Assume that $d$ is even. Let $g \in V_{d}'(H)$; say $g = \Sigma_{i=1}^{m} S_{u_{i}}h_{i}S_{v_{i}}^{\ast}$, where $h_{i} \in H$. We let
$\widehat{g} = \Sigma_{i=1}^{m} S_{u_{i}}S_{v_{i}}^{\ast}$. Since $\widehat{g} \in V_{d}$ and $V_{d}$ is simple and nonabelian, $\widehat{\phi}(\widehat{g}) = 0$.

$$\widehat{\phi}(g) = \widehat{\phi}(\widehat{g}^{-1}g)
= \widehat{\phi}\left( \prod_{i=1}^{m} \Lambda_{v_{i}}(h_{i}) \right)
= \sum_{i=1}^{m} \widehat{\phi}\Lambda_{v_{i}}(h_{i})
= \sum_{i=1}^{m} \phi(h_{i}).$$
This proves uniqueness for the case in which $d$ is even.

Now assume that $d$ is odd. There are two subcases. We first consider the case in which $H \not  \leq A_{d}$. We check uniqueness by verifying that the 
condition $\widehat{\phi}\Lambda_{1}(h) = \phi(h)$ for all $h \in H$ completely determines $\widehat{\phi}$ on a generating set for $V_{d}'(H)$.
We note that $V_{d}'(H) = \langle V_{d}', v, \{ \Lambda_{1}(h) : h \in H \} \rangle$, where $v$ is an arbitrary element of $V_{d} \backslash V_{d}'$.

Let $h \in H \backslash A_{d}$. We write $h = \Sigma_{i=1}^{d} S_{h(i)} h S_{i}^{\ast}$.
Let $\widehat{h} = \Sigma_{i=1}^{d} S_{h(i)}S_{i}^{\ast}$. By Lemma \ref{lemma: oddeventables},
$\widehat{h} \in V_{d} \backslash V_{d}'$. We want to compute $\widehat{\phi}(\widehat{h})$.
$$\widehat{\phi}(\widehat{h}^{-1}h) = \widehat{\phi}\left( \prod_{i=1}^{d} \Lambda_{i}(h) \right)
= \sum_{i=1}^{d} \widehat{\phi}\Lambda_{i}(h) = d \phi(h).$$
Thus $\widehat{\phi}(\widehat{h}) = \widehat{\phi}(h) - d\phi(h)$.

We now split the original table for $h$ at the first column. The resulting table is even, by Lemma \ref{lemma: oddeventables}.
We let $\widetilde{h}$ denote the result of striking out the middle row of the latter table. Thus, $\widetilde{h} \in V_{d}'$. We have
$$ \widehat{\phi}(h) = \widehat{\phi}(\widetilde{h}^{-1}h) = (2d-1)\widehat{\phi}\Lambda_{1}(h) = (2d-1)\phi(h).$$
It follows that $\widehat{\phi}(\widehat{h}) = (d-1)\phi(h)$.

One also has $\widehat{\phi}(V_{d}') = 0$ (since $V_{d}'$ is simple and nonabelian) and $\widehat{\phi}\Lambda_{1}(h) = \phi(h)$ for
each $h \in H$. Thus, we've completely determined $\widehat{\phi}$ on a generating set, proving uniqueness.

In the other subcase, $d$ is odd and $H \leq A_{d}$. This case essentially follows the pattern of the case in which $d$ is even.
The proof is omitted.

Finally, we prove the existence of $\widehat{\phi}$ in cases (1) and (2).
We assume first that $d$ is  even. According to Nekrashevych \cite{NekJOT}, we can extend a homomorphism $\pi: H \rightarrow A$ to
$\pi: V_{d}(H) \rightarrow A$ if
$$ h = \sum_{i=1}^{n} S_{y_{i}} h_{i} S_{x_{i}}^{\ast} \text{ implies } \pi(h) = \sum_{i=1}^{n} \pi(h_{i}).$$
We apply this principle to $\phi: H \rightarrow A$. Let $h \in H$, and suppose $h = \sum_{i=1}^{n} S_{y_{i}}h_{i}S_{x_{i}}^{\ast}$. The
definition of $h$ implies that $h_{i} = h$, for $i=1,\ldots, n$. We must have $n= 1 + (d-1)k$, for some $k \in \mathbb{N} \cup \{ 0 \}$.
Therefore,
$$ \sum_{i=1}^{n} \phi(h_{i}) = \left[ 1 + (d-1)k \right] \phi(h) = \phi(h).$$
It follows that Nekrashevych's condition is satisfied, so there is a well-defined homomorphism $\widehat{\phi}: V_{d}(H) \rightarrow A$ which
extends $\phi: H \rightarrow A$. We must show that $\phi(h) = \widehat{\phi}\Lambda_{1}(h)$, for all $h \in H$. Let $h \in H$. We have
$\Lambda_{1}(h) = \sum_{i=1}^{d} S_{1h(i)}hS_{1i}^{\ast} + \sum_{i=2}^{d} S_{j}S_{j}^{\ast}$. 
Set $\widetilde{h} = \sum_{i=1}^{d} S_{1h(i)}S_{1i}^{\ast} + \sum_{i=2}^{d} S_{j}S_{j}^{\ast}$. Since $\widetilde{h} \in V_{d}$, 
$\widehat{\phi}(\widetilde{h}) = 0$.
$$ \widehat{\phi}\Lambda_{1}(h) = \widehat{\phi}(\widetilde{h}^{-1}\Lambda_{1}(h)) = 
\widehat{\phi}\left( \prod_{i=1}^{d} \Lambda_{1i}(h)\right) = d\widehat{\phi}\Lambda_{1}(h).$$
Let $\widehat{h} = \sum_{i=1}^{d} S_{h(i)}S_{i}^{\ast}$. Since $\widehat{h} \in V_{d}$, $\widehat{\phi}(\widehat{h}) = 0$.
$$\widehat{\phi}(h) = \widehat{\phi}(\widehat{h}^{-1}h) 
= \widehat{\phi}\left( \prod_{i=1}^{d} \Lambda_{i}(h) \right) = d \widehat{\phi}\Lambda_{1}(h).$$
Therefore $\widehat{\phi}\Lambda_{1}(h) = \widehat{\phi}(h) = \phi(h)$, for all $h \in H$, as required.

Now assume that $d$ is odd, and $\phi: H \rightarrow A$ satisfies $(d-1) \phi(h) = 0$, for all $h \in H \cap A_{d}$. We want to find
$\widehat{\phi}: V_{d}'(H) \rightarrow A$ such that $\widehat{\phi}\Lambda_{1}(h) = \phi(h)$, for all $h \in H$. According to Nekrashevych
\cite{NekJOT}, a homomorphism $\phi: H \rightarrow A$ satisfying:
\begin{itemize}
\item [$\ast$] There is $z \in A$, where $2z = 0$, such that $h = \sum_{i=1}^{m} S_{u_{i}}h_{i}S_{v_{i}}^{\ast}$ ($h \in H$)
implies
\begin{itemize}
\item $\phi(h) = \sum_{i=1}^{m} \phi(h_{i})$ if $\sum_{i=1}^{m} S_{u_{i}}S_{v_{i}}^{\ast}$ is even, and
\item  $\phi(h) = \sum_{i=1}^{m} \phi(h_{i}) + z$ if $\sum_{i=1}^{m} S_{u_{i}}S_{v_{i}}^{\ast}$ is odd.
\end{itemize}
\end{itemize}
extends to a homomorphism $\widehat{\phi}: V_{d}'(H) \rightarrow A$. (Note that, in our case, the above equations can be simplified since 
$h_{i} = h$ for $i=1, \ldots, m$.) If $H \leq A_{d}$, then we set $z=0$. If $H \not \leq A_{d}$, we set $z= (d-1)\phi(h)$ for some
(equivalently, any) $h \in H \backslash A_{d}$. (Any two elements $h_{1}, h_{2} \in H \backslash A_{d}$ must satisfy 
$(d-1)\phi(h_{1}) = (d-1)\phi(h_{2})$, since $(d-1)\phi(h_{1}h_{2}^{-1}) = 0$ by our assumptions.)

We check Nekrashevych's condition. First suppose that $h \in H \cap A_{d}$. Now $h = \sum_{i=1}^{m} S_{u_{i}}hS_{v_{i}}^{\ast}$. Since we must
have $m = 1 + (d-1)k$, for some $k \in \mathbb{N} \cup \{ 0 \}$,
$$ \sum_{i=1}^{m} \phi(h) = (1 + (d-1)k)\phi (h) = \phi(h).$$
If $h \in H \backslash A_{d}$, then we can express $h$ in the same form as in the preceding lines. We have $m = 1 + (d-1)k$, for some
$k \in \mathbb{N} \cup \{ 0 \}$. The table for $h$ is odd or even accordingly as $k$ is odd or even. If $k$ is odd, then
\begin{align*}
\sum_{i=1}^{m} \phi(h) + z &= \left[ 1 + (d-1)(2\ell - 1) \right] \phi(h) + (d-1)\phi(h) \\
&= \left[ 1 + 2(d-1)\ell \right] \phi(h)\\
&= \phi(h) + (d-1)\ell \phi(h^{2}) \\
&= \phi(h).
\end{align*}
If $k$ is even, then $\sum_{i=1}^{m} \phi(h) = \phi(h)$, exactly as in the case when $d$ is even. It follows that there is a homomorphism
$\widehat{\phi}: V_{d}'(H) \rightarrow A$ extending $\phi: H \rightarrow A$.

The final step is to verify that $\widehat{\phi}\Lambda_{1}(h) = \phi(h)$ for all $h \in H$. Here we merely sketch the argument. If 
$h \in H \cap A_{d}$, then we can apply the argument for the case $d$ even with essentially no change. If $h \in H \backslash A_{d}$, 
then we can still mimic the case $d$ even, except that we must split the table for $\Lambda_{1}(h)$ twice (in order to produce an even table). The
remaining adjustments are straightforward.

The proof of the final statement (that $\widehat{\phi}(V_{d}'(H)) = \phi(H)$) follows from the expressions for $\widehat{\phi}(h)$ that
were derived in the proof of uniqueness.  
\end{proof}

\begin{definition}
If $G$ is a group, then we let $G^{k} = \{ g^{k} \mid g \in G \}$.  We note
that $G^{k}$ is merely a set, not necessarily a group.
\end{definition}

\begin{definition}
If $G$ is a group, then we let $G_{ab}$ denote its abelianization $G/[G,G]$.
\end{definition}

\begin{theorem} \label{thm: quotient}
Let $H \leq \Sigma_{d}$, and
let $\widehat{\phi}: V_{d}'(H) \rightarrow V_{d}'(H)_{ab}$ denote the canonical projection. 
The composition $\widehat{\phi} \Lambda_{1}: H \rightarrow V_{d}'(H)_{ab}$ is surjective.
\begin{enumerate}
\item If $d$ is even, then the kernel of $\widehat{\phi} \Lambda_{1}$ is 
$$ N = \langle [H,H], H^{d-1} \rangle.$$
\item If $d$ is odd, then the kernel of $\widehat{\phi} \Lambda_{1}$ is
$$ N = \langle [H,H], (H \cap A_{d})^{d-1} \rangle.$$
\end{enumerate}
In either case, we have $V_{d}'(H)_{ab} \cong H/N$.
\end{theorem}

\begin{proof}
We will prove the theorem in the case $d$ is even, the other case being similar. By Proposition \ref{prop: necessaryprop},
we have $(d-1)\widehat{\phi}\Lambda_{1}(h) = 0$ for all $h \in H$. Thus, reading Proposition \ref{prop: sufficientprop} with
$\phi = \widehat{\phi}\Lambda_{1}|_{H}$, we find that $\phi(H) = \widehat{\phi}(V_{d}'(H)) = V_{d}'(H)_{ab}$. It follows that
$\widehat{\phi} \Lambda_{1}: H \rightarrow V_{d}'(H)_{ab}$ is surjective.

We have surjective homomorphisms $\phi: H \rightarrow V_{d}'(H)_{ab}$ and $\pi: H \rightarrow H/N$, where the latter is the canonical
projection. Since $(d-1)\pi(h) = 0$ for all $h \in H$, there is a unique homomorphism $\widehat{\pi}: V_{d}'(H) \rightarrow H/N$ such that
$\widehat{\pi} \Lambda_{1} = \pi$, by Proposition \ref{prop: sufficientprop}. Since $\phi(N) = 0$, there is a well-defined homomorphism
$\theta: H/N \rightarrow V_{d}'(H)_{ab}$ such that $\theta \pi = \phi$.

We claim that $\theta$ is an isomorphism. The function $\theta$ is surjective since $\theta \pi = \phi$ and $\phi$ is surjective.
Suppose $\theta(hN) = 0$ (i.e., $\theta\pi(h) = 0$). It follows that $\widehat{\phi}\Lambda_{1}(h) = \phi(h) = 0$, so that 
$\Lambda_{1}(h) \in [V_{d}'(H), V_{d}'(H)]$. Thus
$hN = \pi(h) = \widehat{\pi}\Lambda_{1}(h) = 0$, since $H/N$ is abelian and $\Lambda_{1}(h)$ is a product of commutators, so $\theta$ is injective.

Thus $N = Ker\pi = Ker\phi = Ker\widehat{\phi}\Lambda_{1}$. 
\end{proof}

\subsection*{The Commutator Subgroup of $V_{d}'(H)$ is Simple}
The argument of this section is adapted slightly from Brin \cite{BrinHDTG}.

\begin{lemma}
Let $K$ be a closed proper subset of $A^{\omega}$, and let $U$ be an open subset of $A^{\omega}$.  There
is $v \in V_{d}'$ such that $v(K) \subseteq U$.
\end{lemma}

\begin{proof}
Let $K$ be closed; let $B_{1}$, $B_{2}$, $\ldots$, $B_{d-1}$ be open disjoint metric balls in $K^{c}$.  We
can find $v_{1} \in V_{d}'$ such that $v_{1}(B_{i}) = iA^{\omega}$, for 
$i = 1, \ldots, d-1$. It follows that $v_{1}(K) \subseteq dA^{\omega}$.

Given an open subset $U$ in $A^{\omega}$, we let $\hat{B} \subseteq U$ be an open metric ball.  By Lemma \ref{lemma: conjugate},
there is $v_{2} \in V_{d}'$ so that $v_{2}(dA^{\omega}) = \hat{B}$.  Therefore, $v_{2}v_{1}(K) \subseteq v_{2}(dA^{\omega}) = \hat{B} \subseteq U$,
and $v_{2}v_{1} \in V_{d}'$.
\end{proof}

\begin{definition}
Let $f: A^{\omega} \rightarrow A^{\omega}$ be a homeomorphism. The \emph{support} of $f$, denoted $\text{supp}(f)$, 
is $\{ x \in A^{\omega} \mid f(x) \neq x \}$.
\end{definition}

\begin{lemma}
The group $V_{d}'$ can be generated by a set $S$ such that,
for any $x, y \in S$, there is an open ball $B \subseteq A^{\omega}$ so that
$x|_{B} = y|_{B} = id_{B}$.
\end{lemma}

\begin{proof}
We sketch the proof. Suppose first that $d$ is even (and so $V_{d}' = V_{d}$). 
An element $\tau = \tau(u,v) \in V_{d}$ ( where $u, v \in A^{\ast}$, $|u|, |v| \geq 3$) 
is a \emph{small transposition} if $\tau(ua_{1}a_{2} \ldots) = va_{1}a_{2}\ldots$,
$\tau(va_{1}a_{2} \ldots) = ua_{1}a_{2}\ldots$, and $\tau$ fixes every other point 
in $A^{\omega}$. We let $S$ be the set of all small transpositions. If $\tau_{1}, \tau_{2} \in S$, then
$\text{supp}(\tau_{1}) \cup \text{supp}(\tau_{2})$ is a proper closed subset of $A^{\omega}$, so there is
an open ball $B \subseteq A^{\omega}$ such that $B \cap (\text{supp}(\tau_{1}) \cup \text{supp}(\tau_{2})) = \emptyset$.
Now we need only check that $S$ generates $V_{d}$.

We appeal to the well-known interpretation of the groups $V_{d}$ using tree pairs. An element $v \in V_{d}$ can be expressed as
a triple $(T_{1}, T_{2}, \sigma)$, where $T_{1}$ and $T_{2}$ are rooted ordered $d$-ary trees, and $\sigma$ is a bijection 
between their leaves. The nodes (vertices) of the trees represent balls in $A^{\omega}$. Given $v = (T_{1}, T_{2}, \sigma)$, we can
introduce  cancelling carets in order to express $v$ as $(T'_{1}, T'_{2}, \sigma')$, where each of $T'_{1}$ and $T'_{2}$ contains
the full rooted ordered $d$-ary subtree of depth $3$. Now if $\tau = \tau(u_{1},u_{2}) \in S$, then $\tau v = ( T'_{1}, T''_{2}, \sigma'')$, 
where $T''_{2}$ is obtained from $T'_{2}$ by interchanging the trees below $u_{1}$ and $u_{2}$. It follows that there is a sequence
$\tau_{1}, \tau_{2}, \tau_{3}, \ldots, \tau_{n} \in S$ such that
$\tau_{n}\tau_{n-1}\ldots \tau_{1}v = ( T'_{1}, T'_{1}, \widehat{\sigma})$, for some permutation $\widehat{\sigma}$. Now $\widehat{\sigma}$ is simply a 
permutation of the leaves of $T'_{1}$, and it follows that $\widehat{\sigma} \in \langle S \rangle$. Thus, $S$ generates $V_{d}$.

If $d$ is odd, then the proof is similar, but one needs to use even permutations. We let $S = \{ \tau_{1}\tau_{2} \mid \tau_{1}, \tau_{2} \text{ are disjoint small
transpositions} \}$. One now proceeds in the same way. We can always choose $\tau_{1}$ in such a way that it permutes the leaves of the range tree $T'_{2}$,
and so only the transposition $\tau_{2}$ alters $T'_{2}$. Thus, for $v \in V'_{d}$, there is a sequence $s_{1}, s_{2}, \ldots, s_{n} \in S$ such that 
$s_{n}s_{n-1}\ldots s_{1} v = ( T'_{1}, T'_{1}, \sigma)$ where $\sigma$ is an even permutation of the leaves. It follows that $\sigma \in \langle S \rangle$.  
 \end{proof}

\begin{lemma}
If $N \trianglelefteq [ V_{d}'(H), V_{d}'(H)]$ and $N \neq 1$, then $V_{d}' \leq N$.
\end{lemma}

\begin{proof}
Let $N \trianglelefteq [ V_{d}'(H), V_{d}'(H)]$, and let $1 \neq j \in N$.  There is some open ball $E \subseteq Ends(T_{d})$
such that $j(E) \cap E = \emptyset$.  We choose some generating set $S$ for $V_{d}'$ as in the previous lemma, and let $x, y \in S$ be arbitrary.
We will show that $[x,y] := xyx^{-1}y^{-1} \in N$.  We first note that $V_{d}' \leq [V_{d}'(H), V_{d}'(H)]$, since $V_{d}'$ is 
simple and nonabelian, and therefore
must be sent to $0$ by the projection $\pi: V_{d}'(H) \rightarrow V_{d}'(H)_{ab}$.

By the defining property of $S$, there is some open ball $B$ so that $x_{\mid B} = y_{ \mid B} = id_{B}$.  
By Lemma \ref{lemma: conjugate}, there is $k \in V_{d}'$
so that $k(B^{c}) \subseteq E$.  We claim that 
$$\hat{y} := y^{j^{k^{-1}}} := y^{k^{-1}jk} = k^{-1}jkyk^{-1}j^{-1}k$$
and $x$ commute.  

One easily shows that $supp(a^{b}) = b \cdot supp(a)$ and that elements of $V_{d}'(H)$ with disjoint supports
must commute.  Now
$$\text{supp}\left( y^{(j^{k^{-1}})} \right) = \left( j^{k^{-1}} \right) \cdot \text{supp}(y) 
= k^{-1}jk \cdot \text{supp}(y) 
\subseteq k^{-1}jk(B^{c})
\subseteq k^{-1}j(E).$$
and
$$ k \cdot supp(x) \subseteq k \cdot (B^{c}) \subseteq E \quad \Rightarrow \quad supp(x) \subseteq k^{-1}(E),$$
so it follows that the supports of $\hat{y}$ and $x$ are disjoint, so that these elements must commute.  Moreover,
one readily checks that $[x, \hat{y}] = [x,y]$ modulo $N$. It follows directly that $[x,y] \in N$, since $[x, \hat{y}] = 1$.

It now follows that 
$\displaystyle V_{d}' / V_{d}' \cap N$ is abelian.  Since $V_{d}'$ is simple and nonabelian, we have $V_{d}' \leq N$.
\end{proof}

\begin{theorem} \label{thm: simplecommutator}
The commutator subgroup of $V_{d}'(H)$ is simple.
\end{theorem}

\begin{proof}
Let $1 \neq N \trianglelefteq [V_{d}'(H), V_{d}'(H)]$.  Set
$$ M = \bigcap_{v \in V_{d}'(H)} vNv^{-1}.$$
Each group $vNv^{-1}$ is non-trivial and normal in $[V_{d}'(H), V_{d}'(H)]$.  
It follows from the previous lemma that $V_{d}' \leq vNv^{-1}$ for all
$v \in V_{d}'(H)$, so $V_{d}' \leq M$.  Now $M \trianglelefteq V_{d}'(H)$ by construction.
Since the only proper quotients of $V_{d}'(H)$ are abelian according to Nekrashevych \cite{NekJOT}, 
every normal subgroup of
$V_{d}'(H)$ contains $[V_{d}'(H), V_{d}'(H)]$.  It follows
that
$$ [V_{d}'(H), V_{d}'(H)] \leq M \leq N \leq [V_{d}'(H), V_{d}'(H)].$$
Therefore, all of the above containments are equalities.
\end{proof}

We conclude with a couple of simple applications of the ideas from this section. 

\begin{corollary}
If $d \geq 2$ and $H \leq \Sigma_{d}$, then $V_{d}(H)$ has a simple subgroup of finite index.
\end{corollary}

\begin{proof}
The group $V'_{d}(H)$ has finite index in $V_{d}(H)$. The commutator subgroup of $V'_{d}(H)$ is simple by Theorem \ref{thm: simplecommutator},
and has finite index in $V'_{d}(H)$ by Theorem \ref{thm: quotient}. Thus, $[V'_{d}(H), V'_{d}(H)]$ is a simple subgroup of finite index in $V_{d}(H)$.
\end{proof}

\begin{corollary}
If $d \geq 2$ is even, then $V_{d}(\Sigma_{d})$ is simple. If $d \geq 3$ is odd, then $V_{d}(\Sigma_{d})$ has a simple subgroup of index $2$.
\end{corollary}

\begin{proof}
If $d$ is even, then $V_{d}(\Sigma_{d}) = V'_{d}(\Sigma_{d})$ by definition. The latter group is equal to its commutator subgroup by Theorem \ref{thm: quotient},
which is simple by Theorem \ref{thm: simplecommutator}. If $d$ is odd, then we still have $V_{d}(\Sigma_{d}) = V'_{d}(\Sigma_{d})$, since $\Sigma_{d}$ contains
odd permutations. The commutator subgroup $[V'_{d}(\Sigma_{d}), V'_{d}(\Sigma_{d})]$ is simple by Theorem \ref{thm: simplecommutator} 
and has index two in $V'_{d}(\Sigma_{d})$ by Theorem \ref{thm: quotient}.
\end{proof}
     

\section{Some FSS groups are braided diagram groups} \label{sec:diagramgroups}

In this section, we will show that the class of braided diagram groups (Definition \ref{def: diagramgroups}) over tree-like semigroup presentations
(Definition \ref{def: treelike}) is exactly the same as the class of FSS groups defined by small similarity structures  
(Definition \ref{def: smallsim}). 
We review all of the necessary definitions below. The main results of the section are Theorem \ref{thm: bigisomorphism}
and Corollary \ref{cor:maincor}.

The theory of braided diagram groups was first sketched by Guba and Sapir \cite{GubaSapir}.

\subsection*{Braided diagram groups over semigroup presentations}

\begin{definition}
Let $\Sigma$ be a set, called an \emph{alphabet}. The \emph{free semigroup on $\Sigma$}, denoted $\Sigma^{+}$, is
the collection of all positive non-empty strings formed from $\Sigma$, i.e., 
$$ \Sigma^{+} = \{ u_{1}u_{2} \ldots u_{n} \mid n \in \mathbb{N}, u_{i} \in \Sigma \text{ for } i \in \{ 1, \ldots, n \} \}.$$
The \emph{free monoid on $\Sigma$}, denoted $\Sigma^{\ast}$, is the union $\Sigma^{+} \cup \{ 1 \}$, where $1$ denotes the
empty string. (Here we assume that $1 \not \in \Sigma$ to avoid ambiguity.)

We write $w_{1} \equiv w_{2}$ if $w_{1}$ and $w_{2}$ are equal as words in $\Sigma^{\ast}$.

The operations in $\Sigma^{+}$ and $\Sigma^{\ast}$ are concatenation.
\end{definition}

\begin{definition}
A \emph{semigroup presentation} $\mathcal{P} = \langle \Sigma \mid \mathcal{R} \rangle$ consists of an alphabet $\Sigma$ and
a set $\mathcal{R} \subseteq \Sigma^{+} \times \Sigma^{+}$. The elements of $\mathcal{R}$ are called \emph{relations}.
\end{definition}

\begin{remark}
A relation $(w_{1}, w_{2}) \in \mathcal{R}$ can be viewed as an equality between the words $w_{1}$ and $w_{2}$. 
We use ordered pairs to describe these equalities because we will occasionally want to make a distinction between the 
left and right sides of a relation.

A semigroup presentation $\mathcal{P}$ determines a semigroup  $S_{\mathcal{P}}$. We define a relation $\sim$ on 
$\Sigma^{+}$ as follows: $w_{1} \sim w_{2}$ if $w_{1} \equiv u \ell v$ and $w_{2} \equiv urv$ where $u, v \in \Sigma^{\ast}$
and $(\ell, r) \in \mathcal{R}$. The transitive, symmetric closure of $\sim$, which we will denote $\dot{\sim}$, is an equivalence
relation on $\Sigma^{+}$. The equivalence classes of $\dot{\sim}$ are the elements of $S_{\mathcal{P}}$. The operation of $S_{\mathcal{P}}$
is concatenation, which is well-defined and associative on the equivalence classes.

In this paper, we will make very little direct use of the semigroup $S_{\mathcal{P}}$.
\end{remark}

\begin{definition} \label{def: braideddiagrams}
(Braided Semigroup Diagrams)
A \emph{frame} is a homeomorphic copy of $\partial([0,1]^{2}) = ( \{ 0, 1 \} \times [0,1]) \cup ( [0,1] \times \{ 0, 1 \} )$. A frame
has a \emph{top} side, $(0,1) \times \{ 1 \}$, a \emph{bottom} side, $(0,1) \times \{ 0 \}$, and \emph{left} and \emph{right} sides,
$\{ 0 \} \times [0,1]$ and $\{ 1 \} \times [0,1]$, respectively. The top and bottom of a frame have obvious left to right orderings.

A \emph{transistor} is a homeomorphic copy of $[0,1]^{2}$. A transistor has top, bottom, left, and right sides, just as a frame does. The
top and bottom of a transistor also have obvious left to right orderings.

A \emph{wire} is a homeomorphic copy of $[0,1]$. Each wire has a bottom $0$ and a top $1$.

Let $\mathcal{P} = \langle \Sigma \mid \mathcal{R} \rangle$ be a semigroup presentation. Let $\mathcal{T}(\Delta)$ be a finite
(possibly empty) set of transistors. Let $\mathcal{W}(\Delta)$ be a finite, nonempty  set of wires. We let $F(\Delta) = \partial ( [0,1]^{2})$ be
a frame. We let $\ell_{\Delta}: \mathcal{W}(\Delta) \rightarrow \Sigma$ be an arbitrary function, called the \emph{labelling function}.

For each wire $W \in \mathcal{W}(\Delta)$, we choose a point $t(W)$ on the bottom of a transistor, or on the top of the frame, and a point $b(W)$
on the top of a transistor, or on the bottom of the frame. The points $t(W)$ and $b(W)$ are called the \emph{top} and \emph{bottom contacts} of $W$, respectively.

We attach the top of each wire $W$ to $t(W)$ and the bottom of $W$ to $b(W)$. The resulting topological space $\Delta$ is called a \emph{braided diagram over
$\mathcal{P}$} if the following additional conditions are satisfied:

\begin{enumerate}
\item If $W_{i}$, $W_{j} \in \mathcal{W}(\Delta)$, $t(W_{i}) = t(W_{j})$ only if $W_{i} = W_{j}$, and $b(W_{i}) = b(W_{j})$ only if $W_{i} = W_{j}$.
In other words, the disjoint union of all of the wires maps injectively into the quotient.
(We note that, by definition, one cannot have $t(W_{i}) = b(W_{j})$, even if $i = j$.)
\item We consider the top of some transistor $T \in \mathcal{T}(\Delta)$. Reading from left to right, we find contacts 
$$ b(W_{i_{1}}), b(W_{i_{2}}), \ldots, b(W_{i_{n}}),$$
where $n \geq 0$. The word $\ell_{t}(T) = \ell(W_{i_{1}}) \ell(W_{i_{2}}) \ldots \ell(W_{i_{n}})$ is called the \emph{top label of $T$}.
Similarly, reading from left to right along the bottom of $T$, we find contacts
$$ t(W_{j_{1}}), t(W_{j_{2}}), \ldots, t(W_{j_{m}}),$$
where $m \geq 0$. The word $\ell_{b}(T) = \ell(W_{j_{1}}) \ell(W_{j_{2}}) \ldots \ell(W_{j_{m}})$ is called the \emph{bottom label of $T$}.
We require that, for any $T \in \mathcal{T}(\Delta)$, either $(\ell_{t}(T), \ell_{b}(T)) \in \mathcal{R}$ or $(\ell_{b}(T), \ell_{t}(T)) \in \mathcal{R}$.
(We emphasize that it is not sufficient for $\ell_{t}(T)$ to be equivalent to $\ell_{b}(T)$ modulo the relation $\sim$ determined by $\mathcal{R}$.)
\item We define a relation $\preceq$ on $\mathcal{T}(\Delta)$ as follows. Write $T_{1} \preceq T_{2}$ if there is some wire $W$ such that
$t(W) \in T_{2}$ and $b(W) \in T_{1}$. We require that the transitive closure $\dot{\preceq}$ of $\preceq$ be a strict partial order on $\mathcal{T}(\Delta)$.
\end{enumerate}
\end{definition}

\begin{definition} \label{def: (a,b)diagrams}
Let $\Delta$ be a braided diagram over $\mathcal{P}$. Reading from left to right across the top of the frame $F(\Delta)$,
we find contacts 
$$ t(W_{i_{1}}), t(W_{i_{2}}), \ldots, t(W_{i_{n}}),$$
for some $n \geq 0$. The word $\ell(W_{i_{1}})\ell(W_{i_{2}}) \ldots \ell(W_{i_{n}}) = \ell_{t}(\Delta)$
is called the \emph{top label of $\Delta$}. We can similarly define the \emph{bottom label of $\Delta$}, $\ell_{b}(\Delta)$.
We say that $\Delta$ is a \emph{braided $(\ell_{t}(\Delta), \ell_{b}(\Delta))$-diagram over $\mathcal{P}$}.
\end{definition}

\begin{remark} \label{rem: caution} 
One should note that braided diagrams, despite the name, are not truly braided. In fact, two braided diagrams are equivalent
(see Definition \ref{def: equivalentdiagrams}) if there is a certain type of marked homeomorphism between them. Equivalence therefore doesn't depend
on any embedding into a larger space. Braided diagram groups (Definition \ref{def: diagramgroups}) also seem to have little in common 
with Artin's braid groups.  
\end{remark}

\begin{example}
Let $\mathcal{P} = \langle a, b, c \mid ab=ba, ac=ca, bc=cb \rangle$. Figure \ref{figure1} shows an example of a braided $(aabc, acba)$-diagram
over the semigroup presentation $\mathcal{P}$. The frame is the box formed by the dashed line. The wires that appear to cross in the figure do not
really touch, and it is unnecessary to specify which wire passes over the other one. See Remark \ref{rem: caution}.
\begin{center}
\begin{figure}[!h]
\includegraphics{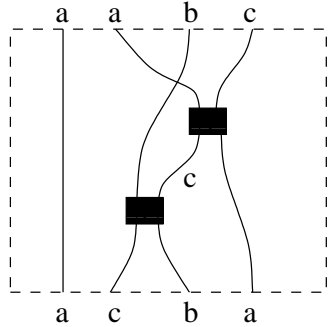}
\caption{A braided $(aabc, acba)$-diagram over the semigroup presentation $\mathcal{P} = \langle a, b, c \mid ac=ca, ab=ba, bc=cb \rangle$.}
\label{figure1}
\end{figure}
\end{center}
\end{example}

If $\Delta$ is a braided diagram over $\mathcal P$, then its \emph{inverse}  $\Delta^{-1}$ is defined in the obvious way: the picture is reflected across a horizontal line so that tops and bottoms are interchanged.

\begin{remark}
Suppose that $\Delta$ is a braided diagram over some semigroup presentation $\mathcal{P}$. Each transistor of $\Delta$ must have nonempty top and bottom labels,
by Definition \ref{def: braideddiagrams}. It also follows that the top and bottom labels of $\Delta$ itself are nonempty. Indeed, if $\Delta$ has a least one
transistor, then it will have at least one transistor $T$ that is maximal with respect to the strict partial order $\dot{\preceq}$. There is at least one wire
$W$ such that $b(W)$ is on the top of $T$. The only possibility for $t(W)$ is that it is on the top of the frame, so the top label of $\Delta$ is nonempty. Similarly,
the bottom label of $\Delta$ is nonempty. If $\Delta$ has no transistors, it will nevertheless have at least one wire $W$ by Definition \ref{def: braideddiagrams},
and the ends of $W$ will be attached to the bottom and top of the frame, making the top and bottom labels of $\Delta$ nonempty in this case as well. 
\end{remark}

\begin{definition}
(Concatenation of braided diagrams)
Let $\Delta_{1}$ and $\Delta_{2}$ be braided diagrams over $\mathcal{P}$. We suppose that $\Delta_{1}$ 
is a $(w_{1}, w_{2})$-diagram and $\Delta_{2}$ is a $(w_{2},w_{3})$-diagram. We define the concatenation $\Delta_{1} \circ \Delta_{2}$ 
as follows.

Suppose that $W_{i_{1}}, W_{i_{2}}, \ldots W_{i_{n}} \in \mathcal{W}(\Delta_{1})$ are the wires of $\Delta_{1}$ which meet the bottom of the frame
$F(\Delta_{1})$, listed in such a way that $b(W_{i_{p}})$ is to the left of $b(W_{i_{q}})$ if $p < q$.
We let $W_{j_{1}}, W_{j_{2}}, W_{j_{3}}, \ldots, W_{j_{n}} \in \mathcal{W}(\Delta_{2})$ be the wires of $\Delta_{2}$, similarly listed in the order that their 
top contacts are arranged from left to right 
on the frame $F(\Delta_{2})$. We note that $\ell(W_{i_{k}}) = \ell(W_{j_{k}})$ for
$k \in \{ 1, \ldots, n \}$ by our assumptions. Remove the bottom of $F(\Delta_{1})$ and the top of $F(\Delta_{2})$, and identify the top of the wire 
$W_{j_{k}}$ with the bottom of the wire $W_{i_{k}}$. Glue the point $(0,0) \in F(\Delta_{1})$ to $(0,1) \in F(\Delta_{2})$ and
$(1,0) \in F(\Delta_{1})$ to $(1,1) \in F(\Delta_{2})$. The resulting space is the concatenation $\Delta_{1} \circ \Delta_{2}$. 
There is a natural labelling function $\ell$ on the new collection of wires making $\Delta_{1} \circ \Delta_{2}$ a braided diagram over $\mathcal{P}$.
\end{definition}

\begin{definition}
(Dipoles)
Let $\Delta$ be a braided semigroup diagram over $\mathcal{P}$. We say that the transistors $T_{1}, T_{2} \in \mathcal{T}(\Delta)$, $T_{1} \, \dot{\preceq} \, T_{2}$,
form a \emph{dipole} if:
\begin{enumerate}
\item the bottom label of $T_{1}$ is the same as the top label of $T_{2}$, and
\item there are wires $W_{i_{1}}, W_{i_{2}}, \ldots, W_{i_{n}} (n \geq 1)$ such that the bottom contacts of  $T_{2}$, read from left to right, are
precisely $$t(W_{i_{1}}), t(W_{i_{2}}), \ldots, t(W_{i_{n}})$$
and the top contacts of $T_{1}$, read from left to right, are precisely
$$ b(W_{i_{1}}), b(W_{i_{2}}), \ldots, b(W_{i_{n}}).$$
\end{enumerate}
Define a new braided diagram as follows. Remove the transistors $T_{1}$ and $T_{2}$ and all of the wires $W_{i_{1}}, \ldots, W_{i_{n}}$ connecting the top
of $T_{1}$ to the bottom of $T_{2}$. Let $W_{j_{1}}, \ldots, W_{j_{m}}$ be the wires attached (in that order) 
to the top of $T_{2}$, and let $W_{k_{1}}, \ldots, W_{k_{m}}$ be the wires attached to the bottom of $T_{1}$. We glue the bottom of
$W_{j_{\ell}}$ to the top of $W_{k_{\ell}}$. There is a natural well-defined labelling function on the resulting wires, since 
$\ell(W_{j_{\ell}}) = \ell(W_{k_{\ell}})$ by our assumptions. We say that the new diagram $\Delta'$ is obtained from $\Delta$ by \emph{reducing the dipole 
$(T_{1}, T_{2})$}. The inverse operation is called \emph{inserting a dipole}.
\end{definition}

\begin{definition} \label{def: equivalentdiagrams}
(Equivalent Diagrams)
We say that two diagrams $\Delta_{1}$, $\Delta_{2}$ are \emph{equivalent} if there is a homeomorphism 
$\phi: \Delta_{1} \rightarrow \Delta_{2}$ that preserves the labels on the wires, restricts to a homeomorphism
$\phi_{|}: F(\Delta_{1}) \rightarrow F(\Delta_{2})$, preserves the tops and bottoms of the transistors and frame, and preserves
the left to right orientations on the transistors and the frame. We write $\Delta_{1} \equiv \Delta_{2}$.
\end{definition}

\begin{definition} \label{def: equivalentmodulodipoles}
(Equivalent Modulo Dipoles; Reduced Diagram)
We say that $\Delta$ and $\Delta'$ are \emph{equivalent modulo dipoles}
if there is a sequence $\Delta \equiv \Delta_{1},  \Delta_{2},  \ldots , \Delta_{n} \equiv \Delta'$,
where $\Delta_{i+1}$ is obtained from $\Delta_{i}$ by either inserting or removing a dipole, for $i \in \{ 1, \ldots, n-1 \}$.

A braided diagram $\Delta$ over a semigroup presentation is called \emph{reduced} if it contains no dipoles.
\end{definition}

\begin{example}
In Figure \ref{figure2}, we have two braided diagrams over the semigroup presentation $\mathcal{P} = \langle a, b, c \mid ab=ba, ac=ca, bc=cb \rangle$.
The two rightmost transistors in the diagram on the left form a dipole, and the diagram on the right is the result of reducing that dipole. 
\begin{center}
\begin{figure}[!h]
\includegraphics{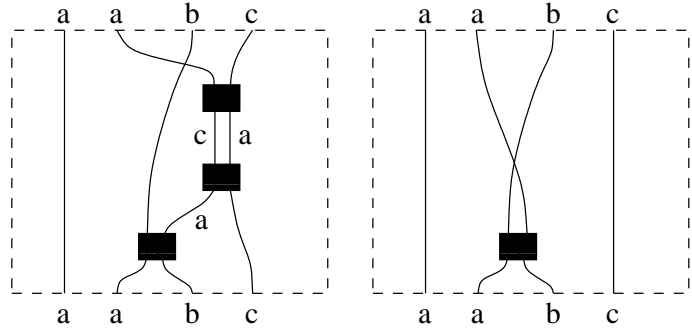}
\caption{The diagram on the right is obtained from the one on the left by reduction of a dipole.}
\label{figure2}
\end{figure}
\end{center}
\end{example}

\begin{proposition} \cite{Far3}
Equivalence modulo dipoles is an equivalence relation on the set of all braided diagrams over $\mathcal{P}$. Each equivalence class contains 
a unique reduced diagram. 
\end{proposition}

\begin{theorem} \label{thm: braideddiagramgroups} \cite{Far3}
Let $\mathcal{P} = \langle \Sigma \mid \mathcal{R} \rangle$ be a semigroup presentation, and let $w \in \Sigma^{+}$. The set of all braided $(w,w)$-diagrams
over $\mathcal{P}$, modulo dipoles, forms a group $D_{b}(\mathcal{P}, w)$ under the operation of concatenation.
\end{theorem}

\begin{definition} \label{def: diagramgroups}
We call $D_{b}(\mathcal{P}, w)$ the \emph{braided diagram group over $\mathcal{P}$ based at $w$}. 
\end{definition}

\subsection*{The isomorphism theorem}

\begin{definition} \label{def: treelike}
A semigroup presentation $\mathcal{P} = \langle \Sigma \mid \mathcal{R} \rangle$ is \emph{tree-like} if,
\begin{enumerate}
\item every relation $(w_{1}, w_{2}) \in \mathcal{R}$ satisfies $|w_{1}| = 1$ and $|w_{2}| > 1$;
\item if $(a, w_{1}), (a, w_{2}) \in \mathcal{R}$, then $w_{1} \equiv w_{2}$.
\end{enumerate}
\end{definition}

By a \emph{linearly ordered ultrametric space} we mean an ultrametric metric space $X$ with a linear order such that whenever $B_1$ and $B_2$ are disjoint balls in $X$ with
some point of $B_1$ less than some point of $B_2$, then every point of $B_1$ is less than every point of $B_2$. Thus, there is an induced linear order on any collection of disjoint balls in $X$. 

\begin{definition} \label{def: smallsim}
Let $X$ be a linearly ordered compact ultrametric space. Let $\sm_{X}$ be a finite similarity structure
on $X$ such that for every pair of balls $B_1, B_2$ in $X$, the following two conditions hold:
\begin{enumerate}
\item $|\sm_{X}(B_{1},B_{2})| \leq 1$, and
\item each $h \in \sm_{X}(B_{1}, B_{2})$ is order-preserving. 
\end{enumerate}
We say that $\sm_{X}$ is a \emph{small}
similarity structure.
\end{definition}

\begin{definition} \label{def: psimx}
Let $X$ be a linearly ordered compact ultrametric space with a small similarity structure $\sm_{X}$.
Define a semigroup presentation $\mathcal{P}_{\sm_{X}} = \langle \Sigma \mid \mathcal{R} \rangle$
as follows. Let 
$$ \Sigma = \{ [B] \mid B \text{ is a ball in } X \}.$$
(Recall that $[B]$ is the $\sm_{X}$-class of the ball $B \subseteq X$.)
If $B \subseteq X$ is a ball, let $B_{1}, \ldots, B_{n}$ be the maximal proper subballs
of $B$, listed in order. If $B$ is a point, then $n=0$. We set
$$ \mathcal{R} = \{ ( [B], [B_{1}][B_{2}]\ldots[B_{n}]) \mid n \geq 1, B \text{ is a ball in } X \}.$$
\end{definition}

\begin{remark} \label{rem: psimx}
We note that $\mathcal{P}_{\sm_{X}}$ will always be a tree-like semigroup presentation, for any choice of
linearly ordered compact ultrametric space $X$ and small similarity structure $\sm_{X}$.
\end{remark}

\begin{theorem} \label{thm: bigisomorphism}
If $X$ is a linearly ordered compact ultrametric space with a small similarity structure $\sm_{X}$,
then 
$$ \Gamma(\sm_{X}) \cong D_{b}(\mathcal{P}_{\sm_{X}}, [X]).$$
Conversely, if $\mathcal{P} = \langle \Sigma \mid \mathcal{R} \rangle$ is a tree-like semigroup presentation,
and $x \in \Sigma$, then there is a linearly ordered compact ultrametric space $X_{\mathcal{P}}$ and a small
finite similarity structure $\sm_{X_{\mathcal{P}}}$ such that
$$D_{b}(\mathcal{P}, x) \cong \Gamma(\sm_{X_{\mathcal{P}}}).$$  
\end{theorem}

\begin{proof}
If $\gamma \in \Gamma(\sm_{X})$, then there are partitions $\mathcal{P}_{1}$, $\mathcal{P}_{2}$ of $X$ into balls,
and a bijection $\phi: \mathcal{P}_{1} \rightarrow \mathcal{P}_{2}$ such that, for any $B \in \mathcal{P}_{1}$,
$\gamma(B) = \phi(B)$ and $\gamma|_{B} \in \sm_{X}(B, \gamma(B))$.
Since $|\sm_{X}(B,\gamma(B))| \leq 1$, the triple $(\mathcal{P}_{1}, \mathcal{P}_{2}, \phi)$ determines $\gamma$ without
ambiguity. We call $(\mathcal{P}_{1}, \mathcal{P}_{2}, \phi)$ a \emph{defining triple} for $\gamma$. Note that a given $\gamma$
will usually have many defining triples. Let $\mathcal{D}$ be the set of all defining triples, for $\gamma$ running over all of
$\Gamma(\sm_{X})$.

We will now define a map $\psi: \mathcal{D} \rightarrow D_{b}(\mathcal{P}_{\sm_{X}}, [X])$. To a partition $\mathcal{P}$
of $X$ into balls, we first assign a braided diagram $\Delta_{\mathcal{P}}$ over $\mathcal{P}_{\sm_{X}}$. There is a transistor 
$T_{B} \in \mathcal{T}(\Delta_{\mathcal{P}})$ for each ball $B$ which properly contains some ball of $\mathcal{P}$. There is a wire
$W_{B} \in \mathcal{W}(\Delta_{\mathcal{P}})$ for each ball $B$ which contains a ball of $\mathcal{P}$. The wires are attached as follows:
\begin{enumerate}
\item If $B=X$, then we attach the top of $W_{B}$ to the top of the frame.
 If $B \neq X$, then the top of the wire $W_{B}$ is attached to the bottom of the transistor $T_{\widehat{B}}$, where $\widehat{B}$ is the (unique) ball
that contains $B$ as a maximal proper subball.

Moreover, we attach the wires in an ``order-respecting" fashion. Thus, if $\widehat{B}$ is a ball properly containing balls of $\mathcal{P}$, we let 
$B_{1}, B_{2}, \ldots, B_{n}$ be the collection of maximal proper subballs of $\widehat{B}$, listed in order. We attach the wires $W_{B_{1}}, W_{B_{2}}, \ldots,
W_{B_{n}}$ so that $t(W_{B_{i}})$ is to the left of $t(W_{B_{j}})$ on the bottom of $T_{\widehat{B}}$ if $i < j$.

\item The bottom of the wire $W_{B}$ is attached to the top of $T_{B}$ if $B$ properly contains a ball of $\mathcal{P}$. If not (i.e., if $B \in \mathcal{P}$),
then we attach the bottom of $W_{B}$ to the bottom of the frame. We can arrange, moreover, that the wires are attached in an order-respecting manner to the
bottom of the frame. (Thus, if $B_{1} < B_{2}$ ($B_{1}, B_{2} \in \mathcal{P}$), we have that $b(W_{B_{1}})$ is to the left of $b(W_{B_{2}})$.)
\end{enumerate}
The labelling function $\ell: \mathcal{W}(\Delta_{\mathcal{P}}) \rightarrow \Sigma$ sends $W_{B}$ to $[B]$. It is straightforward to check that the
resulting $\Delta_{\mathcal{P}}$ is a braided diagram over $\mathcal{P}_{\sm_{X}}$. The top label of $\Delta_{\mathcal{P}}$ is $[X]$.

Given a bijection $\phi: \mathcal{P}_{1} \rightarrow \mathcal{P}_{2}$, where $\mathcal{P}_{1}$ and $\mathcal{P}_{2}$ are partitions of $X$ into balls
and $[B] = [\phi(B)]$, we can define a braided diagram $\Delta_{\phi}$ over $\mathcal{P}_{\sm_{X}}$ as follows. We let 
$\mathcal{T}(\Delta_{\phi}) = \emptyset$, and $\mathcal{W}(\Delta_{\phi}) = \{ W_{B} \mid B \in \mathcal{P}_{1} \}$.  
We attach the top of each wire to the frame in such a way that $t(W_{B_{1}})$ is to the left of $t(W_{B_{2}})$ if $B_{1} < B_{2}$. We attach the
bottom of each wire to the bottom of the frame in such a way that $b(W_{B_{1}})$ is to the left of $b(W_{B_{2}})$ if $\phi(B_{1}) < \phi(B_{2})$.

Now, for a defining triple $(\mathcal{P}_{1}, \mathcal{P}_{2}, \phi) \in \mathcal{D}$, we set
$\psi( (\mathcal{P}_{1}, \mathcal{P}_{2}, \phi)) = \Delta_{\mathcal{P}_{2}} \circ \Delta_{\phi^{-1}} \circ \Delta_{\mathcal{P}_{1}}^{-1}
\in \mathcal{D}_{b}(\mathcal{P}_{\sm_{X}}, [X])$.

We claim that any two defining triples $(\mathcal{P}_{1}, \mathcal{P}_{2}, \phi)$, $(\mathcal{P}_{1}', \mathcal{P}_{2}', \phi')$ 
for a given $\gamma \in \Gamma(\sm_{X})$
have the same image
in $D_{b}(\mathcal{P}_{\sm_{X}}, [X])$, modulo dipoles. We begin by proving an intermediate statement. Let $(\mathcal{P}_{1}, \mathcal{P}_{2}, \phi)$
be a defining triple. Let $B \in \mathcal{P}_{1}$, and let $\widehat{B}_{1}, \ldots, \widehat{B}_{n}$ be the collection of maximal proper subballs of $B$, listed
in order. We let $B' = \phi(B)$ and let $\widehat{B}'_{1}, \ldots, \widehat{B}'_{n}$ be the collection of maximal proper subballs of $B'$. (Note that
$[B'] = [B]$ by our assumptions, so both have the same number of maximal proper subballs, and in fact $[\widehat{B}_{i}] = [\widehat{B}'_{i}]$ for $i = 1, \ldots, n$,
since $\gamma|_{B} \in \sm_{X}(B, B')$ and the elements of $\sm_{X}(B, B')$ preserve order.) We set 
$\widehat{\mathcal{P}}_{1} = (\mathcal{P}_{1} - \{ B \}) \cup \{ \widehat{B}_{1}, \ldots, \widehat{B}_{n} \}$,
$\widehat{\mathcal{P}}_{2} = (\mathcal{P}_{2} - \{ B' \}) \cup \{ \widehat{B}'_{1}, \ldots, \widehat{B}'_{n} \}$,
and $\widehat{\phi}|_{\mathcal{P}_{1} - \{ B \}} = \phi|_{\mathcal{P}_{1} - \{ B \}}$, 
$\widehat{\phi} (\widehat{B}_{i}) = \widehat{B}'_{i}$. We say that $(\widehat{\mathcal{P}}_{1}, \widehat{\mathcal{P}}_{2}, \widehat{\phi})$
is obtained from  $(\mathcal{P}_{1}, \mathcal{P}_{2}, \phi)$ by \emph{subdivision} at $(B, B')$.
We claim that $\psi( (\widehat{\mathcal{P}}_{1}, \widehat{\mathcal{P}}_{2}, \widehat{\phi}))$ is obtained from
$\psi( (\mathcal{P}_{1}, \mathcal{P}_{2}, \phi))$ by inserting a dipole.

We can build $\Delta_{\widehat{\mathcal{P}}_{1}}$ from $\Delta_{\mathcal{P}_{1}}$ as follows.
The bottom of the wire $W_{B}$ is attached to the bottom of the frame in $\Delta_{\mathcal{P}_{1}}$. 
We introduce a new transistor $T_{B}$ and $n$ new wires $W_{\widehat{B}_{1}}, \ldots, W_{\widehat{B}_{n}}$. We attach the tops of
the wires $W_{\widehat{B}_{1}}, \ldots, W_{\widehat{B}_{n}}$ to the bottom of $T_{B}$ (in order). Let $(a,b)$ be an open neighborhood
of $b(W_{B}) \in F(\Delta_{\mathcal{P}_{1}})$ containing no other contacts of wires from $\Delta_{\mathcal{P}_{1}}$. We choose a new
$b(W_{B})$ on the top of $T_{B}$, and attach the bottoms of the wires $W_{\widehat{B}_{1}}, \ldots, W_{\widehat{B}_{n}}$ inside the interval 
$(a,b)$, while preserving the left to right order. The resulting diagram is $\Delta_{\widehat{P}_{1}}$. We can similarly build $\Delta_{\widehat{\mathcal{P}}_{2}}$
from $\Delta_{\mathcal{P}_{2}}$.

We construct $\Delta_{\widehat{\phi}}$ from $\Delta_{\phi}$ as follows. We consider the wire $W_{B}$; let $(a,b)$ be an open interval in the top of 
$F(\Delta_{\phi})$ containing $t(W_{B})$, but no other contacts, and let $(c,d)$ be an open interval in the bottom of $F(\Delta_{\phi})$ containing
$b(W_{B})$, but no other contacts. We remove the wire $W_{B}$ and introduce the wires $W_{\widehat{B}_{1}}, \ldots, W_{\widehat{B}_{n}}$.
We attach the tops of the new wires to $(a,b)$ in such a way that $t(W_{\widehat{B}_{i}})$ is to the left of $t(W_{\widehat{B}_{j}})$ if $i<j$.
Attach the bottoms of the wires to $(c,d)$ so that $b(W_{\widehat{B}_{i}})$ is to the left of $b(W_{\widehat{B}_{j}})$. The
resulting diagram is $\Delta_{\widehat{\phi}}$.

We can now compare $\psi((\mathcal{P}_{1}, \mathcal{P}_{2}, \phi))$ with $\psi((\widehat{\mathcal{P}}_{1}, \widehat{\mathcal{P}}_{2}, \widehat{\phi}))$.
Assume $b(W_{B})$ is the $k$th contact from the left on the bottom of $F(\Delta_{\mathcal{P}_{1}})$, and 
$b(W_{B'})$ is the $\ell$th contact from the left on the bottom of $F(\Delta_{\mathcal{P}_{2}})$. There is a wire
$W \in \mathcal{W}(\Delta_{\phi^{-1}})$ running from the $\ell$th top contact of $\Delta_{\phi^{-1}}$ to the bottom $k$th contact, since $\phi(B) = B'$. The
wire $W_{\widehat{B}'_{i}} \in \mathcal{W}(\Delta_{\widehat{P}_{2}} \circ \Delta_{\widehat{\phi}^{-1}} \circ \Delta_{\widehat{\mathcal{P}}_{1}^{-1}})$ will
run from the bottom of $T_{B'}$ to the $(\ell + i - 1)$st contact at the bottom of $\Delta_{\widehat{P}_{2}}$, through the $(k+i-1)$st contact at the bottom
of $\Delta_{\widehat{\phi}^{-1}}$, and eventually terminate at the $i$th top contact of $T_{B}$. Thus, the wires $W_{\widehat{B}'_{i}}$ leading from the bottom
of $T_{B'}$ to the top of $T_{B}$ are attached in the same order at both ends. It follows that $(T_{B}, T_{B'})$ is a dipole in $\Delta_{\widehat{\mathcal{P}}_{2}}
\circ \Delta_{\widehat{\phi}^{-1}} \circ \Delta_{\widehat{\mathcal{P}}_{1}}^{-1}$. Removing the dipole results in 
$\Delta_{\mathcal{P}_{2}} \circ \Delta_{\widehat{\phi}^{-1}} \circ \Delta_{\mathcal{P}_{1}}^{-1}$.
This proves that $\psi((\widehat{\mathcal{P}}_{1}, \widehat{\mathcal{P}}_{2}, \widehat{\phi}))$ is obtained from 
$\psi((\mathcal{P}_{1}, \mathcal{P}_{2}, \phi))$ by inserting a dipole, as claimed.

Now suppose that $(\mathcal{P}_{1}, \mathcal{P}_{2}, \phi)$ and $(\mathcal{P}'_{1}, \mathcal{P}'_{2}, \phi')$ are defining triples
for the same element $\gamma \in \Gamma(\sm_{X})$. We can find a common refinement $\mathcal{P}''_{1}$ of $\mathcal{P}_{1}$ and $\mathcal{P}'_{1}$.
After repeating subdivision we can pass from $(\mathcal{P}_{1}, \mathcal{P}_{2}, \phi)$ to $(\mathcal{P}''_{1}, \widehat{P}_{2}, \widehat{\phi})$
(for some partition $\widehat{\mathcal{P}}_{2}$ of $X$ into balls and some bijection $\widehat{\phi}: \mathcal{P}''_{1} \rightarrow \widehat{\mathcal{P}}_{2}$).
Since subdivision doesn't change the values of $\psi$ modulo dipoles, $\psi((\mathcal{P}_{1}, \mathcal{P}_{2}, \phi)) = 
\psi((\mathcal{P}''_{1}, \widehat{\mathcal{P}}_{2}, \widehat{\phi}))$ modulo dipoles. Similarly, we can subdivide
$(\mathcal{P}'_{1}, \mathcal{P}'_{2}, \phi')$ repeatedly in order to obtain $(\mathcal{P}''_{1}, \widehat{\mathcal{P}}'_{2}, \widehat{\phi}')$, where
$\psi((\mathcal{P}'_{1}, \mathcal{P}'_{2}, \phi')) = \psi((\mathcal{P}''_{1}, \widehat{\mathcal{P}}'_{2}, \widehat{\phi}'))$ modulo dipoles.
Both $(\mathcal{P}''_{1}, \widehat{\mathcal{P}}'_{2}, \widehat{\phi}')$
and $(\mathcal{P}'_{1}, \widehat{\mathcal{P}}_{2}, \widehat{\phi})$ are defining triples for $\gamma$, so we are forced to have
$\widehat{\phi} = \widehat{\phi}'$ and $\widehat{\mathcal{P}}_{2} = \widehat{\mathcal{P}}'_{2}$.
It follows that $\psi((\mathcal{P}_{1}, \mathcal{P}_{2}, \phi)) = \psi((\mathcal{P}'_{1}, \mathcal{P}'_{2}, \phi'))$, so $\psi$ induces a function
from $\Gamma(\sm_{X})$ to $D_{b}(\mathcal{P}_{\sm_{X}}, [X])$. We will call this function $\widehat{\psi}$.

Now we will show that $\widehat{\psi}: \Gamma(\sm_{X}) \rightarrow D_{b}(\mathcal{P}_{\sm_{X}}, [X])$ is a homomorphism. Let $\gamma, \gamma' \in
\Gamma(\sm_{X})$. After subdividing as necessary, we 
can choose defining triples $(\mathcal{P}_{1}, \mathcal{P}_{2}, \phi)$ and $(\mathcal{P}'_{1}, \mathcal{P}'_{2}, \phi')$ for 
$\gamma$ and $\gamma'$ (respectively) in such a way that $\mathcal{P}_{2} = \mathcal{P}'_{1}$. It follows easily that 
$(\mathcal{P}_{1}, \mathcal{P}'_{2}, \phi' \phi)$ is a defining triple for $\gamma' \gamma$. Therefore, 
$\widehat{\psi}(\gamma' \gamma) = \Delta_{\mathcal{P}'_{2}} \circ \Delta_{(\phi' \phi)^{-1}} \circ \Delta^{-1}_{\mathcal{P}_{1}}$.
Now 
\begin{align*}
\widehat{\psi}(\gamma') \circ \widehat{\psi}(\gamma) &= \Delta_{\mathcal{P}'_{2}} \circ \Delta_{(\phi')^{-1}} \circ \Delta^{-1}_{\mathcal{P}'_{1}} \circ
\Delta_{\mathcal{P}_{2}} \circ \Delta_{\phi^{-1}} \circ \Delta^{-1}_{\mathcal{P}_{1}} \\
&= \Delta_{\mathcal{P}'_{2}} \circ \Delta_{(\phi')^{-1}} \circ \Delta_{\phi^{-1}} \circ \Delta^{-1}_{\mathcal{P}_{1}} \\
&= \Delta_{\mathcal{P}'_{2}} \circ \Delta_{(\phi'\phi)^{-1}} \circ \Delta^{-1}_{\mathcal{P}_{1}}
\end{align*}
Therefore, $\widehat{\psi}$ is a homomorphism. 
   
We now show that $\widehat{\psi}: \Gamma(\sm_{X}) \rightarrow D_{b}(\mathcal{P}_{\sm_{X}}, [X])$ is injective. Suppose that $\widehat{\psi}(\gamma) = 1$.
We choose a defining triple $(\mathcal{P}_{1}, \mathcal{P}_{2}, \phi)$ for $\gamma$ with the property that, if $B \subseteq X$ is a ball, $\gamma(B)$ is a ball,
and $\gamma|_{B} \in \sm_{X}(B, \gamma(B))$, then $B$ is contained in some ball of $\mathcal{P}_{1}$. We claim that 
$\psi((\mathcal{P}_{1}, \mathcal{P}_{2}, \phi))$ is a reduced diagram. If there were a dipole $(T_{1}, T_{2})$, then we would have
$T_{1} \in \mathcal{T}(\Delta_{\mathcal{P}_{1}}^{-1})$ and $T_{2} \in \mathcal{T}(\Delta_{\mathcal{P}_{2}})$, since it is impossible for $\Delta_{\mathcal{P}}$
to contain any dipoles, for any partition $\mathcal{P}$ of $X$ into balls. Thus $T_{1} = T_{B_{1}}$ and $T_{2} = T_{B_{2}}$, where $[B_{1}] = [B_{2}]$ and
the wires from the bottom of $T_{B_{2}}$ attach to the top of $T_{B_{1}}$, in order.
This means that, if $\widehat{B}_{1}, \ldots, \widehat{B}_{n}$ are the maximal proper subballs of $B_{1}$, 
and $\widehat{B}'_{1}, \ldots, \widehat{B}'_{n}$ are the maximal proper subballs of $B_{2}$, then $\gamma(\widehat{B}_{i}) = \widehat{B}'_{i}$, where the latter
is a ball, and $\gamma|_{\widehat{B}_{i}} \in \sm_{X}(\widehat{B}_{i}, \widehat{B}'_{i})$.

Now, since $[B_{1}] = [B_{2}]$, there is $h \in \sm_{X}(B_{1}, B_{2})$. Since $\sm_{X}$ is closed under restrictions and $h$ preserves order, 
we have $h_{i} \in \sm_{X}( \widehat{B}_{i}, \widehat{B}'_{i})$ for $i = 1, \ldots, n$, where $h_{i} = h|_{\widehat{B}_{i}}$. It follows that
$\gamma|_{\widehat{B}_{i}} = h_{i}$, so, in particular, $\gamma|_{B_{1}} = h$.
Since $B_{1}$ properly contains some ball in $\mathcal{P}_{1}$, this is a contradiction. Thus, $\psi((\mathcal{P}_{1}, \mathcal{P}_{2}, \phi))$
is reduced.

We claim that $\psi((\mathcal{P}_{1}, \mathcal{P}_{2}, \phi))$ contains no transistors (due to the condition $\widehat{\psi}(\gamma) = 1$).
We've shown that $\psi((\mathcal{P}_{1}, \mathcal{P}_{2}, \phi))$ is a reduced diagram in the same class as the identity
$1 \in D_{b}(\mathcal{P}_{\sm_{X}}, [X])$. The identity can be represented as the (unique) $([X], [X])$-diagram $\Delta_{1}$ with only a single wire,
$W_{X}$, and no transistors. We must have $\psi((\mathcal{P}_{1}, \mathcal{P}_{2}, \phi)) \equiv \Delta_{1}$. Thus,
there is no ball that properly contains a ball of $\mathcal{P}_{1}$. It can only be that $\mathcal{P}_{1} = \{ X \}$, so we must have
$\gamma \in \sm_{X}(X,X)$. This forces $\gamma = 1$, so $\widehat{\psi}$ is injective.

Finally we must show that $\widehat{\psi}: \Gamma(\sm_{X}) \rightarrow D_{b}(\mathcal{P}_{\sm_{X}}, [X])$ is surjective. Let
$\Delta$ be a reduced $([X],[X])$-diagram over $\mathcal{P}_{\sm_{X}}$. A transistor $T \in \mathcal{T}(\Delta)$ is called \emph{positive}
if its top label is the left side of a relation in $\mathcal{P}_{\sm_{X}}$, otherwise (i.e., if the top label is the right side of a relation
in $\mathcal{P}_{\sm_{X}}$) the transistor $T$ is \emph{negative}. It is easy to see that the sets of positive and negative transistors partition
$\mathcal{T}(\Delta)$. We claim that, if $\Delta$ is reduced, then we cannot have $T_{1} \dot{\preceq} T_{2}$ when $T_{1}$ is positive and $T_{2}$
is negative. If we had such $T_{1} \dot{\preceq} T_{2}$, then we could find $T'_{1} \preceq T'_{2}$, where $T'_{1}$ is positive and $T'_{2}$ is negative.
Since $T'_{1}$ is positive, there is only one wire $W$ attached to the top of $T'_{1}$. This wire must be attached to the bottom of $T'_{2}$, since
$T'_{1} \preceq T'_{2}$, and it must be the only wire attached to the bottom of $T'_{2}$, since $T'_{2}$ is negative and $\mathcal{P}_{\sm_{X}}$
is a tree-like semigroup presentation by \ref{rem: psimx}. Suppose that $\ell(w) = [B]$. By the definition of $\mathcal{P}_{\sm_{X}}$, $[B]$ is the left
side of exactly one relation, namely $([B], [B_{1}][B_{2}]\ldots[B_{n}])$, where the $B_{i}$ are maximal proper subballs of $B$, listed in order. It follows
that the bottom label of $T'_{1}$ is $[B_{1}][B_{2}]\ldots[B_{n}]$ and the top label of $T'_{2}$ is $[B_{1}][B_{2}]\ldots[B_{n}]$. Therefore
$(T'_{1}, T'_{2})$ is a dipole. This proves the claim.

A diagram over $\mathcal{P}_{\sm_{X}}$ is \emph{positive} if all of its transistors are positive, and \emph{negative} if all of its transistors are negative.
We note that $\Delta$ is positive if and only if $\Delta^{-1}$ is negative, by the description of inverses in the proof of Theorem \ref{thm: braideddiagramgroups}.
The above reasoning shows that any reduced $([X], [X])$-diagram over $\mathcal{P}_{\sm_{X}}$ can be written
$\Delta = \Delta^{+}_{1} \circ \left( \Delta_{2}^{+} \right)^{-1}$, where $\Delta^{+}_{i}$ is a positive diagram for $i = 1,2$.

We claim that any positive diagram $\Delta$ over $\mathcal{P}_{\sm_{X}}$ with top label $[X]$ is $\Delta_{\mathcal{P}}$ (up to a reordering of the bottom 
contacts), where $\mathcal{P}$ is some partition of $X$. There is a unique wire $W \in \mathcal{W}(\Delta)$ making a top contact with the frame. We call
this wire $W_{X}$. Note that its label is $[X]$ by our assumptions. The bottom contact of $W_{X}$ lies either on the bottom of the frame, or on top of some
transistor. In the first case, we have $\Delta = \Delta_{\mathcal{P}}$ for $\mathcal{P} = \{ X \}$ and we are done. In the second, the bottom contact of 
$W_{X}$ lies on top of some transistor $T$, which we call $T_{X}$. Since the top label of $T_{X}$ is $[X]$, the bottom label must be $[B_{1}]\ldots [B_{k}]$,
where $B_{1}, \ldots, B_{k}$ are the maximal proper subballs of $X$. Thus there are wires $W_{1}, \ldots, W_{k}$ attached to the bottom of $T_{X}$, and we 
have $\ell(W_{i}) = [B_{i}]$, for $i = 1, \ldots, k$. We relabel each of the wires $W_{B_{1}}, \ldots, W_{B_{k}}$, respectively. Note that
$\{ B_{1}, \ldots, B_{k} \}$ is a partition of $X$ into balls. We can continue in this way, inductively labelling each wire with a ball $B \subseteq X$.
If we let $\overline{B}_{1}, \ldots, \overline{B}_{m}$ be the resulting labels of the wires which make bottom contacts with the frame, then 
$\{ \overline{B}_{1}, \ldots, \overline{B}_{m} \} = \mathcal{P}$ is a partition of $X$ into balls, and $\Delta = \Delta_{\mathcal{P}}$ by construction,
up to a reordering of the bottom contacts.

We can now prove surjectivity of $\widehat{\psi}$. Let $\Delta \in D_{b}(\mathcal{P}_{\sm_{X}}, [X])$ be reduced. We can write
$\Delta = \Delta_{2}^{+} \circ \left( \Delta_{1}^{+} \right)^{-1}$, where $\Delta_{i}^{+}$ is positive, for $i=1,2$. It follows that
$\Delta_{i}^{+} = \Delta_{\mathcal{P}_{i}} \circ \sigma_{i}$, for $i=1,2$, where $\mathcal{P}_{i}$ is a partition of $X$ into balls and $\sigma_{i}$ is 
diagram containing no transistors. Thus, $\Delta = \Delta_{\mathcal{P}_{2}} \circ \sigma_{2} \circ \sigma_{1}^{-1} \circ \Delta_{\mathcal{P}_{1}}^{-1}
= \psi(( \mathcal{P}_{1}, \mathcal{P}_{2}, \phi))$, where $\phi: \mathcal{P}_{1} \rightarrow \mathcal{P}_{2}$ is a bijection determined by
$\sigma_{2} \circ \sigma_{1}^{-1}$. Therefore, $\widehat{\psi}$ is surjective.

For the converse, we must show that if $\mathcal{P} = \langle \Sigma \mid \mathcal{R} \rangle$ is a tree-like semigroup presentation, $x \in \Sigma$, then there
is a linearly ordered compact ultrametric space $X_{\mathcal{P}}$ and a small similarity structure $\sm_{X_{\mathcal{P}}}$ such that
$D_{b}(\mathcal{P}, x) \cong \Gamma(\sm_{X_{\mathcal{P}}})$. 
Construct a labelled ordered simplicial tree $T_{(\mathcal{P}, x)}$ as follows. Begin with a vertex $\ast$, the root, labelled by $x \in \Sigma$. 
By the definition of tree-like semigroup presentation (Definition \ref{def: treelike}), there is at most  one relation in $\mathcal{R}$ having
the word $x$ as its left side. Let us suppose first that $(x,x_{1}x_{2}\ldots x_{k}) \in \mathcal{R}$, where $k \geq 2$. 
We introduce $k$ children of the root, labelled $x_{1}, \ldots, x_{k}$ (respectively), each connected to the root by an edge. 
The children are ordered from left to right in such a way that we read the
word $x_{1}x_{2}\ldots x_{k}$ as we read the labels of the children from left to right. If, on the other hand, $x$ is not the left side of any relation
in $\mathcal{R}$, then the tree terminates -- there is only the root. We continue similarly: if $x_{i}$ is the left side of some relation 
$(x_{i},  y_{1}y_{2}\ldots y_{m}) \in \mathcal{R}$ ($m \geq 2$),
then this relation is unique and we introduce a labelled ordered collection of children, as above. If $x_{i}$ is not the left side of any relation
in $\mathcal{R}$, then $x_{i}$ has no children. This builds a labelled ordered tree $T_{(\mathcal{P}, x)}$. We note that if a vertex $v 
\in T_{(\mathcal{P}, x)}$ is labelled by $y \in \Sigma$, then the subcomplex $T_{v} \leq T_{(\mathcal{P},x)}$ spanned by $v$ and all of its descendants is isomorphic
to $T_{(\mathcal{P}, y)}$, by a simplicial isomorphism which preserves the labelling and the order. 

We let $\mathrm{Ends}(T_{(\mathcal{P}, x)})$ denote the set of all edge-paths $p$ in $T_{(\mathcal{P}, x)}$ such that:
i) $p$ is without backtracking; ii) $p$ begins at the root; iii) $p$ is either infinite, or $p$ terminates at a vertex without children.
We define a metric on $\mathrm{Ends}(T_{(\mathcal{P}, x)})$ as follows. If $p, p' \in \mathrm{Ends}(T_{(\mathcal{P},  x)})$ and $p, p'$ have exactly
$m$ edges in common, then we set $d(p,p') = e^{-m}$. This metric makes $\mathrm{Ends}(T_{(\mathcal{P},x)})$ a compact ultrametric space, and it is linearly ordered
by the ordering of the tree. We can describe the
balls in $\mathrm{Ends}(T_{(\mathcal{P},x)})$ explicitly. Let $v$ be a vertex of $T_{(\mathcal{P},x)}$. We set $B_{v} = 
\{ p \in \mathrm{Ends}(T_{(\mathcal{P},x)}) \mid v \text{ lies  on } p \}$. Every such set is a ball, and every ball in $\mathrm{Ends}(T_{(\mathcal{P},x)})$
has this form. We can now describe a finite similarity structure $\sm_{X_{\mathcal{P}}}$ on $\mathrm{Ends}(T_{(\mathcal{P}, x)})$. Let $B_{v}$ and $B_{v'}$ be the balls
corresponding to the vertices $v, v' \in T_{(\mathcal{P},x)}$. If $v$ and $v'$ have different labels, then we set $\sm_{X_{\mathcal{P}}}(B_{v}, B_{v'}) = \emptyset$.
If $v$ and $v'$ have the same label, say $y \in \Sigma$, then there is label- and order-preserving simplicial isomorphism $\psi: T_{v} \rightarrow T_{v'}$.
Suppose that $p_{v}$ is the unique edge-path without backtracking connecting the root to $v$. Any point in $B_{v}$ can be expressed in the form
$p_{v}q$, where $q$ is an edge-path without backtracking in $T_{v}$. We let  $\widehat{\psi}: B_{v} \rightarrow B_{v'}$ be defined by the
rule $\widehat{\psi}(p_{v}q) = p_{v'}\psi(q)$. The map $\widehat{\psi}$ is easily seen to be a surjective similarity. 
We set $\sm_{X_{\mathcal{P}}}(B_{v}, B_{v'}) = \{ \widehat{\psi} \}$. The resulting assignments give a small similarity structure $\sm_{X_{\mathcal{P}}}$ on 
the linearly ordered compact
ultrametric space $\mathrm{Ends}(T_{(\mathcal{P},x)})$.

Now we can apply the first part of the theorem: setting $X_{\mathcal{P}} = \mathrm{Ends}(T_{(\mathcal{P},x)})$, we
have $\Gamma(\sm_{X_{\mathcal{P}}}) \cong D_{b} (\mathcal{P}_{\sm_{X_{\mathcal{P}}}}, [X_{\mathcal{P}}]) \cong D_{b}(\mathcal{P}, x)$.
\end{proof}

\begin{example}
The generalized Thompson's groups $V_{d}$ are isomorphic to the braided diagram groups $D_{b}(\mathcal{P}, x)$, 
where $\mathcal{P} = \langle x \mid (x,x^{d}) \rangle$.
This fact was already proved in \cite{GubaSapir} and \cite{Far3}, and it is also a consequence of Theorem \ref{thm: bigisomorphism}. 
\end{example}  

\subsection*{FSS groups of small $\mathrm{Sim}$-structures}

In this subsection, we will show how to weaken the hypothesis of Theorem \ref{thm: bigisomorphism} somewhat.

\begin{lemma} \label{lemma: createorder}
If $X$ is a compact ultrametric space and the $\mathrm{Sim}$-structure satisfies $|\sm_{X}(B_{1}, B_{2})| \leq 1$ for every 
pair of balls $B_{1}, B_{2} \subseteq X$, then
there is a linear order $\leq$ on $X$ such that, for each $\gamma \in \mathrm{Sim}_{X}(B_{1},B_{2})$,
$\gamma : B_{1} \rightarrow B_{2}$ is order-preserving (for arbitrary $B_{1}$, $B_{2}$ such that
$\mathrm{Sim}_{X}(B_{1},B_{2}) \neq \emptyset$).
\end{lemma}

\begin{proof}
Choose a collection $\mathcal{B}'$ of balls, one from each $\mathrm{Sim}_{X}$-class of balls in $X$.  We 
let $\mathcal{B} \subseteq \mathcal{B}'$ denote the subcollection of balls that are not singleton sets.
Suppose that $B \in \mathcal{B}$.  Suppose that $\{ B_{1}, \ldots, B_{m} \}$ is the collection of all maximal
proper subballs of $B$.  We impose an (arbitrary) strict linear order $\prec$ on $\{ B_{1}, \ldots, B_{m} \}$,
say
$$ B_{1} \prec B_{2} \prec \ldots \prec B_{m}.$$
We similarly choose a linear order on the maximal proper subballs for each $B \in \mathcal{B}$.   

If $B \subseteq X$ is an arbitrary ball, and $B$ is not a singleton, then there is a unique 
$\hat{B} \in \mathcal{B}$ such that $\mathrm{Sim}_{X}(B, \hat{B}) \neq \emptyset$, and thus
there is a unique $\gamma \in \mathrm{Sim}_{X}(B, \hat{B})$.  If $\{ B_{1}, \ldots, B_{m} \}$ is the
collection of maximal proper subballs of $B$, then we define
$$ B_{i} \prec B_{j} \Leftrightarrow \gamma(B_{i}) \prec \gamma(B_{j}).$$
(The sets $\gamma(B_{i})$ and $\gamma(B_{j})$ are maximal proper subballs in $\hat{B}$ since similarities
take maximal proper subballs to maximal proper subballs, and therefore $\gamma(B_{i})$ and $\gamma(B_{j})$
are comparable under the order defined on proper subballs of $\hat{B}$.)

Now for $x,y \in X$, we write $x \leq y$ if: i) $x=y$, or ii) if there is some ball $B \subseteq X$ such
that $x \in B_{i}$ and $y \in B_{j}$, where $B_{i}$ and $B_{j}$ are maximal proper subballs of $B$, and
$B_{i} \prec B_{j}$.

We claim first that $\leq$ is a linear order on $X$.  Indeed, it is clear that $x \leq x$ for each $x \in X$.
Suppose that $x \leq y$ and $y \leq x$, and suppose, for a contradiction, that $x \neq y$.  It follows
that there are balls $B'$ and $B''$ ($B' \neq B''$) and maximal proper subballs $B_{1}'$, $B_{2}'$ and
$B_{1}''$, $B_{2}''$ such that $B_{1}' \prec B_{2}'$, $B_{1}'' \prec B_{2}''$, $x \in B_{1}' \cap B_{2}''$,
and $y \in B_{2}' \cap B_{1}''$.  Since the balls $B'$ and $B''$ have points in common, they must be nested.
Suppose $B' \subsetneq B''$.  It follows that any two proper subballs of $B'$ must be contained in the same 
maximal proper subball of $B''$.  Since $B_{1}' \cap B_{2}'' \neq \emptyset$, we have
$B_{1}', B_{2}' \subseteq B_{2}''$, and, since $B_{2}' \cap B_{1}'' \neq \emptyset$, we have
$B_{1}', B_{2}' \subseteq B_{1}''$.  This is a contradiction, because $B_{1}'' \cap B_{2}'' = \emptyset$.
Therefore $\leq$ is antisymmetric.

Now suppose that $x \leq y$ and $y \leq z$.  We want to show that $x \leq z$.  We can assume that
$x \neq y$ and $y \neq z$.  There are balls $B'$ and $B''$ and maximal proper subballs 
$B_{1}', B_{2}' \subseteq B'$, $B_{1}'', B_{2}'' \subseteq B''$ such that $B_{1}' \prec B_{2}'$,
$B_{1}'' \prec B_{2}''$, $x \in B_{1}'$, $y \in B_{2}' \cap B_{1}''$, and $z \in B_{2}''$.
Since $y \in B' \cap B''$, it must be that $B'$ and $B''$ are nested.  We consider two cases:
i) $B' = B''$ and ii) $B' \subsetneq B''$.

If $B' = B''$, then we must have $B_{2}' = B_{1}''$ so that
$B_{1}' \prec B_{2}' \prec B_{2}''$, from which the conclusion $x \leq z$ easily follows.  
If $B' \subsetneq B''$, then $B_{1}'$ and $B_{2}'$ are both contained in the same maximal proper subball
of $B''$, and therefore $B_{1}', B_{2}' \subseteq B_{1}''$ (since $B_{2}' \cap B_{1}'' \neq \emptyset$).
Thus $x \in B_{1}''$, so $x \leq z$.

We prove that the order is linear. Let $x, y \in X$, $x \neq y$. There is a ball $B$ that is the smallest of all balls
containing both $x$ and $y$. Let $B'$ be the maximal proper subball of $B$ containing $x$, and let $B''$ be the maximal proper subball
of $B$ containing $y$. Our assumptions imply that $B' \cap B'' = \emptyset$ (since $x \notin B''$ and $y \notin B'$). Thus, either
$B' < B''$ or $B'' < B'$. In either case, $x$ and $y$ are comparable in the order $\leq$.

Finally, it is clear from the definition of $\leq$ that for every pair of balls $B_1, B_2$,  each $\gamma \in \sm_{X}(B_{1}, B_{2})$ 
preserves $\leq$.
\end{proof}

\begin{corollary} \label{cor:maincor}
If $\mathcal{P} = \langle \Sigma \mid \mathcal{R} \rangle$ is a tree-like semigroup presentation and
$w \in \Sigma$, then $D_{b}(\mathcal{P},w)$ is isomorphic to $\Gamma(\sm_{X})$, for some compact ultrametric
space $X$ and finite similarity structure $\sm_{X}$ satisfying $|\sm_{X}(B_{1}, B_{2})| \leq 1$ for all balls $B_{1}, B_{2} \subseteq X$. 
Conversely,
if $X$ is a compact ultrametric space and $\sm_{X}$ is a finite similarity structure satisfying $|\sm_{X}(B_{1}, B_{2})| \leq 1$ for all
balls $B_{1}, B_{2} \subseteq X$,
then there is a tree-like semigroup presentation $\mathcal{P} = \langle \Sigma \mid \mathcal{R} \rangle$ and $w \in \Sigma$
such that $\Gamma(\sm_{X}) \cong D_{b}(\mathcal{P},w)$.
\end{corollary}

\begin{proof}
The first statement is a direct consequence of Theorem \ref{thm: bigisomorphism}. Conversely, if $X$ is a compact ultrametric space
and $\sm_{X}$ has the above properties, we can apply Lemma \ref{lemma: createorder} to create a linear order on $X$ that is preserved
by each $\gamma \in \sm_{X}(B_{1}, B_{2})$, and then apply Theorem \ref{thm: bigisomorphism}.
\end{proof}



\bibliographystyle{plain}
\bibliography{biblio}

\end{document}